\documentclass[10pt,reqno]{amsart}
\usepackage{amsfonts}
\usepackage{amsmath}
\usepackage{CJK}
\usepackage{amssymb}
\usepackage{latexsym,bm}
\usepackage[RGB]{xcolor}
\usepackage{mathrsfs}
\usepackage{amsthm}
\usepackage{hyperref}
\usepackage[normalem]{ulem}
\allowdisplaybreaks[4]
\numberwithin{equation}{section}
\newtheorem{theorem}{Theorem}[section]

\newtheorem{lemma}[theorem]{Lemma}

\newtheorem{proposition}[theorem]{Proposition}

\begin{document}
	\title[]{Asymptotic Stability of Rarefaction Waves for the Hyperbolized Navier-Stokes-Fourier System}
	\author{Yuxi Hu, Mengran Yuan and Jie Zhang}
	\thanks{\noindent Yuxi Hu, Department of Mathematics, China University of Mining and Technology, Beijing, 100083 and Key Laboratory of Scientific and Engineering Computing (Ministry of Education), P.R. China, yxhu86@163.com\\
		\indent Mengran Yuan, Department of Mathematics, China University of Mining and Technology, Beijing, 100083, P.R. China,  ymengran2388@126.com\\
		\indent Jie Zhang, Department of Mathematics, China University of Mining and Technology, Beijing, 100083, P.R. China, 17861270280@163.com}
	\begin{abstract}
		This paper investigates the asymptotic stability of rarefaction waves for a one-dimensional compressible fluid system, where the Newton's law of viscosity and Fourier's law of heat conduction are replaced by Maxwell's law and Cattaneo's law, respectively. The system, which generalizes the classical Navier-Stokes-Fourier equations, features finite signal propagation speeds. We consider the Cauchy problem in Lagrangian coordinates with initial data connecting two different constant states via a rarefaction wave of the corresponding Euler system. Our main result proves that, provided the initial perturbation and wave strength are sufficiently small, the relaxation system admits a unique global solution. Furthermore, this solution converges uniformly to the background rarefaction wave as time approaches infinity.  The proof  is established through a combination of the relative entropy method and usual energy estimates.		
		{\bf AMS classification code}:  35Q35; 76A10; 76N30
	\end{abstract}
	\maketitle

	\section{Introduction}\label{sec1}
	The governing equations for one-dimensional compressible fluid motion take the following form:
	\begin{align}\label{1.1}
		\begin{cases}
			\rho_t+(\rho u)_x =0, \\
			(\rho u)_t+(\rho u^2 )_x+p_x = S_x, \\
			(\rho E)_t+(\rho u E+up+q-uS)_x = 0,
		\end{cases}
	\end{align}
	where $\rho$, $u$, $p=p(\rho, E)$, $S$, $q$, and $E$ denote the fluid density, velocity, pressure, viscous stress, heat flux, and specific total energy, respectively. Here, $E = e + \frac{1}{2}u^2$, and $e$ represents the internal energy of the fluid. 
	Equations \eqref{1.1} express the conservation of mass, balance of momentum, and conservation of energy.  The relation $p=p(\rho, E)$  is called the {\it equation of state}. It depends on the particular gas under consideration. 
	To close the system \eqref{1.1}, constitutive equations for $q$ and $S$ must be provided. Instead of adopting Fourier's law for heat conduction:
	\begin{align}\label{fourier}
		q = -\kappa \theta_x
	\end{align}
	and Newton's law for a Newtonian fluid:
	\begin{align}\label{Newton}
		S = \mu u_x,
	\end{align}
	where $\kappa$ and $\mu$ represent the thermal conductivity and viscosity coefficients, respectively, we employ the Cattaneo-Christov model for heat flux and the Maxwell-Oldroyd model for viscoelastic fluid:
	\begin{align}\label{1.2}
		\begin{cases}
			\tau_1 \rho \left( q_t + u q_x \right) + q + \kappa \theta_x = 0, \\
			\tau_2 \rho \left( S_t + u S_x \right) + S = \mu u_x,
		\end{cases}
	\end{align}
	where $\tau_1$ and $\tau_2$ denote the thermal and stress relaxation times, respectively.  The system \eqref{1.1}, \eqref{1.2} constitutes a hyperbolic reformulation of the classical compressible Navier-Stokes-Fourier system \eqref{1.1}-\eqref{Newton}, thereby introducing the physically realistic property of finite propagation speed, which is absent in the original parabolic framework. Derivations of the constitutive equations \eqref{1.2} are provided in \cite{CJ05,Maxwell67}. 
	
	Note that when constitutive equations \eqref{1.2} are adopted, the thermodynamic relations must be modified accordingly to remain compatible with the second law of thermodynamics (see, e.g., \cite{CN63,HR20,HR23}). Following the approach of Coleman \cite{CN63} , we therefore assume that the internal energy $e$ and the pressure $p$ take the form
	\begin{align*}
		e(\theta,q)= C_v \theta+a(\theta)q^2, \qquad\qquad  p(v, \theta)=\frac{R\theta}{v},
	\end{align*}
	where
	$$
	a(\theta)=\frac{\tau_1}{2 \kappa \theta}.$$
	Here, $C_v=\frac{R}{\gamma-1}>0, R>0$ denote the specific heat at constant volume and the gas constant, respectively. 
	For computational convenience, we introduce the specific volume $v:=\frac{1}{\rho}$ (volume per unit mass) and reformulate the system \eqref{1.1}, \eqref{1.2} in  Lagrangian coordinates as follows:
	\begin{equation}\label{1.3}
		\begin{cases}
			v_t - u_x = 0, \\
			u_t + p(v, \theta)_x = S_x, \\
			\left(C_v+a^{\prime}(\theta)q^{2}\right)\theta_{t} -\frac{2a(\theta)\kappa}{\tau_{1}}q\theta_{x}+ p u_x + q_x = \frac{2a(\theta)v}{\tau_{1}}q^2+\frac{v}{\mu} S^2,\\
			\tau_1 q_t + vq + \kappa\theta_x = 0, \\
			\tau_2 S_t + vS = \mu u_x.
		\end{cases}
	\end{equation}
	We consider the Cauchy problem for system \eqref{1.3} with initial data given by
	\begin{align}\label{1.4}
		(v, u, \theta, q, S)|_{t=0}=(v_0, u_0, \theta_0, q_0, S_0)(x) \rightarrow (v_{\pm}, u_{\pm}, \theta_{\pm}, 0, 0),\quad \text{as} \, x\rightarrow \pm \infty.
	\end{align}
	Here, we assume, for the far field conditions, $v_+\neq v_-, u_+\neq u_-, \theta_+\neq \theta_-$ in general. Furthermore, $(v_-, u_-, \theta_-)$ and $(v_+, u_+, \theta_-)$ are supposed to be the Riemann initial data which generates a centered rarefaction wave
	for the following full Euler system
	\begin{equation} 
		\begin{cases}\label{1.5}
			v_{t} - u_x=0, \\
			u_{t} + p_x=0, \\  
			(C_v\theta+\frac{1}{2} u^2 )_{t} + (pu)_x=0,
		\end{cases}
	\end{equation}
	with initial data
	\begin{align}\label{1.6}
		(u,v,\theta)(0,x) = (u_{0}^{r}, v_{0}^{r}, \theta_{0}^{r})(x) = \begin{cases} 
			(u_{-},v_{-},\theta_{-}), & x<0, \\ 
			(u_{+},v_{+},\theta_{+}), & x > 0. \\ 
		\end{cases}
	\end{align}
	
	We note that substantial evidence indicates that the long-time behavior of the Cauchy problem for the classical compressible Navier-Stokes-Fourier system \eqref{1.1}-\eqref{Newton} is governed by the solution to the Riemann problem \eqref{1.5}-\eqref{1.6}; see, for example, \cite{KVW25}. In particular, Matsumura and Nishihara \cite{MN86,MN92} were the first to establish the asymptotic stability of rarefaction waves for the isentropic Navier-Stokes system, while the full system was later treated in \cite{KMN, LX88, NYZ04}, where analogous results were obtained. The half-space problem was investigated by Kawashima and Zhu in \cite{KZ09}.   Moreover, the asymptotic stability of composite waves-combining rarefaction waves, traveling waves, and viscous contact waves-has also been extensively investigated in \cite{Goodman86,HH12,HLM10,Kanel68,Liu86,SX93}.

	On the other hand, the asymptotic stability of rarefaction waves has also been demonstrated for several related systems: see \cite{MA87} for the Broadwell model, \cite{KT} for a model system of radiating gas, \cite{NNK, BHZ} for a hyperbolic balance law model, and \cite{HW23} for compressible Navier-Stokes system with Maxwell's law.   The purpose of this work is to analyze the time-asymptotic stability of the full relaxed system \eqref{1.3} under the initial condition \eqref{1.4}. This extends the results of \cite{HW23} to the non-isentropic case.

	For reader's convenience, we first recall the rarefaction wave solution to the Riemann problem \eqref{1.5}-\eqref{1.6}. For any $(v_{+}, u_{+}, \theta_{+}) \in \mathbb{R}_{+} \times \mathbb{R}\times\mathbb{R}_{+}$, the 1-rarefaction curve $R_{1}(u_{+},v_{+},\theta_{+})$ is defined as:
	\[
	R_1(v_+, u_+, \theta_+) := \left\{(v, u, \theta) \ \middle| \ 
	\begin{aligned}
		v < v_{+}, 
		s(v, \theta) = s(v_{+}, \theta_{+}) =:s_{+}, 
		u = u_{+} - \int_{v_{+}}^{v} \lambda_{1}(s_{+},v') dv'
	\end{aligned}
	\right\}
	\]
	where the entropy function 
	\[
	s(v, \theta) = \dfrac{R}{\gamma-1} \ln\left(\dfrac{R}{A} \theta v^{\gamma-1}\right)
	\] and  $\lambda_{1}(s,v)=-\sqrt{\frac{\gamma p(v,s)}{v}}$ is the first eigenvalue of the matrix
	\[
	A(U) = \begin{pmatrix}
		0 & -1 & 0 \\
		-\dfrac{p}{v} & -\dfrac{(\gamma-1)u}{v} & \dfrac{\gamma-1}{v} \\
		-\dfrac{pu}{v} & p - \dfrac{(\gamma-1)u^2}{v} & \dfrac{(\gamma-1)u}{v}
	\end{pmatrix}.
	\]
	If $(v_-, u_-, \theta_-)\in R_1(v_+, u_+, \theta_+)$, then Riemann problem \eqref{1.5}-\eqref{1.6}  admits a continuous weak solution of the form $(v^r, u^r, \theta^r)(x/t)$ which is called centered rarefaction wave defined as
	\begin{equation}\label{1.7}
		\lambda_1(v^r(t,x), \theta^r(t,x)) = 
		\begin{cases} 
			\lambda_1(v_-, \theta_-), & x < \lambda_1(v_-, \theta_-)t, \\ 
			\dfrac{x}{t}, & \lambda_1(v_-, \theta_-)t \leq x \leq \lambda_1(v_+, \theta_+)t, \\ 
			\lambda_1(v_+, \theta_+), & x > \lambda_1(v_+, \theta_+)t, 
		\end{cases}	 
	\end{equation}
	and
	\begin{equation}\label{1.8}
		\begin{aligned}
			z_{1}(u^{r}(t,x),v^{r}(t,x)) &= z_{1}(u_{-}, v_{-}) = z_{1}(u_{+}, v_{+}),\\
			s(v^{r}(t,x),\theta^{r}(t,x)) &= s(v_{-},\theta_{-}) = s(v_{+},\theta_{+}),
		\end{aligned}
	\end{equation}
	where $z_{1}(u,v)=u+\int_{v_{+}}^{v}\lambda_{1}(s, v^{\prime})dv^{\prime}$. 3-rarefaction curve $R_{3}$ can be defined via the characteristic speed $\lambda_{3}:=-\lambda_1$ in a similar way.
	
	Before we state the main theorem, we introduce the following noations.  Let $L^p(\mathbb{R})$ and $W^{s,p}(\mathbb{R})$ (for $1 \le p \le \infty$) denote the usual Lebesgue and Sobolev spaces on $\mathbb{R}$, equipped with the norms $\|\cdot\|_{L^p}$ and $\|\cdot\|_{W^{s,p}}$, respectively. Note that $W^{0,p}(\mathbb{R}) = L^p(\mathbb{R})$. As usual, we write $H^s(\mathbb{R})$ for $W^{s,2}(\mathbb{R})$. Let $T > 0$ be a constant and $B$ be a Banach space. We denote by $C^k(0, T; B)$ ($k \ge 0$) the space of $B$-valued, $k$-times continuously differentiable functions on $[0, T]$, and by $L^p(0, T; B)$ the space of $B$-valued $L^p$-functions on $[0, T]$. The corresponding spaces of $B$-valued functions on $[0, \infty)$ are defined similarly.
	
	Now, Let 
	\[
	I_0:= \|(v_0- v_0^r, u_0-u_0^r, \theta_0-\theta_0^r)\|_{L^2}+\|((v_0)_x, (u_0)_x, (\theta_0)_x)\|_{H^1}+\|(q_0, S_0)\|_{H^2}.
	\]
	Our main result is stated as follows.
	\begin{theorem}\label{th1.1}
		Let $(v_-, u_-, \theta_-) \in R_1 (v_+, u_+, \theta_+)$ and $\delta=|v_+-v_-|+|u_+-u_-|+|\theta_+-\theta_-|$. Assume the initial data $(v_0, u_0, \theta_0, q_0, S_0)$ satisfy
		\begin{align*}
			(v_0- v_0^r, u_0-u_0^r, \theta_0-\theta^r, q_0,  S_0) \in L^2,\\
			( (v_0)_x, (u_0)_x, (\theta_0)_x)\in H^1, (q_0, S_0) \in H^2,
		\end{align*}
		and there exists a positive constant $\epsilon_0$ such that if $I_0+\delta < \epsilon_0$, then the initial value problem \eqref{1.3}-\eqref{1.4} has a unique global solution in time satisfying
		\begin{align}\label{1.9}
			\begin{cases}
				(v-v_0^r, u-u_0^r, \theta-\theta_0^r, q,  S)\in C^0(0, +\infty, L^2),\\
				(v_x, u_x, \theta_x)\in C^0(0, +\infty, H^1)\cap L^2(0, +\infty, H^1), (q, S) \in C^0(0, +\infty, H^2)\cap L^2 (0, +\infty, H^2).
			\end{cases}
		\end{align}
		Moreover, this solution approaches the rarefaction wave $(v^r, u^r,\theta^r, 0, 0)$ uniformly in $x\in \mathbb R$ as $t\rightarrow +\infty$:
		\begin{align}\label{2.2}
			\|(v(t,x)-v^r(x/t), u(t,x)-u^r(x/t), \theta(t,x)-\theta^r(t,x), q(t,x),  S(t,x))\|_{L^\infty} \rightarrow 0, \quad \mathrm{as}\quad t\rightarrow \infty.
		\end{align}
	\end{theorem}
	
	Now, we outline the proof strategy for the main theorem. First, within a relative entropy framework centered around a smooth approximate rarefaction wave, we establish the key $L^2$ estimates by  leveraging the inherent thermodynamic structure of the system. In addition to the standard $L^2$ estimates, we obtain an extra weighted dissipation term arising from the gradient of the rarefaction wave. This term plays a decisive role in closing the a priori estimates and in establishing the global stability of the rarefaction wave. Next, we derive $H^2$ estimates and dissipation estimates via the conventional energy method. Finally, combining these a priori estimates with the local existence theorem yields the desired result.
	
	We remark that the analysis for the non-isentropic system is considerably more challenging than that for the isentropic case, due to the inclusion of the temperature equation and the associated relaxation variables. These introduce strong nonlinear couplings absent in the isentropic setting. In particular, the standard $L^2$-energy method applicable to isentropic models fails to close here, owing to the presence of uncontrolled mixed terms.
	
	The paper is structured as follows. We begin in Section 2 by constructing smooth approximations of the rarefaction waves, following the methodology of \cite{MN86, MN92}. In Section 3, we reformulate the original problem and state the main asymptotic stability theorems for the new system (Theorem \ref{thm3.1}). This section also contains the statement of the a priori estimates required for the proof (Proposition \ref{prop3.2}). The detailed proofs of these a priori estimates are then provided in Section 4.

	\section{Smooth approximation of the rarefaction wave}
	Let $(u_{+},v_{+},\theta_{+})$ be given, and $(u_{-},v_{-},\theta_{-})\in R_{1}(u_{+},v_{+},\theta_{+})$.
	Define $w_{-}=\lambda_{1}(v_{-},\theta_{-})$ and $w_{+}=\lambda_{1}(v_{+},\theta_{+})$. We begin with the inviscid Burgers equation
	\begin{equation}
		\begin{cases} \label{3.1}
			w_t + w w_x = 0, \\ 
			w(0,x) = w_0^r(x) = 
			\begin{cases} 
				w_-, & x < 0, \\ 
				w_+, & x > 0. 
			\end{cases} 
		\end{cases}
	\end{equation}
	Let $w_{-}<w_{+}$, the equation \eqref{3.1} admits the self-similar solution $w^{r}(t,x)=w(x/t)$, 
	\begin{equation}\label{3.2}
		w^r(t,x) = w^r\left(\frac{x}{t}\right) = 
		\begin{cases} 
			w_-, & x < w_- t, \\ 
			\dfrac{x}{t}, & w_- t \leq x \leq w_+ t, \\ 
			w_+, & x > w_+ t. 
		\end{cases}
	\end{equation}
	The 1-rarefaction wave \((u^r, v^r, \theta^r)(t, x) = (u^r, v^r, \theta^r)\big(\frac{x}{t}\big)\) of the Riemann problem \eqref{1.5}, defined by equations \eqref{1.6} and \eqref{1.7}, can be expressed as
	\begin{equation}\label{3.3}
		\begin{aligned}
			\lambda_{1}\left( v^{r}\left( \frac{x}{t} \right), \theta^{r}\left( \frac{x}{t} \right) \right) &= w^{r}\left( \frac{x}{t} \right), \\
			z_{1}\left( u^{r}\left( \frac{x}{t} \right), v^{r}\left( \frac{x}{t} \right) \right) &= z_{1}(u_{-}, v_{-}) = z_{1}(u_{+}, v_{+}), \\
			s\left( v^{r}\left( \frac{x}{t} \right), \theta^{r}\left( \frac{x}{t} \right) \right) &= s(v_{-}, \theta_{-}) = s(v_{+}, \theta_{+}).
		\end{aligned}
	\end{equation}
	The self-similar solution $(u^{r},v^{r},\theta^{r})(x/t)$ is Lipschitz continuous and satisfies the Euler system when $t>0$:
	\begin{equation}\label{3.4}
		\begin{cases}
			v^{r}_{t}-u^{r}_{x}=0,\\
			u^{r}_{t}+p(v^{r},\theta^{r})_{x}=0,\\
			\frac{R}{\gamma-1}\theta^{r}_{t}+p(v^{r},\theta^{r})u^{r}_{x}=0.
		\end{cases}
	\end{equation}
	Let $\delta_{R}:=|v_{+}-v_{-}|$ denote rarefaction wave strength.
	It follows from \eqref{3.3} that \(\delta_R \sim |u_+ - u_-| \sim |\theta_+ - \theta_-|\). We now construct a smooth approximation. Consider the Burgers equation
	\begin{equation}\label{3.5}
		\begin{cases}
			w_t + w w_x = 0, \\ 
			w(0,x) = w_0^r(x)=\frac{w_{+}+w_{-}}{2}+\frac{w_{+}-w_{-}}{2}K_{q}\int_{0}^{\epsilon x}(1+y^{2})^{-q}dy,
		\end{cases}
	\end{equation}
	where $\epsilon>0$ is a constant, $K_{q}$ is a constant satisfying $K_{q}\int_{0}^{\epsilon x}(1+y^{2})^{-q}dy=1$, $q>\frac{3}{2}$. 
	We define a smooth approximate 1‑rarefaction wave $(u^{R},v^{R},\theta^{R})(t,x)$ via the relations
	\begin{equation}\label{3.6}
		\begin{aligned}
			\lambda_{1-} &:= \lambda_{1}(u_{-},v_{-}) = w_{-}, \\
			\lambda_{1+} &:= \lambda_{1}(u_{+},v_{+}) = w_{+}, \\
			\lambda_{1}(v^{R},\theta^{R})(t,x) &= w(1+t,x), \\
			z_{1}(u^{R},v^{R})(t,x) &= z_{1}(u_{-},v_{-}) = z_{1}(u_{+},v_{+}), \\
			s(v^{R},\theta^{R})(t,x) &= s(v_{-},\theta_{-}) = s(v_{+},\theta_{+}),
		\end{aligned}
	\end{equation}
	where $w(t,x)$ is the smooth solution of Burgers equation (\ref{3.5}).
	The smooth approximation solution $(u^{R},v^{R},\theta^{R})(t,x)$ satisfies the Euler system
	\begin{equation}\label{3.7}
		\begin{cases}
			v^{R}_{t}-u^{R}_{x}=0,\\
			u^{R}_{t}+p(v^{R},\theta^{R})_{x}=0,\\
			\frac{R}{\gamma-1}\theta^{R}_{t}+p(v^{R},\theta^{R})u^{R}_{x}=0.
		\end{cases}
	\end{equation}
	\begin{lemma}\label{lem2.1}
		For the 1-rarefaction waves $(u^{R},v^{R},\theta^{R})(t,x)$ defined by \eqref{3.6}, let  $\delta_{R}=|v_{+}-v_{-}|\sim|u_{+}-u_{-}|\sim|\theta_{+}-\theta _{-}|$ be the wave strength. Then, we have\\
		(1) $u^{R}_{x}=\frac{2v^{R}}{(\gamma+1)}w_{x}>0$, 
		$v^{R}_{x}=\frac{v^{R}}{(\sqrt{R\gamma \theta^{R}})}u^{R}_{x}>0$,
		$\theta^{R}_{x}=-\frac{(\gamma-1)\theta^{R}}{v^R}v^{R}_{x}<0$ for   $\forall x \in \mathbb{R},\ t \geq 0$.\\
		(2) For any $t>0$, $p\in[1,+\infty)$, $k=2,3,4,$
		$$\|(u_{x}^{R},v^{R}_{x},\theta^{R}_{x})\|_{L^{p}}\leq min\left\{\delta_{R},\delta_{R}^{\frac{1}{p}}(1+t)^{-1+1/p}\right\},$$
		$$\|(\partial_x^k u^{R},\partial_x^kv^{R},\partial_x^k\theta^{R})\|_{L^{p}}\leq
		C(\delta^{-\frac{p-1}{2p}}_{R}\varepsilon_{R}^{(1-\frac{1}{2p})(1-\frac{1}{p})}(1+t)^{-1-\frac{p-1}{2pq}}+\delta^{\frac{1}{p}}_{R}(1+t)^{-2+\frac{1}{p}}),$$
		$$|u_{xx}^{R}|\leq C|u^{R}_{x}|, |\theta_{xx}^{R}|\leq C|\theta^{R}_{x}|,\forall x\in \mathbb{R}.$$
		(3) For $x\geqslant \lambda_{1+}(1+t)$, and $t>0$,
		\begin{center}
			$|(u^{R},v^{R},\theta^{R})(t,x)-(u_{+},v_{+},\theta_{+})|\leq C\delta_{R}e^{-2|x-\lambda_{1+}(1+t)|}$,\\
			$|(u^{R}_x,v^{R}_x,\theta^{R}_x)(t,x)|\leq C\delta_{R}e^{-2|x-\lambda_{1+}(1+t)|}$.
		\end{center}
		(4) For $x\leq \lambda_{1-}(1+t)$, and $t>0$,
		\begin{center}
			$|(u^{R},v^{R},\theta^{R})(t,x)-(u_{-},v_{-},\theta_{-})|\leq C\delta_{R}e^{-2|x-\lambda_{1-}(1+t)|}$,\\
			$|(u^{R}_x,v^{R}_x,\theta^{R}_x)(t,x)|\leq C\delta_{R}e^{-2|x-\lambda_{1-}(1+t)|}$.
		\end{center}
		(5) 
		$
		\lim_{t \to +\infty} \sup_{x \in \mathbb{R}} \left| \left( u^R, v^R, \theta^R \right)(t, x) - \left( u^r, v^r, \theta^r \right)\left( \frac{x}{t} \right) \right| = 0.
		$
	\end{lemma}
	
	\section{Reformulation of the Problem}
	
	The corresponding Euler system associated with the smooth rarefaction wave solution reads as follows
	\begin{equation}\label{3.9}
		\begin{cases}
			v^{R}_{t}-u^{R}_{x}=0,\\
			u^{R}_{t}+p(v^{R},\theta^{R})_{x}=0,\\
			\frac{R}{\gamma-1}\theta^{R}_{t}+p(v^{R},\theta^{R})u^{R}_{x}+q_{x}^{R}=0,\\
			v^Rq^R+\kappa\theta^R_{x}=0,\\
			v^RS^R=\mu u^R_{x}.
		\end{cases}
	\end{equation}
	Subtracting equation \eqref{3.9} from equation \eqref{1.3}, and introducing the following perturbation variables
	\begin{align*}
		\varphi=v-v^R,\quad\psi=u-u^R,\quad\tilde{\theta}=\theta-\theta^R,\quad\tilde{q}=q-q^R,\quad\tilde{S}=S-S^R,
	\end{align*} we obtain
	\begin{equation}\label{3.10}
		\begin{cases}
			\varphi_t - \psi_x = 0, \\
			\psi_t + (p-p^{R})_x = \tilde{S}_x-Q_{1}^{R},\\
			\frac{R}{\gamma-1}\tilde{\theta}_{t}+a^{\prime}(\theta)
			q^{2}\theta_{t} -\frac{2a(\theta)\kappa}{\tau_{1}}\theta_{x}q+
			pu_x-p^{R}u_x^{R}+\tilde{q}_x=\frac{2a(\theta)v}{\tau_{1}}q^2+\frac{v}{\mu} S^2-Q_{2}^{R},\\
			\tau_1\tilde{q}_t + vq-v^{R}q^{R} + \kappa\tilde{\theta}_x = -	\tau_1q^R_t, \\
			\tau_2 \tilde{S}_t + vS-v^{R}S^{R} = \mu \psi_x-\tau_2 S^R_t,
		\end{cases}
	\end{equation}
	where $Q_{1}^{R}=-\mu(\frac{u_{x}^{R}}{v^{R}})_{x}$, $Q_{2}^{R}=-\kappa(\frac{\theta_{x}^{R}}{v^{R}})_{x}$.  Initial condition
	\begin{align}\label{y3.11}
		(\varphi,\psi,\tilde{\theta},\tilde{q},\tilde{S})(x,0)=(\varphi_{0},\psi_{0},\tilde{\theta}_{0},\tilde{q}_{0},\tilde{S}_{0})(x),\quad x\in\mathbb{R},
	\end{align}
	where
	\begin{align*}
		\varphi_{0}=v_{0}-v^R_{0},\quad\psi_{0}=u_{0}-u^R_{0},\quad\tilde{\theta}_{0}=\theta_{0}-\theta^R_{0},\quad\tilde{q}_{0}=q_{0}-q^R_{0},\quad\tilde{S}_{0}=S_{0}-S^R_{0}.
	\end{align*}
	For the reformulated system \eqref{3.10}-\eqref{y3.11}, we establish the following results on global existence and asymptotic stability.
	\begin{theorem}\label{thm3.1}
		Let $(v_-,u_-,\theta_-)\in R_1(v_+,u_+,\theta_+)$ and $\delta=|v_+-v_-|+|u_+-u_-|+|\theta_+-\theta_-|$. Assume the initial data $(\varphi_0,\psi_0,\tilde{\theta}_0,\tilde{q}_0,\tilde{S}_0)\in H^2$ and let $E_0:=\|(\varphi_0,\psi_0,\tilde{\theta}_0,\tilde{q}_0,\tilde{S}_0)\|_{H^2}^2$. Then, there exists a positive constant $\delta_1$ such that if $E_0+\delta<\delta_1$, the initial value problem \eqref{3.10}-\eqref{y3.11} has a unique global solution in time satisfying
		\begin{align}\label{3-4}
			&(\varphi,\psi,\tilde{\theta},\tilde{q},\tilde{S})(t,x)\in C^0(0,+\infty,H^2),\\
			\label{3-5}&(\varphi_x,\psi_x)(t,x)\in L^2(0,+\infty,H^1),\quad(\tilde{\theta}_x,\tilde{q}_x,\tilde{S}_x)\in L^2(0,+\infty,H^2).
		\end{align}
		Moreover, this solution decay to $(0,0,0)$ uniformly in $x$ as $t\rightarrow \infty$:
		\begin{align*}
			\|(\varphi,\psi,\tilde{\theta},\tilde{q},\tilde{S})\|_{W^{1,\infty}}\rightarrow 0,\quad as \quad t\rightarrow \infty.
		\end{align*}
	\end{theorem}

	To prove Theorem \ref{thm3.1}, the key point is to show the a priori estimate of solutions to the problem \eqref{3.10}-\eqref{y3.11}. To do this, we define the energy
	$$
	E(t) := \sup_{0\le s\le t}\|(\varphi,\psi,\tilde\theta,\tilde q,\tilde S)(s,\cdot)\|_{H^2}^2,
	$$
	and dissipation
	\begin{align*}
		\mathcal{D}(t) :=  \|(\varphi_x,\psi_x,\tilde\theta_x)(t,\cdot)\|_{H^1}^2
		+ \|(\tilde q,\tilde S)(t,\cdot)\|_{H^2}^2 +\int_\mathbb{R}|v^R_{x}||(\varphi,\tilde{\theta})|^2\mathrm{d}x .
	\end{align*}
	For notational simplicity, in what follows we write $\int$ for $\int_\mathbb{R}$.
	
	The a priori estimate result is stated as follows.
	
	\begin{proposition}\label{prop3.2}
		Let $T>0$ and suppose $(\phi,\psi,\tilde\theta,\tilde q,\tilde S)$ is a smooth solution
		on $[0,T]$ with
		\[
		(\phi,\psi,\tilde\theta,\tilde q,\tilde S)\in C^0([0,T];H^2)\cap C^1([0,T];H^1).
		\]
		Then there exist constants
		$C>0$ and $\delta_2>0$ (independent of $T$) such that if
		\begin{align}\label{3-6}
			E(T)+\delta<\delta_2,
		\end{align}
		then the solution $(\varphi,\psi,\tilde{\theta},\tilde{q},\tilde{S})$ satisfies
		\begin{align}\label{3-7}
			E(t) +\int_{0}^{s} \mathcal{D}(t) \mathrm{d}s \le C\big( E_0 + E(t)^{3/2} + \delta^\theta \big)
		\end{align}
		for all $t\in(0,T)$ with $\theta=\min\{\frac{1}{2},\frac{3}{2}-\frac{1}{q}\}$.
	\end{proposition}

	\begin{proof}[Proof of Theorem \ref{thm3.1}]
		To begin with, we choose a constant $\delta_2$ such that $C \delta_2^{1/2} \leq \dfrac{1}{2}$.
		Then, under the assumption \ref{prop3.2}, we have
		\begin{align}\label{3-8}
			E(t) \leq 2C(E_0 + \delta^\theta). 
		\end{align}
		Next, we choose $\delta_1$ to be sufficiently small, satisfying $2C(\delta_1 + \delta_1^\theta) + \delta_1 < \frac{\delta_2}{2}$. From this, it follows that $E(t) + \delta < \frac{\delta_2}{2}$ by noting that the initial condition satisfies $E_0 + \delta < \delta_1 < \frac{\delta_2}{2}$. We thus close the a priori assumption \eqref{3-6}. Therefore, combining the local existence theorem with the a priori estimate \eqref{3-8}, we construct a global-in-time solution with regularity \eqref{3-4} and \eqref{3-5} through classical continuation methods.
		
		Since $\mathcal{D}(t)$ is integrable in time, we have
		$(\phi_x,\psi_x,\tilde\theta_x)\in L^2(0,\infty;H^1)$
		and $(\tilde q,\tilde S)\in L^2(0,\infty;H^2)$.
		Together with the uniform bound $E(t)\le C$ and the Sobolev interpolation inequality, we have
		\[
		\|(\phi,\psi,\tilde\theta,\tilde q,\tilde S)(t)\|_{L^\infty}
		\le C\|(\phi,\psi,\tilde\theta,\tilde q,\tilde S)(t)\|_{L^2}^{1/2}\|(\phi_x,\psi_x,\tilde\theta_x,\tilde q_x, \tilde S_x)(t)\|_{L^2}^{1/2},
		\]
		\[
		\|(\phi_x,\psi_x,\tilde\theta_x,\tilde q_x, \tilde S_x)(t)\|_{L^\infty}
		\le C\|(\phi_x,\psi_x,\tilde\theta_x,\tilde q_x, \tilde S_x)(t)\|_{L^2}^{1/2}\|(\phi_{xx},\psi_{xx},\tilde\theta_{xx},\tilde q_{xx}, \tilde S_{xx})(t)\|_{L^2}^{1/2},
		\]
		this yields
		\begin{align}\label{y3.6}
			\|(\phi,\psi,\tilde\theta,\tilde q,\tilde S)(t)\|_{W^{1,\infty}}\to0
			\quad\text{as }t\to\infty.
		\end{align}
		Thus, the proof of the global existence and asymptotic stability is completed.
	\end{proof}
	Finally, applying Theorem \ref{thm3.1}, we conclude the proof of Theorem \ref{th1.1}.
	\begin{proof}[Proof of Theorem \ref{th1.1}]	
		Let $I_0+\delta$ is suitable small where
		\begin{align*}
			I_0 = \| (v_0 - v_0^r, u_0 - u_0^r) \|_{L^2} + \| ((v_0)_x, (u_0)_x) \|_{H^1} + \|(\tilde\theta_0,\tilde q_0,\tilde{S}_0)\|_{H^2}.
		\end{align*}
		Then, we have
		\[
		\|(\varphi_0, \psi_0)\|_{L^2} \le \|(v_0 - v_0^r, u_0 - u_0^r)\|_{L^2} + \|(v_0^r - v^R_0, u_0^r - u^R_0)\|_{L^2} \le I_0 + C\delta,
		\]
		and
		\begin{align*}
			&\|\partial_x(\varphi_0, \psi_0)\|_{H^1} + \|(\tilde\theta_0,\tilde q_0,\tilde S_0)\|_{H^2} \\
			\le& \|\partial_x(v_0, u_0)\|_{H^1} + \|(\theta_{0}, q_0, {S}_0)\|_{H^2} + \|\partial_x(v^R_0, u^R_0)\|_{H^1} 
			+ \|(\theta^{R}_0, q^R_0, S^R_0)\|_{H^2} 
			\le I_0 + C\delta.
		\end{align*}
		Hence, $E_0 + \delta \le I_0 + C\delta$, which can be made sufficiently small.  
		Applying Theorem \ref{thm3.1}, we obtain a unique global solution $(\varphi, \psi, \tilde\theta, \tilde{q}, \tilde S)$ to problem \eqref{3.10}-\eqref{y3.11}.  
		Define 
		\[
		(v,u,\theta, q, S) = (\varphi + v^R, \psi + u^R, \tilde\theta+\theta_{R}, \tilde q+q^R, \tilde S + S^R)
		\]
		which solves the original system \eqref{1.3}-\eqref{1.4}.  
		To verify the convergence in (\ref{2.2}), observe that
		\[
		\|(v - v^r, u - u^r, \theta, q, S)\|_{L^\infty} 
		\le \|(\varphi, \psi, \tilde{\theta}, \tilde{q}, \tilde{S})\|_{L^\infty} + \|(v^R - v^r, u^R - u^r, \theta^{R}, q^R, S^R)\|_{L^\infty} \to 0 
		\]
		as $t \to +\infty$, where we have used the convergence result (\ref{y3.6}) and Lemma \ref{lem2.1}.  
		This completes the proof of Theorem \ref{th1.1}.
	\end{proof}

	\section{Uniform a Priori Estimates}
	
	To derive the key $L^2$ estimates, we introduce the relative entropy as follows:
	\begin{equation}\label{3.11}
		\eta=\frac{R}{\gamma-1}\theta^{R}\phi(\frac{\theta}{\theta^{R}})+R\theta^{R}\phi(\frac{v}{v^{R}})+\frac{1}{2}\psi^2+\frac{\tau_1}{2\kappa\theta}\tilde{q}^2+\frac{\tau_2}{2\mu}\tilde{S}^2,
	\end{equation}
	where $\phi(v)=v-1-\ln v$.
	\subsection{$L^2$ estimates}
	\begin{lemma} \label{lem4.1}
		There exists a constant $C$ such that for any $0\le t\le T$,
		\begin{equation}
			\begin{aligned}
				&\|(\varphi,\psi,\tilde{\theta},\tilde{q},\tilde{S})\|^2_{L^2}+\int_{0}^{t}\int\left(\frac{v}{\kappa\theta}\tilde{q}^2+\frac{v}{\mu}\tilde{S}^2\right)\mathrm{d}x\mathrm{d}t
				+\int_{0}^{t}\int|v^R_{x}||(\varphi,\tilde{\theta})|^2\mathrm{d}x\mathrm{d}t\\
				\leq& C\left(I_{0}+\delta_R\int_{0}^{t}\int\tilde{\theta}_x^2\mathrm{d}x+\delta_{R}^\frac{1}{2}
				+E^{\frac{1}{2}}(t)\int_{0}^{t}\mathcal{D}(t)\mathrm{d}t\right).
			\end{aligned}
		\end{equation}
	\end{lemma}

	\begin{proof}
		We begin by deriving an expression for the time derivative of the entropy function.	First, from equations $(\ref{3.10})_{3}$ and $(\ref{3.10})_{4}$, we compute
		\begin{equation}
			\begin{aligned} \label{3.13}
				&\left[\frac{R}{\gamma-1}\theta^{R}\phi\left(\frac{\theta}{\theta^{R}}\right)\right]_t \\
				=& \frac{R}{\gamma-1}\theta_t^{R}\phi\left(\frac{\theta}{\theta^{R}}\right) 
				+ \frac{R}{\gamma-1}\theta^{R} \left(\frac{1}{\theta^R}-\frac{1}{\theta}\right)
				\left[ \tilde{\theta}_t + \theta_t^{R}\left(1-\frac{\theta}{\theta^{R}}\right)\right] \\
				=& \frac{R}{\gamma-1}\theta_t^{R}\phi\left(\frac{\theta}{\theta^{R}}\right) 
				+\theta^R\left(\frac{1}{\theta^R}-\frac{1}{\theta}\right)\Biggl[ -a'(\theta)q^{2}\theta_{t} 
				+ \frac{2a(\theta)\kappa}{\tau_{1}}\theta_{x}q-pu_x+ p^{R}u_x^{R}\\
				&-\tilde{q}_x+\frac{2a(\theta)v}{\tau_{1}}q^2 
				+\frac{v}{\mu} S^2-Q_{2}^{R}\Biggr]
				+\frac{R}{\gamma-1}\frac{\tilde{\theta}}{\theta}\theta_{t}^{R}\left(1-\frac{\theta}{\theta^{R}}\right)\\
				=&\frac{R}{\gamma-1}\theta_t^{R}\phi\left(\frac{\theta}{\theta^{R}}\right)+\frac{v\tilde{\theta}}{\kappa\theta^2}\left(\tilde{q}^2+2q^R\tilde{q}+(q^R)^2\right)+\frac{\tau_1\theta_{t}}{2\kappa\theta^3}\tilde{\theta}\left(\tilde{q}^2+2q^R\tilde{q}+(q^R)^2\right)+\frac{\theta_x}{\theta^2}\tilde{\theta}q^R\\
				&-\frac{R}{v}\tilde{\theta}\psi_x-\frac{u_x^R}{\theta}\tilde{\theta}(p-p^R)+\frac{\tilde{\theta}}{\theta}\frac{S^2}{\mu}v
				+p^Ru_x^R\frac{\tilde{\theta}^2}{\theta\theta^R}
				+(-\frac{\tilde{\theta}}{\theta}\tilde{q})_x-\frac{\tilde{\theta}}{\theta}Q_2^R+\frac{\theta_x}{\theta^2}\tilde{q}\\
				=&\frac{R}{\gamma-1}\theta_t^{R}\phi\left(\frac{\theta}{\theta^{R}}\right)+\frac{v\tilde{\theta}}{\kappa\theta^2}\left(\tilde{q}^2+2q^R\tilde{q}+(q^R)^2\right)+\frac{\tau_1\theta_{t}}{2\kappa\theta^3}\tilde{\theta}\left(\tilde{q}^2+2q^R\tilde{q}+(q^R)^2\right)+\frac{\theta_x}{\theta^2}\tilde{\theta}q^R\\
				&-\frac{R}{v}\tilde{\theta}\psi_x-\frac{u_x^R}{\theta}\tilde{\theta}(p-p^R)+\frac{\tilde{\theta}}{\theta}\frac{S^2}{\mu}v
				+p^Ru_x^R\frac{\tilde{\theta}^2}{\theta\theta^R}
				+(-\frac{\tilde{\theta}}{\theta}\tilde{q})_x-\frac{\tilde{\theta}}{\theta}Q_2^R\\
				&-\frac{\tau_1}{\kappa\theta}q_t^R\tilde{q}-\frac{\tau_{1}}{2\kappa}\left(\frac{\tilde{q}^2}{\theta}\right)_t-\frac{\tau_{1}}{2\kappa}\tilde{q}^2\frac{\theta_t}{\theta^2}-\frac{v}{\kappa\theta}\tilde{q}^2-\frac{q^R}{\kappa\theta}\varphi\tilde{q}.\\
			\end{aligned}
		\end{equation}
		Next, using equation $(\ref{3.10})_1$, we obtain
		\begin{equation}
			\begin{aligned}\label{3.14}
				&\left[R\theta^{R}\phi\left(\frac{v}{v^{R}}\right)\right]_t \\
				=& R\theta_t^{R}\phi\left(\frac{v}{v^{R}}\right) 
				+ R\theta^{R}\phi'\left(\frac{v}{v^{R}}\right)\left(\frac{v}{v^{R}}\right)_t \\
				=& R\theta_t^{R}\phi\left(\frac{v}{v^{R}}\right) 
				+R\theta^R\left(\frac{1}{v^R}-\frac{1}{v}\right)
				\left[\varphi_t + v_t^{R}\left(1-\frac{v}{v^{R}}\right)\right]\\
				=& R\theta_t^{R}\phi\left(\frac{v}{v^{R}}\right) 
				+R\theta^R\left(\frac{1}{v^R}-\frac{1}{v}\right)
				\left[\psi_x + v_t^{R}\left(1-\frac{v}{v^{R}}\right)\right]\\
				=& R\theta_t^{R}\phi\left(\frac{v}{v^{R}}\right) 
				+R\theta^R\left(\frac{1}{v^R}-\frac{1}{v}\right)\psi_x- \frac{p^Ru_x^R}{v^Rv}\varphi^2.
			\end{aligned}
		\end{equation}
		Furthermore, from equation $(\ref{3.10})_{2,5}$, we get
		\begin{align}\label{3.15}
			\left(\frac{1}{2}\psi^2\right)_t=&\psi\psi_t
			=\psi\left[\tilde{S}_x-Q^R_1-(p-p^R)_x\right]\nonumber\\
			=&\left(\psi\tilde{S}\right)_x-\psi_x\tilde{S}-\psi Q^R_1-\left[\psi(p-p^R)\right]_x+\psi_x(p-p^R)\nonumber\\
			=&-\left[\psi(p-p^R)\right]_x+\left(\psi\tilde{S}\right)_x-\psi Q^R_1+\psi_x(p-p^R)
			-{\frac{\tau_2}{2\mu}(\tilde{S}^2)_t}-\frac{v}{\mu}\tilde{S}^2-\frac{S^R}{\mu}\varphi\tilde{S}-\frac{\tau_2S_t^R}{\mu}\tilde{S}.
		\end{align}
		Combining the equations \eqref{3.13}-\eqref{3.15}, we have
		\begin{equation}\label{3.16}
			\begin{aligned}
				\frac{\partial\eta}{\partial t}
				&=\left[\frac{R}{\gamma-1}\theta^{R}\phi(\frac{\theta}{\theta^{R}})+R\theta^{R}\phi(\frac{v}{v^{R}})+\frac{1}{2}\psi^2+\frac{\tau_1}{2\kappa\theta}\tilde{q}^2+\frac{\tau_2}{2\mu}\tilde{S}^2\right]_t\\
				&=R\theta_t^{R}\phi\left(\frac{v}{v^{R}}\right)- \frac{p^Ru_x^R}{v^Rv}\varphi^2\\
				&\quad+\frac{R}{\gamma-1}\theta_t^{R}\phi\left(\frac{\theta}{\theta^{R}}\right)-\frac{u_x^R}{\theta}\tilde{\theta}(p-p^R)+p^Ru_x^R\frac{\tilde{\theta}^2}{\theta\theta^R}\\
				&\quad+\left[-\frac{\tilde{\theta}}{\theta}\tilde{q}-\psi(p-p^R)+\psi\tilde{S}\right]_x\\
				&\quad+\frac{\tilde{\theta}}{\theta}\frac{S^2}{\mu}v+\frac{v\tilde{\theta}}{\kappa\theta^2}\left(\tilde{q}^2+2q^R\tilde{q}+(q^R)^2\right)-\frac{\tau_{1}}{2\kappa}\tilde{q}^2\frac{\theta_{t}}{\theta^{2}}+\frac{\tau_1\theta_{t}}{2\kappa\theta^3}\tilde{\theta}\left(\tilde{q}^2+2q^R\tilde{q}+(q^R)^2\right)\\
				&\quad+\frac{\theta_x}{\theta^2}\tilde{\theta}q^R-\frac{\tau_1}{\kappa\theta}q_{t}^{R}\tilde{q}-\frac{\tau_2S_t^R}{\mu}\tilde{S}
				-\frac{q^R}{\kappa\theta}\varphi\tilde{q}-\frac{S^R}{\mu}\varphi\tilde{S}-\frac{\tilde{\theta}}{\theta}Q_2^R-\psi Q^R_1-\frac{v}{\kappa\theta}\tilde{q}^2-\frac{v}{\mu}\tilde{S}^2.
			\end{aligned}
		\end{equation}
		We now analyze the terms on the right-hand side of \eqref{3.16}. Introduce 
		\begin{align*}
			&J_{1}=R\theta_t^{R}\phi\left(\frac{v}{v^{R}}\right)- \frac{p^Ru_x^R}{v^Rv}\varphi^2,\\
			&J_{2}=\frac{R}{\gamma-1}\theta_t^{R}\phi\left(\frac{\theta}{\theta^{R}}\right)-\frac{u_x^R}{\theta}\tilde{\theta}(p-p^R)+p^Ru_x^R\frac{\tilde{\theta}^2}{\theta\theta^R}.
		\end{align*}
		Recalling the facts that $\phi(1)=\phi'(1)=0$, and $\phi''(1)=1$, we expand $\phi\left(\frac{v}{v^R}\right)$ to second order:
		\begin{align}\label{4.7}
			\phi\left(\frac{v}{v^R}\right)=\frac{(v-v^R)^2}{2(v^R)^2}+o(|v-v^R|^3).
		\end{align}
		From equation $(\ref{3.7})_{3}$, we have $\theta^R_{t}=-\frac{\gamma-1}{R}p^Ru^R_{x}$.
		Substituting this and expansion \ref{4.7} into $J_1$ yields
		\begin{align}\label{3.18}
			J_{1}&=R\cdot\left(-\frac{\gamma-1}{R}p^Ru^R_{x}\right)\cdot\frac{\varphi^2}{2(v^R)^2}-\frac{p^Ru_x^R}{v^Rv}\varphi^2\nonumber\\
			&=-\frac{p^Ru^R_{x}}{v^R}\varphi^2\left(\frac{\gamma-1}{2v^R}+\frac{1}{v}\right)
			=-\frac{p^Ru^R_{x}}{v^R}\varphi^2\left(\frac{\gamma+1}{2v^R}+\frac{-\varphi}{vv^R}\right)\nonumber\\
			&\leq-\frac{(\gamma+1)p^Ru^R_{x}}{2(v^R)^2}\varphi^2+C|u^R_{x}||\varphi|^3.
		\end{align}
		For $J_2$, we first express the pressure difference as
		\begin{align*}
			p-p^R=\frac{R\theta}{v}-\frac{R\theta^R}{v^R}
			=\frac{R}{v}\tilde{\theta}-\frac{p^R}{v}\varphi.
		\end{align*}
		Substituting this and the second-order expansion of $\phi\left(\frac{\theta}{\theta^R}\right)$ into $J_2$ gives
		\begin{equation}\label{3.19}
			\begin{aligned}
				J_{2}&=-p^Ru^R_{x}\frac{\tilde{\theta}^2}{2(\theta^R)^2}-\frac{u_x^R}{\theta}\tilde{\theta}\left(\frac{R}{v}\tilde{\theta}-\frac{p^R}{v}\varphi\right)+p^Ru_x^R\frac{\tilde{\theta}^2}{\theta\theta^R}\\
				&:=u^R_{x}\left[-\frac{R}{2v^R\theta^R}\tilde{\theta}^2+L\right],
			\end{aligned}
		\end{equation}
		where
		\begin{align}\label{3.20}
			L&=-\frac{R}{\theta v}\tilde{\theta}^2+\frac{p^R}{\theta v}\varphi\tilde{\theta}+p^R\frac{\tilde{\theta}^2}{\theta\theta^R}
			=\frac{R}{v^R v}\tilde{\theta}\varphi \nonumber\\
			&=\frac{R}{v^R}\left[(\frac{1}{v}-\frac{1}{v^R})+\frac{1}{v^R}\right]\tilde{\theta}\varphi
			=-\frac{R}{v^Rvv^R}\tilde{\theta}\varphi^2+\frac{p^R}{\theta^Rv^R}\tilde{\theta}\varphi.
		\end{align}
		Substituting equation \eqref{3.20} into \eqref{3.19}, we obtain
		\begin{equation}
			\begin{aligned}\label{3.21}
				J_2&=u^R_{x}\left[-\frac{R}{2v^R\theta^R}(\theta-\theta^R)+\frac{p^R}{\theta^Rv^R}(v-v^R)\right](\theta-\theta^R)-\frac{Ru^R_{x}}{v^Rvv^R}(\theta-\theta^R)(v-v^R)^2\\
				&\leq u^R_{x}\left[-\frac{R}{2v^R\theta^R}(\theta-\theta^R)+\frac{p^R}{\theta^Rv^R}(v-v^R)\right](\theta-\theta^R)+C|u^R_{x}||\tilde{\theta}||\varphi|^2.
			\end{aligned}
		\end{equation}
		Combining \eqref{3.21} with \eqref{3.18}, and using $u^R_{x}\sim v^R_{x}$ and the {\it Cauchy} inequality,  we deduce
		\begin{align}
			&\int_{0}^{t}\int (J_{1}+J_{2})\mathrm{d}x\mathrm{d}t\nonumber\\
			&\leq \int_{0}^{t}\int-\frac{Ru^R_{x}}{2(v^R)^2}\left[\frac{\gamma \theta^R}{v^R}\varphi^2+\frac{\theta^R}{v^R}\varphi^2-2\varphi\tilde{\theta}+\frac{v^R}{\theta^R}\tilde{\theta}^2\right]\mathrm{d}x\mathrm{d}t+\int_{0}^{t}\int C|u^R_{x}||\varphi|^2 (|\varphi|+|\tilde{\theta}|)\mathrm{d}x\mathrm{d}t \nonumber\\
			&\leq \int_{0}^{t}\int-\frac{Ru^R_{x}}{2(v^R)^2}\left[\frac{(\gamma-1)\theta^R}{v^R}\varphi^2+{\frac{v^R}{2\theta^R}}\tilde{\theta}^2\right]\mathrm{d}x\mathrm{d}t+C E^\frac{1}{2}(t)\int_{0}^{t}\mathcal{D}(t)\mathrm{d}t \nonumber\\
			&\leq -C\int_{0}^{t}\int |v^R_{x}||(\varphi,\tilde{\theta})|^2\mathrm{d}x\mathrm{d}t+C E^\frac{1}{2}(t)\int_{0}^{t}\mathcal{D}(t)\mathrm{d}t.
		\end{align}
		Combining the above estimates and equation \eqref{3.16}, we obtain
		\begin{equation}
			\begin{aligned}
				\frac{\partial}{\partial t}\int\eta \mathrm{d}x\leq\sum\limits_{i=1}^{5}{B_{i}}+S_{1}+S_{2}-CG^R-D+C E^\frac{1}{2}(t)\int_{0}^{t}\mathcal{D}(t)\mathrm{d}t,
			\end{aligned}
		\end{equation}
		where
		\begin{align*}
			&B_{1}=\int \left(\frac{\tilde{\theta}}{\theta}\frac{S^2}{\mu}v+\frac{v\tilde{\theta}}{\kappa\theta^2}\left(\tilde{q}^2+2q^R\tilde{q}+(q^R)^2\right)\right)\mathrm{d}x,\\
			&B_{2}=\int \left(\frac{\tau_1\theta_{t}}{2\kappa\theta^3}\tilde{\theta}\left(\tilde{q}^2+2\tilde{q}q^R+(q^R)^2\right)-\frac{\tau_{1}}{2\kappa}\tilde{q}^2\frac{\theta_{t}}{\theta^{2}}+\frac{\theta_x}{\theta^2}\tilde{\theta}q^R\right)\mathrm{d}x,\\
			&B_{3}=\int\left(-\frac{\tau_1}{\kappa\theta}q_{t}^{R}\tilde{q}-\frac{\tau_2S_t^R}{\mu}\tilde{S}\right)\mathrm{d}x,
			\quad B_{4}=\int\left(-\frac{q^R}{\kappa\theta}\varphi\tilde{q}-\frac{S^R}{\mu}\varphi\tilde{S}\right)\mathrm{d}x,\\
			&S_{1}=\int-\psi Q^R_1\mathrm{d}x,\quad S_{2}=\int-\frac{\tilde{\theta}}{\theta}Q_2^R\mathrm{d}x,
			\quad
			G^R=\int|v^R_{x}||(\varphi,\tilde{\theta})|^2\mathrm{d}x,
			\quad D=\int\left(\frac{v}{\kappa\theta}\tilde{q}^2+\frac{v}{\mu}\tilde{S}^2\right)\mathrm{d}x.
		\end{align*}
		The term $D$ is clearly dissipative. And the term $G^R$ offers an additional damping mechanism weighted by the gradient of the reference solution. In the following, we estimate each of the remaining terms $B_i(i=1,\cdots,5), S_1$ and $S_2$ in sequence.
		
		$\boldsymbol\bullet$\textbf{The estimate of $B_{1}$:}
		First, we have
		\begin{equation}
			\begin{aligned}
				B_{1}=\int \frac{\tilde{\theta}}{\theta}\frac{S^2}{\mu}v+\frac{v\tilde{\theta}}{\kappa\theta^2}\left(\tilde{q}^2+2q^R\tilde{q}+(q^R)^2\right)\mathrm{d}x
				:=B_{11}+B_{12}.
			\end{aligned}
		\end{equation}
		By applying equation $(\ref{3.3})_{5}$, we obtain
		\begin{equation*}
			\begin{aligned}
				B_{11}&=\int \frac{\tilde{\theta}}{\theta}\frac{S^2}{\mu}v\mathrm{d}x
				=\int \frac{\tilde{\theta}v}{\theta\mu}\left[(S-S^R)^2+2(S-S^R)S^R+(S^R)^2\right]\mathrm{d}x\\
				&\leq
				\int\left(|\tilde{\theta}||\tilde{S}|^2+|\tilde{\theta}|\frac{\mu^2|u^R_{x}|^2}{|v^R|^2}\right)\mathrm{d}x
				\leq C\int|\tilde{\theta}|\left(|\tilde{S}|^2+|u^R_{x}|^2\right)\mathrm{d}x.
			\end{aligned}
		\end{equation*}
		From equation $(\ref{3.3})_{4}$, we know $q^R=-\frac{\kappa\theta^R_{x}}{v^R}$, thus we have
		\begin{align*}
			B_{12}&=\int\frac{v\tilde{\theta}}{\kappa\theta^2}\left(\tilde{q}^2+2q^R\tilde{q}+(q^R)^2\right)\mathrm{d}x \nonumber\\
			&\leq C\int\left(\frac{v}{\kappa\theta^2}\tilde{\theta}\tilde{q}^2+\frac{v(\theta^R_{x})^2}{(v^R)^2\theta^2}\tilde{\theta}\right)\mathrm{d}x \leq C\int|\tilde{\theta}|\left(|\tilde{q}|^2+|\theta^R_{x}|^2\right)\mathrm{d}x.
		\end{align*}
		Furthermore, noting that $|\theta^R_{x}|\sim|u^R_{x}|\sim|v^R_{x}|$ and applying Young's inequality with suitable exponents, we conclude
		\begin{align}\label{3.33}
			B_{1}&\leq C\int|\tilde{\theta}|(|\tilde{S}|^2+|\tilde{q}|^2)+|\tilde{\theta}|\left(|u^R_{x}|^2+|\theta^R_{x}|^2\right)\mathrm{d}x \nonumber\\
			&\leq C E^\frac{1}{2}(t) \mathcal{D}(t)+\epsilon\int|\tilde{\theta}|^2\mathrm{d}x
			+C(\epsilon)\int\left(|u^R_{x}|^4+|\theta^R_{x}|^4\right)\mathrm{d}x \nonumber\\
			&\leq \epsilon\int|\tilde{\theta}|^2\mathrm{d}x+C E^\frac{1}{2}(t)\mathcal{D}(t)+C\delta_{R}^{\frac{1}{2}}.			
		\end{align}
		$\boldsymbol\bullet$\textbf{The estimate of $B_{2}$:}
		We observe that $\|\theta_{t}\|_{L^\infty}\leq\|(\theta-\theta^R)_{t}\|_{L^\infty}+\|\theta^R_{t}\|_{L^\infty}\leq C(\epsilon_{1}+\delta_{R})$, thus
		\begin{equation}
			\begin{aligned}\label{3.34}
				B_{2}&=\int \left(\frac{\tau_1\theta_{t}}{2\kappa\theta^3}\tilde{\theta}\left(\tilde{q}^2+2\tilde{q}q^R+(q^R)^2\right)-\frac{\tau_{1}}{2\kappa}\tilde{q}^2\frac{\theta_{t}}{\theta^{2}}+\frac{\theta_x}{\theta^2}\tilde{\theta}q^R\right)\mathrm{d}x\\
				&\leq C(\epsilon_{1}+\delta_R)\int\tilde{q}^2\mathrm{d}x+C(\epsilon_{1}+\delta_R)\int|\tilde{\theta}||q^R|^2\mathrm{d}x+\int\frac{\tilde{\theta}}{\theta^2}\left[(\theta-\theta^R)_x+\theta^R_x\right]\frac{\kappa\theta^R_x}{v}\mathrm{d}x+C E^\frac{1}{2}(t)\mathcal{D}(t)\\
				&\leq (\epsilon_{1}+\delta_R)\int(\tilde{q}^2)\mathrm{d}x+C(\epsilon_{1}+\delta_R)\left(\int|\theta^R_x|^3\mathrm{d}x+\int|\theta^R_x||\theta-\theta^R|^2\mathrm{d}x\right)+\int|\theta^R_x|\tilde{\theta}_x^2\mathrm{d}x\\
				&\leq (\epsilon_{1}+\delta_R)\int(\tilde{q}^2+|\theta^R_x|\tilde{\theta}^2)\mathrm{d}x+C\delta_R(1+t)^{-2}+C\delta_R\int\tilde{\theta}_x^2\mathrm{d}x.
			\end{aligned}
		\end{equation}
		$\boldsymbol\bullet$\textbf{The estimate of $B_{3}$:}
		Taking derivative with respect to $t$ to the equations $(\ref{3.7})_{2},(\ref{3.7})_{3}$, we get
		\begin{equation}
			\begin{aligned}\label{4.17}
				&\theta^R_{xt}=-(\gamma-1)\left(\frac{\theta^R_{x}u^R_{x}}{v^R}+\frac{\theta^Ru^R_{xx}}{v^R}-\frac{\theta^Ru^R_{x}v^R_{x}}{(v^R)^2}\right),\\
				&u^R_{xt}=-p^R_{xx}=-\left(\frac{R}{(v^R)^2}(\theta^R_{xx}v^R-\theta^Rv^R_{xx})-\frac{2R}{(v^R)^3}(\theta^R_{x}v^Rv^R_{x}-\theta^R(v^R_{x})^2)\right).
			\end{aligned}
		\end{equation}
		Using \eqref{4.17}, we estimate the two components of $B_3$. First,
		\begin{align*}
			\left|-\frac{\tau_{1}}{\kappa\theta}q^R_{t}\tilde{q}\right|&=\left|-\frac{\tau_{1}}{\kappa\theta}\cdot\left[-(\frac{\kappa\theta^R_{xt}}{v^R}-\frac{\kappa\theta^R_{x}v^R_{t}}{(v^R)^2})\right]\tilde{q}\right|\\
			&=\left|-\frac{\tau_{1}}{\theta (v^R)^2}\tilde{q}\left[(\gamma-1)(\theta^R_{x}u^R_{x}+\theta^Ru^R_{xx}-\frac{\theta^Ru^R_{x}v^R_{x}}{v^R})-u^R_{x}\theta^R_{x}\right]\right|\\
			&\leq C|\tilde{q}|(\left|u^R_{x}\theta^R_{x}\right|+\left|u^R_{xx}\right|+\left|u^R_{x}v^R_{x}\right|).
		\end{align*}
		Second, since $S^R = \frac{\mu u_x^R}{v^R},$
		\begin{align*}
			\left|-\frac{\tau_{2}}{\mu}S^R_{t}\tilde{S}\right|&=\left|-\frac{\tau_{2}}{\mu}\cdot\mu\frac{u^R_{xt}v^R-u^R_{x}v^R_{t}}{(v^R)^2}\tilde{S}\right|
			=\left|\tau_{2}\tilde{S}\left(\frac{p^R_{xx}}{v^R}+\frac{(u^R_{x})^2}{(v^R)^2}\right)\right|\\
			&\leq C|\tilde{S}|\left(\left|\theta^R_{xx}\right|+\left|v^R_{xx}\right|+\left|\theta^R_{x}v^R_{x}\right|+\left|(v^R_{x})^2\right|+\left|(u^R_{x})^2\right|\right).
		\end{align*}
		Combining the above estimates and applying H\"{o}lder inequalities, and then using Young's inequality to separate the mixed products (e.g., $ | u^R_x \theta^R_x |_{L^2}^2 \le | u^R_x |_{L^4}^2 | \theta^R_x |_{L^4}^2 \le C(| u^R_x |_{L^4}^4 + | \theta^R_x |_{L^4}^4) $), we obtain 
		\begin{align}\label{3.38}
			B_{3}&\leq C\int|\tilde{q}|(\left|u^R_{x}\theta^R_{x}\right|+\left|u^R_{xx}\right|+\left|u^R_{x}v^R_{x}\right|)\mathrm{d}x \nonumber\\
			&\quad+C\int|\tilde{S}|(\left|\theta^R_{xx}\right|+\left|v^R_{xx}\right|+\left|\theta^R_{x}v^R_{x}\right|+\left|(v^R_{x})^2\right|+\left|(u^R_{x})^2\right|)\mathrm{d}x \nonumber\\
			&\leq\epsilon\int\left(|\tilde{q}|^2+|\tilde{S}|^2\right)\mathrm{d}x+C(\epsilon)\left(\|(u^R_{xx},v^R_{xx},\theta^R_{xx})\|^2_{L^2}+\|(u^R_{x},v^R_{x},\theta^R_{x})\|^4_{L^4}\right) \nonumber\\
			&\leq\epsilon\int\left(|\tilde{q}|^2+|\tilde{S}|^2\right)\mathrm{d}x+C(\epsilon)\left(\|(u^R_{xx},v^R_{xx},\theta^R_{xx})\|^2_{L^2}+\|(u^R_{x},v^R_{x},\theta^R_{x})\|^4_{L^4}\right) \nonumber\\
			&\leq\epsilon\int\left(|\tilde{q}|^2+|\tilde{S}|^2\right)\mathrm{d}x+ C(\epsilon)\left(\delta_{R}^{\frac{4}{3}-\frac{2}{q}}(1+t)^{-2-\frac{1}{2q}}+\delta_{R}(1+t)^{-3}\right).
		\end{align}
		$\boldsymbol\bullet$\textbf{The estimate of $B_{4}$:} We have
		\begin{equation}\label{3.39}
			\begin{aligned}
				B_{4}&\leq\int\left|\frac{q^R}{\kappa\theta}\varphi\tilde{q}+\frac{S^R}{\mu}\varphi\tilde{S}\right|\mathrm{d}x
				=\int\left|\frac{\theta^R_{x}}{\kappa\theta v^R}\varphi\tilde{q}+\frac{u^R_{x}}{v^R}\varphi\tilde{S}\right|\mathrm{d}x\\
				&\leq\epsilon\int\left(|\tilde{q}|^2+|\tilde{S}|^2\right)\mathrm{d}x+C(\epsilon)\int\left(|\theta^R_{x}||\varphi|^2+|u^R_{x}||\varphi|^2\right)\mathrm{d}x\\
				&\leq\epsilon D+C(\epsilon) G^R,
			\end{aligned}
		\end{equation}
		$\boldsymbol\bullet$\textbf{The estimate of $S_{1}$, $S_{2}$:}
		First, we have
		\begin{equation}\label{3.41}
			\begin{aligned}
				S_{1}&\leq\int\left|\mu\psi\frac{u^R_{xx}v^R-u^R_{x}v^R_{x}}{(v^R)^2}\right|\mathrm{d}x\\
				&\leq C\| \psi\|_{L^2}(\| u^R_{xx}\|_{L^2}+\| v^R_{x}\|_{L^4}\| u^R_{x}\|_{L^4})\\
				&\leq C\epsilon_{1}(\delta_{R}^{\frac{2}{3}-\frac{1}{q}}(1+t)^{-1-\frac{1}{4q}}+\delta_{R}^{\frac{1}{2}}(1+t)^{-\frac{3}{2}}).
			\end{aligned}
		\end{equation}
		Similarly,
		\begin{equation}\label{3.42}
			\begin{aligned}
				S_{2}\leq C\epsilon_{1}(\delta_{R}^{\frac{2}{3}-\frac{1}{q}}(1+t)^{-1-\frac{1}{4q}}+\delta_{R}^{\frac{1}{2}}(1+t)^{-\frac{3}{2}}).
			\end{aligned}
		\end{equation}
		Therefore, combining the above estimates and integrating over $[0,t]$, we obtain
		\begin{equation}
			\begin{aligned}
				&\int\left(\frac{R}{\gamma-1}\theta^{R}\phi(\frac{\theta}{\theta^{R}})+R\theta^{R}\phi(\frac{v}{v^{R}})+\frac{1}{2}\psi^2+\frac{\tau_1}{2\kappa\theta}\tilde{q}^2+\frac{\tau_2}{2\mu}\tilde{S}^2\right)\mathrm{d}x
				+\int_{0}^{t}\int\left(\frac{v}{\kappa\theta}\tilde{q}^2+\frac{v}{\mu}\tilde{S}^2\right)\mathrm{d}x\mathrm{d}t\\&
				+\int_{0}^{t}\int|v^R_{x}||(\varphi,\tilde{\theta})|^2\mathrm{d}x\mathrm{d}t
				\leq C\left(I_{0}+\delta_R\int_{0}^{t}\int\tilde{\theta}_x^2\mathrm{d}x+\delta_{R}^\frac{1}{2}
				+E^{\frac{1}{2}}(t)\int_{0}^{t}\mathcal{D}(t)\mathrm{d}t\right)
			\end{aligned}
		\end{equation}
		where $$I_{0}=\int\left(\frac{R}{\gamma-1}\theta_{0}^{R}\phi(\frac{\theta_{0}}{\theta_{0}^{R}})+R\theta_{0}^{R}\phi(\frac{v_{0}}{v_{0}^{R}})+\frac{(u_{0}-u_{0}^R)^2}{2}+\frac{\tau_1}{2\kappa\theta_{0}}(q_{0}-q_{0}^R)^2+\frac{\tau_2}{2\mu}(S_{0}-S_{0}^R)^2\right)\mathrm{d}x.$$
		By Taylor expansion, one has
		\begin{align*}
			&\phi(\frac{\theta}{\theta^{R}}):=\frac{\theta}{\theta^{R}}-1-\ln\frac{\theta}{\theta^{R}}=\frac{1}{2\xi^2}(\theta-\theta^R)^2,\\
			&\phi(\frac{v}{v^{R}}):=\frac{v}{v^{R}}-1-\ln \frac{v}{v^{R}}=\frac{1}{2\eta^2}(v-v^R)^2,
		\end{align*}
		where $\xi\in(\theta^R,\theta)$, $\eta\in(v^R,v)$.
		Hence, Lemma \ref{lem4.1} follows immediately.
	\end{proof}
	\subsection{High-order estimate}
	We can rewrite equation \eqref{3.10} as
	\begin{equation}
		\begin{cases}\label{3.48}
			\varphi_t - \psi_x = 0 \\
			\psi_t + \frac{R}{v}(\tilde{\theta}_x+\theta^R_x)-\frac{R\theta}{v^2}(\varphi_x+v^R_x)-\frac{R\theta^R_x}{v^R}+\frac{R\theta^Rv^R_x}{(v^R)^2}= \tilde{S}_x-Q_{1}^{R}\\
			\frac{R}{\gamma-1}\tilde{\theta}_{t}-\frac{\tau_{1}(\tilde{q}+q^R)^2\tilde{\theta}_{t}}{2\kappa(\tilde{\theta}+\theta^R)^2}
			-\frac{(\tilde{\theta}+\theta^R)_x(\tilde{q}+q^R)}{\tilde{\theta}+\theta^R}+p\psi_{x}+
			(p-p^{R})u_x^{R}+\tilde{q}_x\\
			\qquad\qquad\qquad=\frac{(\tilde{q}+q^R)^2v}{\kappa(\tilde{\theta}+\theta^R)}+\frac{v}{\mu} (\tilde{S}^2+2S^R\tilde{S}+(S^R)^2)-Q_{2}^{R}+\frac{\tau_{1}(\tilde{q}+q^R)^2\theta^R_t}{2\kappa(\tilde{\theta}+\theta^R)^2}\\
			\tau_1\tilde{q}_t + vq-v^{R}q^{R} + \kappa\tilde{\theta}_x = -	\tau_1q^R_t \\
			\tau_2 \tilde{S}_t + vS-v^{R}S^{R} = \mu \psi_x-\tau_2 S^R_t
		\end{cases}
	\end{equation}
	with initial condition
	\begin{align*}
		(\varphi,\psi,\tilde{\theta},\tilde{q},\tilde{S})(x,0)=(\varphi_{0},\psi_{0},\tilde{\theta}_{0},\tilde{q}_{0},\tilde{S}_{0})(x),\quad x\in\mathbb{R},
	\end{align*}
	where
	\begin{equation*}
		\begin{aligned}
			\varphi_{0}=v_{0}-v^R_{0},\quad\psi_{0}=u_{0}-u^R_{0},\quad\tilde{\theta}_{0}=\theta_{0}-\theta^R_{0},\quad\tilde{q}_{0}=q_{0}-q^R_{0},\quad\tilde{S}_{0}=S_{0}-S^R_{0}.
		\end{aligned}
	\end{equation*}	
	\begin{lemma}\label{lem4.2}
		There exists a constant $C$ such that for $0\le t\le T$, we have
		\begin{equation}
			\begin{aligned}
				&\|(\varphi_{x},\psi_{x},\tilde{\theta}_{x},\tilde{q}_{x},\tilde{S}_{x})\|^2_{L^2}+\int_{0}^{t}(\|\tilde{q}_{x}\|^2_{L^2}+\|\tilde{S}_{x}\|^2_{L^2})\mathrm{d}t\\
				&\leq C\left(E_0+E^{\frac{1}{2}}(t)\int_{0}^{t}\mathcal{D}(t)\mathrm{d}t+\delta_{R}\int_{0}^{t}\mathcal{D}(t)\mathrm{d}t+\delta_{R}^{\frac{3}{2}-\frac{1}{q}}+\delta_{R}^\frac{1}{6}\right).
			\end{aligned}
		\end{equation}
	\end{lemma}
	\begin{proof}
		Taking derivatives with respect to $x$ to the equations $(\ref{3.48})$, we  get
		\begin{equation}\label{3.51}
			\begin{cases}
				\varphi_{tx}-\psi_{xx}=0,\\
				\psi_{tx}-\left(\frac{R\theta_x}{v^2}-\frac{2R\theta v_x}{v^3}\right)(\varphi_x+v^R_x)-\frac{R\theta}{v^2}(\varphi_{xx}+v^R_{xx})-\frac{Rv_x}{v^2}(\tilde{\theta}_x+\theta^R_x)+\frac{R}{v}(\tilde{\theta}_{xx}+\theta^R_{xx})\\
				\qquad\qquad\qquad+\frac{2R\theta_x^Rv_x^R+R\theta^Rv^R_{xx}}{(v^R)^2}-\frac{2R\theta^R(v_x^R)^2}{(v^R)^3}-\frac{R\theta_{xx}^R}{v^R}
				=\tilde{S}_{xx}-(Q^R_{1})_{x},\\
				\frac{R}{\gamma-1}\tilde{\theta}_{tx}-\left[\frac{\tau_{1}(\tilde{q}+q^R)^2}{2\kappa(\tilde{\theta}+\theta^R)^2}\tilde{\theta_{t}}\right]_{x}-\left[\frac{(\tilde{\theta}+\theta^R)_{x}(\tilde{q}+q^R)}{\tilde{\theta}+\theta^R}\right]_{x}+
				\left(-\frac{R\theta}{v^2}(\varphi_x+v^R_x)+\frac{R}{v}(\tilde{\theta}_x+\theta^R_x)\right)\psi_{x}\\
				\qquad\qquad\qquad+\frac{R\theta}{v}\psi_{xx}+(p-p^R)_{x}u^R_x+(p-p^R)u^R_{xx}+\tilde{q}_{xx}\\
				\qquad\qquad\qquad=\left[\frac{(\tilde{q}+q^R )^2v}{\kappa(\tilde{\theta}+\theta^R)}\right]_{x}+\left[\frac{v}{\mu}(\tilde{S}^2+2S^R\tilde{S}+(S^R)^2)\right]_{x}-(Q^R_{2})_{x}+\left[\frac{\tau_{1}(\tilde{q}+q^R)^2\theta^R_{t}}{2\kappa(\tilde{\theta}+\theta^R)^2}\right]_{x},\\
				\tau_{1}\tilde{q}_{tx}+(v\tilde{q}+q^R\varphi)_{x}+\kappa\tilde{\theta}_{xx}=-\tau_{1}q^R_{tx},\\
				\tau_{2}\tilde{S}_{tx}+(v\tilde{S}+S^R\varphi)_{x}-\mu\psi_{xx}=-\tau_{2}S^R_{tx}.
			\end{cases}
		\end{equation}
		Multiplying $(\ref{3.51})_{1,2,3,4,5}$ by $\varphi_{x}$, $\frac{v^2}{R\theta}\psi_{x}$, $\frac{v^2}{R\theta^2}\tilde{\theta}_{x}$, $\frac{v^2}{\kappa R\theta^2}\tilde{q}_{x}$, $\frac{v^2}{\mu R\theta}\tilde{S}_{x}$, respectively, and integrating the results over $[0,t]\times\mathbb{R}$, we can obtain
		\begin{equation}\label{3.52}
			\begin{aligned}
				&\int\left(\frac{1}{2}\varphi_{x}^2+\frac{v^2}{2R\theta}\psi_{x}^2+\frac{v^2}{2(\gamma-1)\theta^2}\tilde{\theta}^2_{x}+\frac{\tau_{1}v^2}{2\kappa R\theta^2}\tilde{q}^2_{x}+\frac{\tau_{2}v^2}{2\mu R\theta}\tilde{S}^2_{x}\right)\mathrm{d}x\bigg|_{0}^{t}+\\
				&\int_{0}^{t}\int\frac{v^3}{\kappa R\theta^2}\tilde{q}^2_{x}\mathrm{d}x\mathrm{d}t+
				\int_{0}^{t}\int\frac{v^3}{\mu R\theta}\tilde{S}^2_{x} \mathrm{d}x\mathrm{d}t=:\sum\limits_{j=1}^{17}\int_{0}^{t}R_{j}\mathrm{d}t,
			\end{aligned}
		\end{equation}
		where
		\begin{align*}
			&R_{1}=\int\left(\frac{v}{R\theta}\left(v_t-\frac{v\theta_t}{2\theta}\right)\psi_{x}^2+\frac{v}{(\gamma-1)\theta^2}\left(v_t-\frac{v\theta_t}{\theta}\right)\tilde{\theta}_{x}^2\right)\mathrm{d}x,\\
			&R_{2}=\int \bigg[\left(\frac{\theta_x}{\theta^2}-\frac{2v_x}{v\theta}\right)(\varphi_{x}+v_{x})\psi_{x}+\frac{v_x}{\theta}\psi_x(\tilde{\theta}_{x}+\theta^R_{x})-\frac{v}{\theta}\theta_{xx}^R\psi_{x}\\
			&\quad\quad+\frac{(2\theta^R_x v^R_x+\theta^Rv^R_xx)v^2}{\theta(v^R)^2}\psi_{x}+\frac{2\theta^R(v_x^R)^2v^2}{\theta (v^R)^3}\psi_{x}+\frac{v^2\theta_{xx}^R}{\theta v^R}\psi_{x}\bigg] \mathrm{d}x,\\
			&R_{3}=-\int (Q^R_{1})_{x}\frac{v^2}{R\theta}\psi_{x}\mathrm{d}x,\quad R_{4}=\int\left[\frac{\tau_{1}(\tilde{q}+q^R)^2}{2\kappa(\tilde{\theta}+\theta^R)^2}\tilde{\theta_{t}}\right]_{x}\frac{v^2}{R\theta^2}\tilde{\theta}_{x}\mathrm{d}x,\\ &R_{5}=\int\left[\frac{(\tilde{\theta}+\theta^R)_{x}(\tilde{q}+q^R)}{\tilde{\theta}+\theta^R}\right]_{x}\frac{v^2}{R\theta^2}\tilde{\theta}_{x}\mathrm{d}x,\\
			&R_{6}=\int \left[\frac{1}{\theta}(\varphi_x+v^R_x)-\frac{v}{\theta^2}(\tilde{\theta}_{x}+\theta^R_x)\right]\psi_{x}\tilde{\theta}_{x}\mathrm{d}x,\\
			&R_{7}=-\int(p-p^R)_{x}u^R_x \frac{v^2}{R\theta^2}\tilde{\theta}_{x}\mathrm{d}x,\quad R_{8}=\int(p-p^R)u^R_{xx}\frac{v^2}{R\theta^2}\tilde{\theta}_{x}\mathrm{d}x,\\
			&R_{9}=\int\left[\frac{\tau_{1}(\tilde{q}+q^R )^2v}{\kappa(\tilde{\theta}+\theta^R)}\right]_{x} \frac{v^2}{R\theta^2}\tilde{\theta}_{x}\mathrm{d}x,\quad
			R_{10}=\int \left[\frac{v}{\mu}(\tilde{S}^2+2S^R\tilde{S}+(S^R)^2)\right]_{x} \frac{v^2}{R\theta^2}\tilde{\theta}_{x}\mathrm{d}x,\\
			&R_{11}=-\int(Q^R_{2})_{x}\frac{v^2}{R\theta^2}\tilde{\theta}_{x}\mathrm{d}x,\quad
			R_{12}=\int\left[\frac{\tau_{1}(\tilde{q}+q^R)^2\theta^R_{t}}{2\kappa(\tilde{\theta}+\theta^R)^2}\right]_{x}\frac{v^2}{R\theta^2}\tilde{\theta}_{x}\mathrm{d}x,\\
			&R_{13}=-\int(v_{x}\tilde{q}+q^R_{x}\varphi+q^R\varphi_{x})\frac{v^2}{\kappa R\theta^2}\tilde{q}_{x}\mathrm{d}x-\int\frac{\tau_1v}{\kappa R\theta^2}\left(v_t-\frac{v\theta_t}{\theta}\right)\tilde{q}_{x}^2\mathrm{d}x,\\
			&R_{14}=-\int(v_{x}\tilde{S}+S^R_{x}\varphi+S^R\varphi_{x})\frac{v^2}{\mu R\theta}\tilde{S}_{x}\mathrm{d}x-\int\frac{\tau_2v}{\mu R\theta}\left(v_t-\frac{v\theta_t}{2\theta}\right)\tilde{S}_{x}^2\mathrm{d}x,\\
			&R_{15}=-\int\tau_{1}q^R_{tx}\frac{v^2}{\kappa R\theta^2}\tilde{q}_{x}\mathrm{d}x,\quad
			R_{16}=-\int\tau_{2}S^R_{tx}\frac{v^2}{\mu R\theta}\tilde{S}_{x}\mathrm{d}x,\\
			&R_{17}=\int\left(\frac{v_x}{\theta}-\frac{v\theta_x}{\theta^2}\right)\tilde{\theta}_{x}\psi_{x}-\frac{v}{\theta R}\left(2v_x-\frac{v\theta_x}{\theta}\right)\psi_{x}\tilde{S}_x+\frac{2v}{R\theta^2}\left(v_x-\frac{v\theta_x}{\theta}\right)\tilde{q_{x}}\tilde{\theta}_{x}\mathrm{d}x.
		\end{align*}
		We now proceed to estimate each $R_i$ ($i=1, 2, \cdots, 17$) in turn. First, we have
		\begin{align}\label{3.54}
			\int_{0}^{t}R_{1}\mathrm{d}t&=\int_{0}^{t}\int \left(\frac{v}{R\theta}\left(v_t-\frac{v\theta_t}{2\theta}\right)\psi_{x}^2+\frac{v}{(\gamma-1)\theta^2}\left(v_t-\frac{v\theta_t}{\theta}\right)\tilde{\theta}_{x}^2\right)\mathrm{d}x\mathrm{d}t \nonumber\\
			&\leq C\int_{0}^{t}\int(|\tilde{\theta}_{t}|+|\psi_{x}|+|u^R_{x}|)(\psi_{x}^2+\tilde{\theta}_{x}^2)\mathrm{d}x\mathrm{d}t \nonumber\\
			&\leq CE^{\frac{1}{2}}(t)\int_{0}^{t}\mathcal{D}(t)\mathrm{d}t+C\delta_{R}\int_{0}^{t}\mathcal{D}(t)\mathrm{d}t,
		\end{align}and
		\begin{align}
			\int_{0}^{t}R_{2}\,\mathrm{d}t
			&=\int_{0}^{t}\int \bigg[\left(\frac{\theta_x}{\theta^2}-\frac{2v_x}{v\theta}\right)(\varphi_{x}+v_{x})\psi_{x}+\frac{v_x}{\theta}\psi_x(\tilde{\theta}_{x}+\theta^R_{x})-\frac{v}{\theta}\theta_{xx}^R\psi_{x} \nonumber\\
			&\quad+\frac{(2\theta^R_x v^R_x+\theta^Rv^R_{xx})v^2}{\theta(v^R)^2}\psi_{x}+\frac{2\theta^R(v_x^R)^2v^2}{\theta (v^R)^3}\psi_{x}+\frac{v^2\theta_{xx}^R}{\theta v^R}\psi_{x}\bigg] \,\mathrm{d}x\mathrm{d}t \nonumber\\
			&\leq C\int_{0}^{t} \bigg( \bigl\| \bigl( v^R_{x}, \theta^R_{x}, \tilde{\theta}_{x}, \varphi_{x} \bigr) \bigr\|_{L^\infty}
			\cdot \bigl\| \bigl( \varphi_{x}, \psi_{x}, \tilde{\theta}_{x} \bigr) \bigr\|_{L^2}^2 \nonumber\\
			&\qquad+\left(\| v^R_{x}\|_{L^4}^2+\| \theta^R_{x}\|_{L^4}\| v^R_{x}\|_{L^4}+\| v^R_{xx}\|_{L^2}+\| \theta^R_{xx}\|_{L^2}\right)\| \psi_{x}\|_{L^2}\bigg)\mathrm{d}t \nonumber\\
			&\leq C\left(E^{\frac{1}{2}}(t)\int_{0}^{t}\mathcal{D}(t)\mathrm{d}t+\delta_{R}\int_{0}^{t}\mathcal{D}(t)\mathrm{d}t+\delta_{R}^{\frac{1}{2}}+\delta_{R}^{\frac{3}{2}-\frac{1}{q}}\right),
		\end{align}
		where use Lemma \ref{lem2.1}.	And similarly,
		\begin{equation}
			\begin{aligned}
				\int_{0}^{t}R_{3}\mathrm{d}t&=\int_{0}^{t}\int \mu(\frac{u^R_{x}}{v^R})_{xx}\frac{v^2}{R\theta}\psi_{x}\mathrm{d}x\mathrm{d}t\\
				&\leq C\int_{0}^{t}\int\left(\frac{u^R_{xxx}}{v^R}-\frac{2u^R_{xx}v^R_{x}}{(v^R)^2}-\frac{u^R_{x}v^R_{xx}}{(v^R)^2}+\frac{2u^R_{x}(v^R)^2}{(v^R)^3}\right)\psi_{x}\mathrm{d}x\mathrm{d}t\\
				&\leq C(\delta_{R}^{\frac{3}{2}-\frac{1}{q}}+\delta_{R}^{\frac{1}{2}}).
			\end{aligned}
		\end{equation}
		In view of the bounds $\|\theta_{t}\|_{L^\infty}\leq C(\epsilon_{1}+\delta_{R})$, $\|\tilde{q}_{t}\|_{L^\infty}\leq C(\epsilon_{1}+\delta_{R})$, the term $R_4$ can be estimated as follows
		\begin{equation*}
			\begin{aligned}
				R_{4}&= \int\left[\frac{\tau_{1}(\tilde{q}+q^R)^2}{2\kappa(\tilde{\theta}+\theta^R)^2}\tilde{\theta_{t}}\right]_{x}\frac{v^2}{R\theta^2}\tilde{\theta}_{x}\mathrm{d}x\\
				&\leq C\int \left[\frac{2(\tilde{q}+q^R)(\tilde{q}+q^R)_{x}\tilde{\theta}_{t}\tilde{\theta}_{x}}{(\tilde{\theta}+\tilde{\theta}^R)^2}+\frac{(\tilde{q}+q^R)^2\tilde{\theta}_{tx}\tilde{\theta}_{x}}{(\tilde{\theta}+\tilde{\theta}^R)^2}-\frac{2(\tilde{q}+q^R)^2\tilde{\theta}^2_{x}\tilde{\theta}_{t}}{(\tilde{\theta}+\tilde{\theta}^R)^3}-\frac{2(\tilde{q}+q^R)^2\tilde{\theta}_{t}\theta^R_{x}\tilde{\theta}_{x}}{(\tilde{\theta}+\tilde{\theta}^R)^3}\right]\mathrm{d}x,
			\end{aligned}
		\end{equation*}
		where
		\begin{equation*}
			\begin{aligned}
				\int\frac{(\tilde{q}+q^R)^2\tilde{\theta}_{tx}\tilde{\theta}_{x}}{(\tilde{\theta}+\tilde{\theta}^R)^2}\mathrm{d}x&=\int\frac{\mathrm{d}}{\mathrm{d}t}\frac{(\tilde{q}+q^R)^2\tilde{\theta}^2_{x}}{2(\tilde{\theta}+\theta^R)^2}\mathrm{d}x-\int\frac{(\tilde{q}+q^R)(\tilde{q}+q^R)_{t}\tilde{\theta}^2_{x}}{(\tilde{\theta}+\theta^R)^2}-\frac{(\tilde{q}+q^R)^2(\tilde{\theta}+\theta^R)_{t}\tilde{\theta}^2_{x}}{(\tilde{\theta}+\theta^R)^3}\mathrm{d}x\\
				&\leq \frac{\mathrm{d}}{\mathrm{d}t}\int\frac{(\tilde{q}+q^R)^2\tilde{\theta}^2_{x}}{2(\tilde{\theta}+\theta^R)^2}\mathrm{d}x+\int\bigg|\frac{(\tilde{q}+q^R)(\tilde{q}+q^R)_{t}\tilde{\theta}^2_{x}}{(\tilde{\theta}+\theta^R)^2}+\frac{(\tilde{q}+q^R)^2(\tilde{\theta}+\theta^R)_{t}\tilde{\theta}^2_{x}}{(\tilde{\theta}+\theta^R)^3}\bigg|\mathrm{d}x\\
				&\leq \frac{\mathrm{d}}{\mathrm{d}t}\int\frac{(\tilde{q}+q^R)^2\tilde{\theta}^2_{x}}{2(\tilde{\theta}+\theta^R)^2}\mathrm{d}x+C(\epsilon_{1}+\delta_R)\| \tilde{\theta}_{x}\|_{L^2}^2.
			\end{aligned}
		\end{equation*}
		Combining the above estimates and Lemma \ref{lem2.1}, we can derive
		\begin{equation}
			\begin{aligned}
				\int_{0}^{t}R_{4}\mathrm{d}t&\leq\int\frac{\tau_{1}(\tilde{q}+q^R)^2\tilde{\theta}^2_{x}}{2\kappa(\tilde{\theta}+\theta^R)^2}\bigg|_{0}^{t}\mathrm{d}x+C(\epsilon_{1}+\delta^R)\int_{0}^{t}\| \tilde{q}_{x}\|_{L^2}^2+\| \tilde{\theta}_{x}\|_{L^2}^2\mathrm{d}t\\
				&\quad+C\| \tilde{\theta}_{x}\|_{L^\infty}\| \tilde{\theta}_{t}\|_{L^\infty}\|\theta^R_{x}\|_{L^\infty}\int_{0}^{t}\|\tilde{q}\|_{L^2}^2\mathrm{d}x\mathrm{d}t+\int_{0}^{t}(\|\theta^R_{xx}\|_{L^2}+\|\theta^R_x\|_{L^4}\| v^R_x\|_{L^4}\\
				&\quad+\|\theta^R_x\|_{L^4}\| \theta^R_{xx}\|_{L^4}+\|\theta^R_x\|_{L^6}^2\| v^R_x\|_{L^6}+\|\theta^R_x\|^2_{L^4}+\|\theta^R_x\|^3_{L^6})\|\tilde{\theta}_x\|_{L^2}\mathrm{d}t\\
				&\leq\int\frac{\tau_{1}(\tilde{q}+q^R)^2\tilde{\theta}^2_{x}}{2\kappa(\tilde{\theta}+\theta^R)^2}\bigg|_{0}^{t}\mathrm{d}x+C\left((\epsilon_{1}+\delta_R)\int_{0}^{t}\mathcal{D}(t)\mathrm{d}t+E^{\frac{1}{2}}(t)\int_{0}^{t}\mathcal{D}(t)\mathrm{d}t+\delta_R^{\frac{3}{2}-\frac{1}{q}}+\delta_R^{\frac{1}{2}}\right).
			\end{aligned}
		\end{equation}
		Furthermore,
		\begin{align*}
			R_{5}&=\int\left[\frac{(\tilde{\theta}+\theta^R)_{x}(\tilde{q}+q^R)}{\tilde{\theta}+\theta^R}\right]_{x}\frac{v^2}{R\theta^2}\tilde{\theta}_{x}\mathrm{d}x\\
			&=\int \frac{v^2}{R\theta^2}\bigg(\frac{(\tilde{q}+q^R)}{\tilde{\theta}+\theta^R}\tilde{\theta}_{xx}\tilde{\theta}_{x}+\tilde{\theta}_{x}\frac{\theta^R_{xx}(\tilde{q}+q^R)}{\tilde{\theta}+\theta^R}+\frac{\tilde{\theta}^2_{x}(\tilde{q}+q^R)_{x}}{\tilde{\theta}+\theta^R}+\frac{\tilde{\theta}_{x}(\tilde{q}+q^R)_{x}\theta^R_{x}}{\tilde{\theta}+\theta^R}\\
			&\quad-\frac{\tilde{\theta}^3_{x}(\tilde{q}+q^R)}{(\tilde{\theta}+\theta^R)^2}-\frac{\tilde{\theta}_{x}(\tilde{q}+q^R)(\theta^R_{x})^2}{(\tilde{\theta}+\theta^R)^2}-\frac{(\tilde{\theta}_x^2\theta^R_x(\tilde{q}+q^R))}{(\tilde{\theta}+\theta^R)^2}\bigg)\mathrm{d}x\\
			&\leq C\int \bigg|\frac{(\tilde{q}+q^R)}{\tilde{\theta}+\theta^R}\tilde{\theta}_{xx}\tilde{\theta}_{x}+\tilde{\theta}_{x}\frac{\theta^R_{xx}(\tilde{q}+q^R)}{\tilde{\theta}+\theta^R}+\frac{\tilde{\theta}^2_{x}(\tilde{q}+q^R)_{x}}{\tilde{\theta}+\theta^R}+\frac{\tilde{\theta}_{x}(\tilde{q}+q^R)_{x}\theta^R_{x}}{\tilde{\theta}+\theta^R}\\
			&\quad\quad-\frac{\tilde{\theta}^3_{x}(\tilde{q}+q^R)}{(\tilde{\theta}+\theta^R)^2}-\frac{\tilde{\theta}_{x}(\tilde{q}+q^R)(\theta^R_{x})^2}{(\tilde{\theta}+\theta^R)^2}-\frac{(\tilde{\theta}_x^2\theta^R_x(\tilde{q}+q^R))}{(\tilde{\theta}+\theta^R)^2}\bigg|\mathrm{d}x,\\
		\end{align*}
		where
		\begin{align*}
			\int\frac{\tilde{q}+q^R}{\tilde{\theta}+\theta^R}\tilde{\theta}_{xx}\tilde{\theta}_{x}\mathrm{d}x&\leq\int\left( \frac{(\tilde{q}+q^R)\tilde{\theta}_{x}^2}{2(\tilde{\theta}+\theta^R)}\right)_{x}\mathrm{d}x-\int\left(\frac{\tilde{q}+q^R}{2(\tilde{\theta}+\theta^R)}\right)_{x}\tilde{\theta}^2_{x}\mathrm{d}x\\
			&\leq-\int\frac{\tilde{q}_{x}\tilde{\theta}^2_{x}+q^R_x\tilde{\theta}^2_x}{2(\tilde{\theta}+\theta^R)}-\frac{(\tilde{q}+q^R)(\tilde{\theta}+\theta^R)_{x}\tilde{\theta}^2_{x}}{2(\tilde{\theta}+\theta^R)^2}\mathrm{d}x,
		\end{align*}
		thus, we have
		\begin{align}
			\int_{0}^{t}R_{5}\mathrm{d}t&\leq C\int_{0}^{t}\int|\tilde{q}_{x}\tilde{\theta}^2_{x}|+|q^R_{x}\tilde{\theta}^2_{x}|+|(\tilde{q}+q^R)(\tilde{\theta}+\theta^R)_{x}\tilde{\theta}^2_{x}|+|\tilde{\theta}_{x}\theta^R_{xx}|+|\tilde{\theta}_{x}q^R_{x}\theta^R_{xx}|+|\tilde{\theta}^3_{x}(\tilde{q}+q^R)| \nonumber\\
			&\quad|(\tilde{q}+q^R)_x\tilde{\theta}^2_x|+|\tilde{\theta}_{x}\tilde{q_{x}}\theta^R_x|+|\tilde{\theta}_xq^R_x\theta^R_x|+|\tilde{\theta}_{x}\tilde{q}(\theta^R_{x})^2|+|\tilde{\theta}_{x}q^R(\theta^R_x)^2|+|\theta^R_x(\tilde{q}+q^R)(\tilde{\theta}_{x})^2|\mathrm{d}x\mathrm{d}t \nonumber\\
			&\leq C\left(E^{\frac{1}{2}}(t)\int_{0}^{t}\mathcal{D}(t)\mathrm{d}t+\delta_{R}^2\int_{0}^{t}\mathcal{D}(t)\mathrm{d}t+\delta_{R}^{\frac{3}{2}-\frac{1}{q}}+\delta_{R}^{\frac{1}{2}}\right).
		\end{align}
		We next consider
		\begin{align}
			\int_{0}^{t}R_{6}\mathrm{d}t&=\int_{0}^{t}\int \left[\frac{1}{\theta}(\varphi_x+v^R_x)-\frac{v}{\theta^2}(\tilde{\theta}_{x}+\theta^R_x)\right]\psi_{x}\tilde{\theta}_{x}\mathrm{d}x\mathrm{d}t\nonumber\\
			&\leq C\int_{0}^{t}\int(|\varphi_{x}|+|v^R_{x}|+|\tilde{\theta}_{x}|+|\theta^R_{x}|)(\psi_{x}^2+\tilde{\theta}_{x}^2)\mathrm{d}x\mathrm{d}t\nonumber\\
			&\leq CE^{\frac{1}{2}}(t)\int_{0}^{t}\mathcal{D}(t)\mathrm{d}t+C\delta_{R}\int_{0}^{t}\mathcal{D}(t)\mathrm{d}t,
		\end{align}
		and similarly,
		\begin{align}
			\int_{0}^{t}R_{7}\mathrm{d}t&=-\int_{0}^{t}\int(p-p^R)_{x}u^R_x \frac{v^2}{R\theta^2}\tilde{\theta}_{x}\mathrm{d}x\mathrm{d}t\nonumber\\
			&=\int_{0}^{t}\int \frac{1}{\theta}(\varphi_{x}+v^R_{x})u^R_x \tilde{\theta}_{x}-\frac{v}{\theta^2}(\tilde{\theta}+\theta^R)_xu^R_x\tilde{\theta}_{x}-\frac{v^2\theta^Rv^R_xu^R_x}{(v^R)^2\theta^2}\tilde{\theta}_{x}+\frac{\theta^R_xu^R_xv^2}{v^R\theta^2}\tilde{\theta}_{x}\mathrm{d}x\mathrm{d}t \nonumber\\
			&\leq C\left(\delta_{R}\int_{0}^{t}\mathcal{D}(t)\mathrm{d}t+\delta_{R}^{\frac{1}{2}}\right),
		\end{align}
		\begin{equation}
			\begin{aligned}
				\int_{0}^{t}R_{8}\mathrm{d}t&=-\int_{0}^{t}\int(p-p^R)u^R_{xx}\frac{v^2}{R\theta^2}\tilde{\theta}_{x}
				\mathrm{d}x\mathrm{d}t
				\leq \int_{0}^{t} C\|\tilde{\theta}_{x}\|_{L^2}\| u^R_{xx}\|_{L^2}\mathrm{d}t
				\leq C\left( \delta_{R}^{\frac{3}{2}-\frac{1}{q}}+\delta_{R}^{\frac{1}{2}}\right).
			\end{aligned}
		\end{equation}
		For the term $R_9$,  we expand
		\begin{equation*}
			\begin{aligned}
				\left[\frac{(\tilde{q}+q^R )^2v}{\kappa(\tilde{\theta}+\theta^R)}\right]_{x}
				&=\frac{2\tilde{q}_{x}\tilde{q}v+2\tilde{q}_{x}q^Rv+2\tilde{q}q^R_{x}v+\tilde{q}^2v_{x}+2\tilde{q}q^Rv_{x}+2q^Rq^R_xv+(q^R)^2v_x}{\kappa(\tilde{\theta}+\theta^R)}\\
				&\quad-\frac{(\tilde{q})^2v\tilde{\theta}_{x}+2\tilde{q}q^Rv\tilde{\theta}_{x}+\tilde{q}^2v\theta^R_{x}+2\tilde{q}q^Rv\theta^R_{x}+(q^R)^2v\tilde{\theta}_{x}+(q^R)^2v\theta^R_x}{\kappa(\tilde{\theta}+\theta^R)^2}.
			\end{aligned}
		\end{equation*}
		Recalling $q^R=-\frac{\kappa\theta^R_{x}}{v^R}$, and $q^R_{x}=-\frac{\kappa\theta^R_{xx}}{v^R}+\frac{\kappa\theta^R_{x}v^R_{x}}{(v^R)^2}$, we obtain the estimate
		\begin{align}
			\int_{0}^{t}R_{9}\mathrm{d}t &\leq C\int_{0}^{t}\int \bigl|\tilde{q}_{x}\tilde{q}\tilde{\theta}_{x} +\tilde{q}q^R_{x}\tilde{\theta}_{x}+\tilde{q}_xq^R\tilde{\theta}_{x}+q^Rq^R_x\tilde{\theta}_{x}+\tilde{q}^2\varphi_x\tilde{\theta}_{x}+\tilde{q}^2v^R_x\tilde{\theta}_{x}
			+q^R\varphi_x\tilde{\theta}_x\nonumber\\
			&\qquad\qquad\quad+q^Rv^R_x\tilde{\theta}_{x}+(q^R)^2\varphi_x\tilde{\theta}_{x}+(q^R)^2v^R_x\tilde{\theta}_{x}|\mathrm{d}x\mathrm{d}t\nonumber\\
			&\quad+\int_{0}^{t}\int|\tilde{q}^2\tilde{\theta}_{x}^2+q^R\tilde{\theta}_{x}^2+\tilde{q}^2\theta^R_x\tilde{\theta}_{x}+q^R\theta^R_x\tilde{\theta}_{x}+(q^R)^2\tilde{\theta}_{x}^2+(q^R)^2\theta^R_x\tilde{\theta}_{x}|\mathrm{d}x\mathrm{d}t\nonumber\\
			&\leq C\delta_{R}\int_{0}^{t}\mathcal{D}(t)\mathrm{d}t+C(\delta_{R}^{\frac{3}{2}-\frac{1}{q}}+\delta_{R}^{\frac{1}{2}})+CE^{\frac{1}{2}}(t)\int_{0}^{t}\mathcal{D}(t)\mathrm{d}t+CE(t)\int_{0}^{t}\mathcal{D}(t)\mathrm{d}t.
		\end{align}
		Proceeding analogously,
		\begin{align}
			\int_{0}^{t}R_{10}\mathrm{d}t&=\int_{0}^{t}\int \left[\frac{v}{\mu}(\tilde{S}^2+2S^R\tilde{S}+(S^R)^2)\right]_{x} \frac{v^2}{R\theta^2}\tilde{\theta}_{x}\mathrm{d}x\mathrm{d}t\nonumber\\
			&\leq C\int_{0}^{t}\int\big(\varphi_{x}\tilde{S}^2+v^R_{x}\tilde{S}^2+\varphi_{x}(S^R)^2+v^R_{x}(S^R)^2\nonumber\\
			&\qquad\qquad+2v\tilde{S}\tilde{S}_{x}+2v\tilde{S}_{x}S^R+2\tilde{S}\tilde{S}_{x}+2S^RS^R_{x}\big)\tilde{\theta}_{x}\mathrm{d}x\mathrm{d}t\nonumber\\
			&\leq C(\epsilon_{1}+\delta_{R})\int_{0}^{t}\mathcal{D}(t)\mathrm{d}t+(\delta_{R}^{\frac{3}{2}-\frac{1}{q}}+\delta_{R}^{\frac{1}{2}})+CE^{\frac{1}{2}}(t)\int_{0}^{t}\mathcal{D}(t)\mathrm{d}t,
		\end{align}
		and
		\begin{align}
			\int_{0}^{t}R_{11}\mathrm{d}t&\leq C\int_{0}^{t}\int (\frac{\theta^R_{x}}{v^R})_{xx}\tilde{\theta}_{x}\mathrm{d}x\mathrm{d}t \nonumber\\
			&\leq C\int_{0}^{t}\int\left(\frac{\theta^R_{xxx}}{v^R}-\frac{2\theta^R_{xx}v^R_{x}}{(v^R)^2}-\frac{\theta^R_{x}v^R_{xx}}{(v^R)^2}+\frac{2\theta^R_{x}(v^R_x)^2}{(v^R)^3}\right)\tilde{\theta}_{x}\mathrm{d}x\mathrm{d}t \nonumber\\
			&\leq C(\delta_{R}^{\frac{3}{2}-\frac{1}{q}}+\delta_{R}^{\frac{1}{2}}).
		\end{align}
		We differentiate both sides of equation $(\ref{3.7})_3$ with respect to $x$, which yields
		$\theta^R_{tx}=-(\gamma-1)\left(\frac{\theta^R_xu^R_x}{v^R}+\frac{\theta^Ru^R_{xx}}{v^R}-\frac{\theta^Ru^R_x\theta^R_x}{(v^R)^2}\right)$. Using this identity, the term $R_{12}$ can be estimated as follows
		\begin{align}
			\int_{0}^{t}R_{12}\mathrm{d}t&=\int_{0}^{t}\int\left[\frac{\tau_{1}(\tilde{q}+q^R)^2\theta^R_{t}}{2\kappa(\tilde{\theta}+\theta^R)^2}\right]_{x}\frac{v^2}{R\theta^2}\tilde{\theta}_{x}\mathrm{d}x\mathrm{d}t \nonumber\\
			&\leq C\int_{0}^{t}\int\left(\frac{2(\tilde{q}_{x}+q^R)(\tilde{q}_{x}+q^R)_x\theta^R_{t}+(\tilde{q}_{x}+q^R)^2\theta^R_{tx}}{(\tilde{\theta}+\theta)^2}-\frac{2(\tilde{q}_{x}+q^R)^2\theta^R_{t}(\tilde{\theta}+\theta^R)_{x}}{(\tilde{\theta}+\theta)^3}\right)\tilde{\theta}_{x}\mathrm{d}x\mathrm{d}t \nonumber\\
			&\leq C\delta_{R}\int_{0}^{t}\mathcal{D}(t)\mathrm{d}t+C(\delta_{R}^{\frac{3}{2}-\frac{1}{q}}+\delta_{R}^{\frac{1}{2}}).
		\end{align}
		Next, we obtain
		\begin{align}
			\int_{0}^{t}R_{13}\mathrm{d}t&=-\int_{0}^{t}\int(v_{x}\tilde{q}+q^R_{x}\varphi+q^R\varphi_{x})\frac{v^2}{\kappa R\theta^2}\tilde{q}_{x}\mathrm{d}x\mathrm{d}t-\int_{0}^{t}\int\frac{\tau_1v}{\kappa R\theta^2}\left(v_t-\frac{v\theta_t}{\theta}\right)\tilde{q}_{x}^2\mathrm{d}x\mathrm{d}t\nonumber\\
			&\leq C\int_{0}^{t}\int| \varphi_{x}\tilde{q}\tilde{q}_{x}+v^R_{x}\tilde{q}\tilde{q}_{x}+\theta^R_{xx}\varphi\tilde{q}_{x}+\theta^R_{x}v^R_{x}\varphi\tilde{q}_{x}+\theta^R_{x}\varphi_{x}\tilde{q}_{x}|\mathrm{d}x\mathrm{d}t+CE^{\frac{1}{2}}(t)\int_{0}^{t}\mathcal{D}(t)\mathrm{d}t\nonumber\\
			&\leq C\left(E^{\frac{1}{2}}(t)\int_{0}^{t}\mathcal{D}(t)\mathrm{d}t+\delta_{R}\int_{0}^{t}\mathcal{D}(t)\mathrm{d}t+\delta_{R}^{\frac{3}{2}-\frac{1}{q}}+\delta_{R}^{\frac{1}{2}}\right).
		\end{align}
		By the same argument, we can obtain
		\begin{align}
			\int_{0}^{t}R_{14}\mathrm{d}t&=-\int_{0}^{t}\int(v_{x}\tilde{S}+S^R_{x}\varphi+S^R\varphi_{x})\frac{v^2}{\mu R\theta}\tilde{S}_{x}\mathrm{d}x\mathrm{d}t-\int_{0}^{t}\int\frac{\tau_2v}{\mu R\theta}\left(v_t-\frac{v\theta_t}{2\theta}\right)\tilde{S}_{x}^2\mathrm{d}x\mathrm{d}t\nonumber\\
			&\leq C\left(E^{\frac{1}{2}}(t)\int_{0}^{t}\mathcal{D}(t)\mathrm{d}t+\delta_{R}\int_{0}^{t}\mathcal{D}(t)\mathrm{d}t+\delta_{R}^{\frac{3}{2}-\frac{1}{q}}+\delta_{R}^{\frac{1}{2}}\right).
		\end{align}
		From equation \eqref{3.9}, we derive that
		\begin{align*}
			q^R_{tx}=&-\frac{\kappa\theta^R_{xxt}}{v^R}+\frac{\kappa\theta^R_{xx}u^R_{x}}{(v^R)^2}+\frac{\kappa\theta^R_{xt}v^R_{x}}{(v^R)^2}+\frac{\kappa\theta^R_{x}u^R_{xx}}{(v^R)^2}-\frac{2\kappa\theta^R_{x}v^R_{x}u^R_{x}}{(v^R)^3},\\
			\theta^R_{xt}=&-(\gamma-1)\left(\frac{\theta^R_{x}u^R_{x}}{v^R}+\frac{\theta^Ru^R_{xx}}{v^R}-\frac{\theta^Ru^R_{x}v^R_{x}}{(v^R)^2}\right),\\
			\theta^R_{xxt}=&-(\gamma-1)\bigg(\frac{\theta^R_{xx}u^R_{x}}{v^R}+\frac{2\theta^R_{x}u^R_{xx}}{v^R}-\frac{2\theta^R_{x}u^R_{x}v^R_{x}}{(v^R)^2}+\frac{\theta^Ru^R_{xxx}}{v^R}\\
			&-\frac{2\theta^Ru^R_{xx}v^R_{x}}{(v^R)^2}-\frac{\theta^Ru^R_{x}v^R_{xx}}{(v^R)^2}+\frac{2\theta^Ru^R_{x}(v^R_{x})^2}{(v^R)^3}\bigg).
		\end{align*}
		Substituting these expressions into $R_{15}$ leads to
		\begin{align}\label{4.52}
			\int_{0}^{t}R_{15}\mathrm{d}t&=-\int_{0}^{t}\int\tau_{1}q^R_{tx}\frac{v^2}{\kappa R\theta^2}\tilde{q}_{x}\mathrm{d}x\mathrm{d}t \nonumber\\
			&\leq C\int_{0}^{t}\int|\theta^R_{xx}u^R_{x}+\theta^R_{x}u^R_{xx}+\theta^R_{x}u^R_{x}v^R_{x}+\theta^Ru^R_{xxx}+\theta^Ru^R_{xx}v^R_{x}+\theta^Ru^R_{x}v^R_{xx}+\theta^Ru^R_{x}(v^R_{x})^2| \nonumber\\
			&\qquad\qquad\qquad\qquad\qquad\qquad\qquad\qquad\qquad\qquad\qquad\qquad\qquad\qquad\qquad\quad\qquad
			\cdot|\tilde{q}_{x}|\mathrm{d}x\mathrm{d}t \nonumber\\
			&\leq C(\delta_{R}^{\frac{3}{2}-\frac{1}{q}}+\delta_{R}^\frac{1}{2}).
		\end{align}
		Similar to $R_{15}$, we can obtain
		\begin{equation}
			\begin{aligned}
				\int_{0}^{t}R_{16}\mathrm{d}t=-\int_{0}^{t}\int\tau_{2}S^R_{tx}\frac{v^2}{\mu R\theta}\tilde{S}_{x}\mathrm{d}x\mathrm{d}t
				\leq C(\delta_{R}^{\frac{3}{2}-\frac{1}{q}}+\delta_{R}^\frac{1}{2}).
			\end{aligned}
		\end{equation}
		For the final term $R_{17}$, we have
		\begin{align}\label{3.78}
			\int_{0}^{t}R_{17}\mathrm{d}t&=\int_{0}^{t}\int\left(\left(\frac{v_x}{\theta}-\frac{v\theta_x}{\theta^2}\right)\tilde{\theta}_{x}\psi_{x}-\frac{v}{\theta R}\left(2v_x-\frac{v\theta_x}{\theta}\right)\psi_{x}\tilde{S}_x+\frac{2v}{R\theta^2}\left(v_x-\frac{v\theta_x}{\theta}\right)\tilde{q_{x}}\tilde{\theta}_{x}\right)\mathrm{d}x\mathrm{d}t \nonumber\\
			&\leq C\left(E^{\frac{1}{2}}(t)\int_{0}^{t}\mathcal{D}(t)\mathrm{d}t+\delta_{R}\int_{0}^{t}\mathcal{D}(t)\mathrm{d}t\right).
		\end{align}
		Substituting equations $(\ref{3.54})$-$(\ref{3.78})$ into $(\ref{3.52})$, we obtain the estimate
		\begin{equation}
			\begin{aligned}
				&\|(\varphi_{x},\psi_{x},\tilde{\theta}_{x},\tilde{q}_{x},\tilde{S}_{x})\|^2_{L^2}+\int_{0}^{t}(\|\tilde{q}_{x}\|^2_{L^2}+\|\tilde{S}_{x}\|^2_{L^2})\mathrm{d}t\\
				&\leq C\left(E_0+E^{\frac{1}{2}}(t)\int_{0}^{t}\mathcal{D}(t)\mathrm{d}t+\delta_R\int_{0}^{t}\mathcal{D}(t)\mathrm{d}t+\delta_{R}^{\frac{3}{2}-\frac{1}{q}}+\delta_{R}^\frac{1}{2}\right).
			\end{aligned}
		\end{equation}
		This completes the proof of Lemma \ref{lem4.2}.
	\end{proof}
	Next, we establish the second-order estimates.
	\begin{lemma}\label{lem4.3}
		There exists a constant $C$ such that for $0\le t\le T$, we have,
		\begin{equation}
			\begin{aligned}
				& \left\|\varphi_{x x}\right\|_{L^{2}}^{2}+\left\|\varphi_{x x}\right\|_{L^{2}}^{2}+\left\|\tilde{\theta}_{x x}\right\|_{L^{2}}^{2}+\left\|\tilde{q}_{x x}\right\|_{L^{2}}^{2}+\left\|\tilde{S}_{x x}\right\|_{L^{2}}^{2}+\int_{0}^{t}\left(\left\|\tilde{q}_{x x}\right\|_{L^{2}}^{2}+\left\|\tilde{S}_{x x}\right\|_{L^{2}}^{2}\right) \mathrm{d} t \\
				\leq & C \left(E_{0}+ E^{\frac{1}{2}}(t) \int_{0}^{t} \mathcal{D} (t) \mathrm{d} t+\delta_{R} \int_{0}^{t} \mathcal{D}(t)  \mathrm{d} t+\delta_{R}^{\frac{2}{3}- \frac{1}{q}}+\delta_{R}^{\frac{1}{2}}\right).
			\end{aligned}
		\end{equation}
	\end{lemma}
	\begin{proof}
		Taking derivative with respect to $x$ twice to equation $(\ref{3.48})$, one get
		\begin{equation}\label{3.81}
			\begin{cases}
				\varphi_{txx}-\psi_{xxx}=0\\
				\psi_{txx}-\frac{R(v_{xx}\tilde{\theta}_x+2v_x\tilde{\theta}_{xx})}{v^2}+\frac{2Rv_x^2\tilde{\theta}_{x}}{v^3}-\frac{R(v_{xx}\theta^R_x+2v_x\theta^R_{xx})}{v^2}+\frac{2Rv_x^2\theta^R_x}{v^3}+\frac{R\tilde{\theta}_{xxx}+R\theta^R_{xxx}}{v}\\
				\quad-\frac{R\theta_{xx}(\varphi_x+v^R_x)+2R\theta_x(\varphi_{xx}+v^R_{xx})}{v^2}+\frac{4R\theta_x(\varphi_x+v^R_x)^2}{ v^3}+\frac{2R\theta v_{xx}(\varphi_x+v^R_x)}{v^3}+\frac{4R\theta v_x(\varphi_{xx}+v^R_{xx})}{v^3}\\
				\quad-\frac{6R\theta v^2_x(\varphi_x+v^R_x)}{v^4}-\frac{R\theta(\varphi_{xxx}+v^R_{xxx})}{v^2}+\frac{3Rv^R_{xx}\theta^R_{xx}+3Rv^R_{x}\theta^R_{xx}}{(v^R)^2}-\frac{6R(v^R_x)^2\theta^R_x}{(v^R)^3}-\frac{R\theta^R_{xxx}}{v^R}\\
				\quad+\frac{R\theta^Rv^R_{xxx}}{(v^R)^2}-\frac{6R\theta ^Rv^R_{xx}v^R_{x}}{(v^R)^3}+\frac{6R\theta^R(v^R_x)^3}{(v^R)^4}=\tilde{S}_{xxx}-(Q^R_{1})_{xx}\\
				\frac{R}{\gamma-1}\tilde{\theta}_{txx}-\left[\frac{\tau_{1}(\tilde{q}+q^R)^2}{2\kappa(\tilde{\theta}+\theta^R)^2}\tilde{\theta}_{t}\right]_{xx}-\left[\frac{(\tilde{\theta}+\theta^R)_{x}(\tilde{q}+q^R)}{\tilde{\theta}+\theta^R}\right]_{xx} +(p\psi_{x})_{xx}\\
				\quad+\left((p-p^R)_{x}u^R_x+(p-p^R)u^R_{xx}\right)_x+\tilde{q}_{xxx}\\
				\quad=\left[\frac{\tau_{1}(\tilde{q}+q^R )^2v}{\kappa(\tilde{\theta}+\theta^R)}\right]_{xx}+\left[\frac{v}{\mu}(\tilde{S}^2+2S^R\tilde{S}+(S^R)^2)\right]_{xx}-(Q^R_{2})_{xx}+\left[\frac{\tau_{1}(\tilde{q}+q^R)^2\theta^R_{t}}{2\kappa(\tilde{\theta}+\theta^R)^2}\right]_{xx}\\
				\tau_{1}\tilde{q}_{txx}+(v\tilde{q}+q^R\varphi)_{xx}+\kappa\tilde{\theta}_{xxx}=-\tau_{1}q^R_{txx}\\
				\tau_{2}\tilde{S}_{txx}+(v\tilde{S}+S^R\varphi)_{xx}-\mu\psi_{xxx}=-\tau_{2}S^R_{txx}
			\end{cases}
		\end{equation}
		Multiplying $(\ref{3.81})_{1,2,3,4,5}$ by $\varphi_{xx}$, $\frac{v^2}{R\theta}\psi_{xx}$, $\frac{v^2}{R\theta^2}\tilde{\theta}_{xx}$, $\frac{v^2}{\kappa R\theta^2}\tilde{q}_{xx}$, $\frac{v^2}{\mu R\theta}\tilde{S}_{xx}$, respectively, and integrating
		the results, we obtain
		\begin{equation}\label{3.82}
			\begin{aligned}
				&\int\left(\frac{1}{2}\varphi_{xx}^2+\frac{v^2}{2R\theta}\psi_{xx}^2+\frac{v^2}{2(\gamma-1)\theta^2}\tilde{\theta}^2_{xx}+\frac{\tau_{1}v^2}{2\kappa R\theta^2}\tilde{q}^2_{xx}+\frac{\tau_{2}v^2}{2\mu R\theta}\tilde{S}^2_{xx}\right)\mathrm{d}x\bigg|_{0}^{t}+\\
				&\int_{0}^{t}\int\frac{v^3}{\kappa R\theta^2}\tilde{q}^2_{xx}\mathrm{d}x\mathrm{d}t+
				\int_{0}^{t}\int\frac{v^3}{\mu R\theta}\tilde{S}^2_{xx}\mathrm{d}x\mathrm{d}t =:\sum\limits_{j=1}^{16}\int_{0}^{t}R_{j}\mathrm{d}t,
			\end{aligned}
		\end{equation}
		where
		\begin{align*}
			R_{1}&=\int\left(\frac{v}{R\theta}\left(v_t-\frac{v\theta_t}{2\theta}\right)\psi_{xx}^2+\frac{v}{(\gamma-1)\theta^2}\left(v_t-\frac{v\theta_t}{\theta}\right)\tilde{\theta}_{xx}^2\right)\mathrm{d}x,\\
			R_{2} &=\int\bigg[\frac{(v_{xx}\tilde{\theta}_x+2v_x\tilde{\theta}_{xx})}{\theta}\psi_{xx}-\frac{2v_x^2\tilde{\theta}_{x}}{\theta v}\psi_{xx}+\frac{(v_{xx}\theta^R_x+2v_x\theta^R_{xx})}{\theta}\psi_{xx}-\frac{2v_x^2\theta^R_x}{\theta v}\psi_{xx}-\frac{v\theta^R_{xxx}}{\theta}\psi_{xx}\\
			&\quad+\frac{\theta_{xx}(\varphi_x+v^R_x)+2\theta_x(\varphi_{xx}+v^R_{xx})}{\theta}\psi_{xx}-\frac{4\theta_x(\varphi_x+v^R_x)^2}{v\theta}-\frac{2 v_{xx}(\varphi_x+v^R_x)}{v}\psi_{xx}\\
			&\quad-\frac{4 v_x(\varphi_{xx}+v^R_{xx})}{v}\psi_{xx}+\frac{6 v^2_x(\varphi_x+v^R_x)}{v^2}\psi_{xx}+v^R_{xxx}\psi_{xx}-\frac{3v^R_{xx}\theta^R_{xx}+3v^R_{x}\theta^R_{xx}}{\theta(v^R)^2}v^2\psi_{xx}\\
			&\quad+\frac{6(v^R_x)^2\theta^R_xv^2\psi_{xx}}{\theta(v^R)^3}+\frac{v^2\theta^R_{xxx}}{\theta v^R}\psi_{xx}-\frac{v^2\theta^Rv^R_{xxx}\psi_{xx}}{\theta(v^R)^2}+\frac{6v^2\theta ^Rv^R_{xx}v^R_{x}\psi_{xx}}{\theta(v^R)^3}-\frac{6v^2\theta^R(v^R_x)^3\psi_{xx}}{\theta(v^R)^4}\bigg]\mathrm{d}x,\\
			R_{3}&=-\int (Q^R_{1})_{xx}\frac{v^2}{R\theta}\psi_{xx}\mathrm{d}x,\qquad R_{4}=\int\left[\frac{\tau_{1}(\tilde{q}+q^R)^2}{2\kappa(\tilde{\theta}+\theta^R)^2}\tilde{\theta}_{t}\right]_{xx}\frac{v^2}{R\theta^2}\tilde{\theta}_{xx}\mathrm{d}x,\\ R_{5}&=\int\left[\frac{(\tilde{\theta}+\theta^R)_{x}(\tilde{q}+q^R)}{\tilde{\theta}+\theta^R}\right]_{xx}\frac{v^2}{R\theta^2}\tilde{\theta}_{xx}\mathrm{d}x,\\
			R_{6} &= -\int\bigg[\frac{2(\varphi_x+v^R_x)^2}{v\theta}\psi_x\tilde{\theta}_{xx}+\frac{v(\tilde{\theta}_{xx}+\theta^R_{xx})}{\theta^2}\psi_x\tilde{\theta}_{xx}-\frac{2(\tilde{\theta}_x+\theta^R_x)v_x}{\theta^2}\psi_x\tilde{\theta}_{xx}-\frac{2(\varphi_x+v^R_x)}{\theta}\psi_{xx}\tilde{\theta}_{xx}\\
			&\quad+\frac{2v(\tilde{\theta}+\theta^R_x)}{\theta^2}\psi_{xx}\tilde{\theta}_{xx}-\frac{(\varphi_{xx}+v^R_{xx})}{\theta}\psi_x\tilde{\theta}_{xx}\bigg]\mathrm{d}x,\\
			R_{7} &= -\int\bigg((p-p^R)_{xx}u^R_x\frac{v^2}{R\theta^2}\tilde{\theta}_{xx}+2(p-p^R)_{x}u^R_{xx}\frac{v^2}{R\theta^2}\tilde{\theta}_{xx}+(p-p^R)u^R_{xxx}\frac{v^2}{R\theta^2}\tilde{\theta}_{xx}\bigg)\mathrm{d}x, \\
			R_{8} &= \int \left( \frac{(\tilde{q} + q^{R})^2 v}{\kappa (\tilde{\theta} + \theta^{R})} \right)_{xx}\frac{v^2}{R\theta^2} \tilde{\theta}_{xx} \mathrm{d}x, \qquad
			R_{9} = \int \left[ \frac{v}{\mu} \left( \tilde{S}^{2} + 2 \tilde{S} S^{R} + (S^{R})^{2} \right) \right]_{xx}\frac{v^2}{R\theta^2} \tilde{\theta}_{xx} \, \mathrm{d}x, \\
			R_{10} &= -\int (Q_{2}^{R})_{xx}\frac{v^2}{R\theta^2}\tilde{\theta}_{xx} \, \mathrm{d}x, 
			\quad\quad	R_{11}=\int\left[\frac{\tau_{1}(\tilde{q}+q^R)^2\theta^R_{t}}{2\kappa(\tilde{\theta}+\theta^R)^2}\right]_{xx}\frac{v^2}{R\theta^2}\tilde{\theta}_{xx}\mathrm{d}x,\\
			R_{12} &= -\int\left(v_{xx}\tilde{q}+2v_{x}\tilde{q}_x+v\tilde{q}_{xx}+q^R_{xx}\varphi+2q^R_x\varphi_{x}+q^R\varphi_{xx}\right)\frac{v^2}{\kappa R\theta^2}\tilde{q}_{xx}\mathrm{d}x\\
			&\quad-\int\frac{\tau_1v}{\kappa R\theta^2}\left(v_t-\frac{v\theta_t}{\theta}\right)\tilde{q}_{xx}^2\mathrm{d}x, \\
			R_{13} &= -\int\left( v_{xx} \tilde{S}  + 2 v_{x} \tilde{S}_{x} +v\tilde{S}_{xx} + S_{xx}^{R} \varphi 
			+ 2S_{x}^{R} \varphi_{x}  +  S^{R} \varphi_{xx} \right)\frac{v^2}{R\mu\theta}\tilde{S}_{xx} \mathrm{d}x\\
			&\quad
			-\int\frac{\tau_2v}{\mu R\theta}\left(v_t-\frac{v\theta_t}{2\theta}\right)\tilde{S}_{xx}^2\mathrm{d}x, \\
			R_{14}&=-\int\tau_{1}q^R_{txx}\frac{v^2}{\kappa R\theta^2}\tilde{q}_{xx}\mathrm{d}x,\qquad
			R_{15}=-\int\tau_{2}S^R_{txx}\frac{v^2}{\mu R\theta}\tilde{S}_{xx}\mathrm{d}x,\\
			R_{16}&=\int\left(\left(\frac{v_x}{\theta}-\frac{v\theta_x}{\theta^2}\right)\tilde{\theta}_{xx}\psi_{xx}-\frac{v}{\theta R}\left(2v_x-\frac{v\theta_x}{\theta}\right)\psi_{xx}\tilde{S}_{xx}+\frac{2v}{R\theta^2}\left(v_x-\frac{v\theta_x}{\theta}\right)\tilde{q}_{xx}\tilde{\theta}_{xx}\right)\mathrm{d}x.
		\end{align*}
		We now proceed to estimate each $R_i$ separately, $i=1, 2, \cdots, 16$. First, we have
		\begin{align}
			\int_{0}^{t}R_{1}\mathrm{d}t&=\int_{0}^{t}\int \frac{v}{R\theta}\left(v_t-\frac{v\theta_t}{2\theta}\right)\psi_{xx}^2+\frac{v}{(\gamma-1)\theta^2}\left(v_t-\frac{v\theta_t}{\theta}\right)\tilde{\theta}_{xx}^2\mathrm{d}x\mathrm{d}t \nonumber\\
			&\leq C\int_{0}^{t}\int(|\tilde{\theta}_{t}|+|\psi_{x}|+|u^R_{x}|)(\psi_{xx}^2+\tilde{\theta}_{xx}^2)\mathrm{d}x\mathrm{d}t \nonumber\\
			&\leq CE^{\frac{1}{2}}(t)\int_{0}^{t}\mathcal{D}(t)\mathrm{d}t+C\delta_{R}\int_{0}^{t}\mathcal{D}(t)\mathrm{d}t.
		\end{align}
		By the positivity of $v,\theta$ and the a priori assumptions, 
		all the coefficient functions in $R_2$ are uniformly bounded.
		Applying Hölder's inequality, Sobolev embedding, and Lemma~\ref{lem2.1},
		we deduce that
		\begin{align*}
			\int_{0}^{t}R_{2}\mathrm{d}t	&\leq C\int_{0}^{t}\|( v^R_{x}, v^R_{xx}, \theta^R_{x},\theta^R_{xx}, \tilde{\theta}_{x}, \tilde{\theta}_{xx}, \varphi_{x})\|_{L^\infty}\cdot\|( \varphi_{xx},\psi_{xx}, \tilde{\theta}_{xx}, \tilde{\theta}_{x}, \varphi_{x})\|^2_{L^2}\mathrm{d}t\\
			&\quad
			+ C\int_0^t
			\Big(
			\|v^R_x\|_{L^6}^3
			+\|v^R_x\|_{L^4}\|v^R_{xx}\|_{L^4}
			+\|\theta^R_x\|_{L^4}\|v^R_{xx}\|_{L^4} 
			+\|v^R_{xxx}\|_{L^2}
			+\|\theta^R_{xxx}\|_{L^2}
			\Big)
			\|\psi_{xx}\|_{L^2}\,\mathrm{d}t\\
			&\leq C\left(E^{\frac{1}{2}}(t)\int_{0}^{t}\mathcal{D}(t)\mathrm{d}t+\delta_{R}\int_{0}^{t}\mathcal{D}(t)\mathrm{d}t+\delta_{R}^{\frac{1}{2}}+\delta_{R}^{\frac{3}{2}-\frac{1}{q}}\right).
		\end{align*}
		For the term $R_3$, we note that
		\begin{align*}
			\left|\left(\frac{u_{x}^{R}}{v^{R}}\right)_{x x x}\right| \leq C\left|u_{x xx x}^{R}-u_{x x x}^{R} v_{x}^{R}-u_{x x}^{R} v_{x x}^{R}+u_{x x}^{R} (v_{x}^{R})^2+u_{x}^{R} v_{x x x}^{R}+u_{x}^{R} v_{x}^{R} v_{x x}^{R}-u_{x}^{R} (v_{x}^{R})^{3}\right| .
		\end{align*}
		Applying Lemma \ref{lem2.1} and H\"{o}lder inequality, we have
		\begin{align}
			\int_{0}^{t} R_{3} \mathrm{d} t& \leq C\int_{0}^{t} \int \left|   \left(\frac{u_{x}^{R}}{v^{R}}\right)_{xxx}\psi_{xx}\right| \mathrm{d}x\mathrm{d}t 
			\leq C\delta_{R}^{\frac{1}{2}}+C\delta_{R}^{\frac{3}{2}-\frac{1}{q}}.
		\end{align}
		Next, we estimate $R_4$. Upon expanding the second spatial derivative, we obtain the identities
		\begin{align*}
			&\left(\frac{(\tilde{q}+q^R)^{2} \tilde{\theta}_{t}}{(\tilde{\theta}+\theta^{R})^{2}}\right)_{xx}\\
			&=\left(\frac{2(\tilde{q}+q^R)(\tilde{q}+q^R)_{x} \tilde{\theta}_{t}}{(\theta^{R}+\tilde{\theta})^{2}}+\frac{(\tilde{q}+q^R)^{2}\tilde{\theta}_{t x}}{(\tilde{\theta}+\theta^{R})^{2}}-\frac{2(\tilde{q}+q^R)^{2} \tilde{\theta}_{t} \tilde{\theta}_{x}}{(\tilde{\theta}+\theta^{R})^{3}}-\frac{2 (\tilde{q}+q^R)^{2} \tilde{\theta}_{t} \theta_{x}^{R}}{(\tilde{\theta}+\theta^{R})^{3}}\right)_{x}\\
			&=\frac{2(\tilde{q}+q^R)_{x}^{2} \tilde{\theta}_{t}+2 (\tilde{q}+q^R) (\tilde{q}+q^R)_{xx} \tilde{\theta}_{t}+4 (\tilde{q}+q^R)(\tilde{q}+q^R)_{x} \tilde{\theta}_{t x}+(\tilde{q}+q^R)^{2} \tilde{\theta}_{t x x}}{(\tilde{\theta}+\theta^{R})^{2}}\\
			&\quad-\frac{2 (\tilde{q}+q^R)^{2} \tilde{\theta}_{tx} \theta_{x}^{R}+8(\tilde{q}+q^R)(\tilde{q}+q^R)_{x}\tilde{\theta}_{t} \tilde{\theta}_{x}+4 (\tilde{q}+q^R)^{2} \tilde{\theta}_{t x} \tilde{\theta}_{x}}{( \tilde{\theta}+\theta^{R})^{3}} \\
			&\quad-\frac{4 (\tilde{q}+q^R)(\tilde{q}+q^R)_{x} \tilde{\theta}_{t} \theta_{x}^{R}+2 (\tilde{q}+q^R)^{2}\tilde{\theta}_{t} \tilde{\theta}_{tx}}{( \tilde{\theta}+\theta^{R})^{3}}\\
			&\quad+\frac{6(\tilde{q}+q^R)^{2} \tilde{\theta}_{t} \tilde{\theta}_{x}(\tilde{\theta}+\theta^{R})_{x}-4 (\tilde{q}+q^R)(\tilde{q}+q^R)_{x}\tilde{\theta}_{t}\theta_{x}^{R}-2 (\tilde{q}+q^R)^{2}\tilde{\theta}_{tx}\theta_{x}^{R}}{(\tilde{\theta}^{2}+\theta^{2})^4}\\
			&\quad-\frac{2 (\tilde{q}+q^R)^{2}\tilde{\theta}_{t}\theta_{xx}^{R}-6(\tilde{q}+q^R) \tilde{\theta}_{t} \theta_{x}^{R}(\tilde{\theta}_{x}+\theta_{x}^{R})}{(\tilde{\theta}^{2}+\theta^{R})^4}.
		\end{align*}
		In the above expansion, the product involving $\tilde{\theta}_{txx}$ can be rewritten as
		\begin{align*}
			&\frac{(\tilde{q}+q^R)^{2} \tilde{\theta}_{txx}}{(\tilde{\theta}+\theta^{R})^{2}} \cdot \tilde{\theta}_{x x}\\
			\quad&=\left(\frac{(\tilde{q}+q^R)^{2} \tilde{\theta}_{x x}^{2}}{2(\tilde{\theta}+\theta^{R})^{2}}\right)_{t}-\frac{(\tilde{q}+q^R) (\tilde{q}+q^R)_{t} \tilde{\theta}_{x x}^{2}}{2\kappa(\tilde{\theta}+\theta^{R})^2}+\frac{ (\tilde{q}+q^R)^{2}(\tilde{\theta}+\theta^{R})_{t} \tilde{\theta}_{x x}^{2}}{2\kappa(\tilde{\theta}+\theta^{R})^{3}}.
		\end{align*}
		Consequently, combining the above expansions and applying H\"{o}lder inequality, we derive
		\begin{align*}
			\int_{0}^{t}R_{4} \mathrm{d}t&\leq C \int_{0}^{t}\int\left(\frac{\tau_{1}(\tilde{q}+q^R)^{2}}{2 \kappa\left(\tilde{\theta}+\theta^{R}\right)^2} \tilde{\theta}_{t}\right)_{xx}\tilde{\theta}_{xx} \mathrm{d}x\mathrm{d}t\\
			&\leq C \int_{0}^{t}\|\big(\tilde{\theta}_{t},\tilde{q}, \tilde{q}_{x}, q^R,q^R_x\big)\|_{L^\infty}
			\cdot\|\big(\tilde{q},\tilde{q}_{x},\tilde{q}_{xx},\tilde{\theta}_{x},\tilde{\theta}_{xx}\big)\|^2_{L^2}\mathrm{d}t\\
			&\quad+C\int_{0}^{t}(\| q^R_x\|^2_{L^4}+\| q^R_{xxx}\|_{L^2}+\| q^R\|_{L^4}\| q^R_{xx}\|_{L^4}+\| q^R_x\|_{L^4}\|\theta^R_x\|_{L^4})\|\tilde{\theta}_{xx}\|_{L^2}\mathrm{d}t\\
			&\quad+\int\frac{\tau_{1}}{2 \kappa}\left(\frac{(\tilde{q}+q^R)^{2} \tilde{\theta}_{x x}^{2}}{(\tilde{\theta}+\theta^{R})^{2}}\right)\bigg|^{t}_{0}\mathrm{d}x\\
			&\leq C \left(E^{\frac{1}{2}}(t) \int_{0}^{t}\mathcal{D}(t) \mathrm{d} t+\delta_{R}\int_{0}^{t} \mathcal{D}(t) \mathrm{d} t
			+\delta_R^{\frac{2}{3}-\frac{1}{q}}+\delta_R^{\frac{1}{2}}\right)
			+\int\frac{\tau_{1}}{2 \kappa}\left(\frac{(\tilde{q}+q^R)^{2} \tilde{\theta}_{x x}^{2}}{(\tilde{\theta}+\theta^{R})^{2}}\right)\bigg|^{t}_{0}\mathrm{d}x.
		\end{align*}
		The estimate for $R_5$ is handled in the same manner as $R_4$. We first expand the second derivative
		\begin{align*}
			&\left(\frac{(\tilde{\theta}+\theta^{R})_{x} (\tilde{q}+q^R)}{(\tilde{\theta}+\theta^{R})}\right)_{x x}\\
			&=\left(\frac{\tilde{\theta}_{xx} (\tilde{q}+q^R)}{\tilde{\theta}+\theta^{R}}+\frac{\theta_{x x}^{R} (\tilde{q}+q^R)}{\tilde{\theta}+\theta^{R}}+\frac{\tilde{\theta}_{x} (\tilde{q}+q^R)_{x}+\theta_{x}^{R} (\tilde{q}+q^R)_{x}}{\tilde{\theta}+\theta^{R}}-\frac{(\tilde{q}+q^R) (\tilde{\theta}+\theta^R)_{x}^{2}}{(\tilde{\theta}+\theta^{R})^{2}}\right)_{x}\\
			&=\frac{\tilde{\theta}_{x x x}( \tilde{q}+q^R)+\tilde{\theta}_{x x} (\tilde{q}+q^R)_{x}}{\tilde{\theta}+\theta^{R}}-3\frac{\tilde{\theta}_{x x} (\tilde{q}+q^R)(\tilde {\theta}+\theta^{R})_{x}}{(\tilde{\theta}+\theta^{R})^{2}}+\frac{\theta_{x x x}^{R} (\tilde{q}+q^R)+\theta_{x x}^{R} (\tilde{q}+q^R)_{x}}{\tilde{\theta}+\theta^{R}}\\
			&\quad-3\frac{\theta_{x x}^{R} (\tilde{q}+q^R)(\tilde{\theta}+\theta^{R})_{x}}{(\tilde{\theta}+\theta^{R})^{2}}+\frac{\tilde{\theta}_{x x} (\tilde{q}+q^R)_{x}+\tilde{\theta}_{x} (\tilde{q}+q^R)_{x x}+\theta_{x x}^{R} (\tilde{q}+q^R)_{x}+\theta_{x}^{R} (\tilde{q}+q^R)_{x x}}{\tilde{\theta}+\theta^{R}} \\
			&\quad -\frac{2\left(\tilde{q}+q^R\right)_{x}\left( \tilde{\theta} + \theta^{R} \right)_{x}^2}{\left( \tilde{\theta} + \theta^{R} \right)^{2}} 
			+ \frac{ 2 \left( \tilde{q}+q^R \right) \left( \tilde{\theta} + \theta^{R} \right)_{x}^3 }{ \left( \tilde{\theta} + \theta^{R} \right)^{3} }
		\end{align*}
		in which
		\begin{align}
			\frac{\tilde{\theta}_{xxx} (\tilde{q}+q^R)}{\tilde{\theta}+\theta^{R}}\tilde{\theta}_{xx}
			=\left(\frac{(\tilde{q}+q^R)\tilde{\theta}_{x x}^{2}}{2(\tilde{\theta}+\theta^{R})}\right)_{x}
			-\frac{(\tilde{q}+q^R)_x \tilde{\theta}_{xx}^{2} }{ 2(\tilde{\theta} + \theta^{R}) }
			+ \frac{  (\tilde{q}+q^R)(\tilde{\theta}+\theta^R)_{x}\tilde{\theta}_{xx}^{2} }{ (\tilde{\theta} + \theta^{R})^2 }\nonumber.
		\end{align}
		Following the same argument as for $R_4$, we finally obtain the estimate
		\begin{align*}
			\int_{0}^{t}R_{5}\mathrm{d}t
			\leq C \left(E^{\frac{1}{2}}(t) \int_{0}^{t} \mathcal{D}(t) \mathrm{d}t+\delta_{R} \int_{0}^{t} \mathcal{D}(t) \mathrm{d}t+\delta_{R}^{\frac{3}{2}-\frac{1}{q}}+\delta_{R}^{\frac{1}{2}}\right).
		\end{align*}
		Furthermore, we have
		\begin{align}
			\int_{0}^{t} R_{6} \,\mathrm{d} t &= -\int_{0}^{t}\int\Biggl[\frac{2(\varphi_x+v^R_x)^2}{v\theta}\psi_x\tilde{\theta}_{xx}
			+\frac{v(\tilde{\theta}_{xx}+\theta^R_{xx})}{\theta^2}\psi_x\tilde{\theta}_{xx}
			-\frac{2(\tilde{\theta}_x+\theta^R_x)v_x}{\theta^2}\psi_x\tilde{\theta}_{xx} \nonumber \\
			&\quad-\frac{2(\varphi_x+v^R_x)}{\theta}\psi_{xx}\tilde{\theta}_{xx}
			+\frac{2v(\tilde{\theta}+\theta^R_x)}{\theta^2}\psi_{xx}\tilde{\theta}_{xx}
			-\frac{(\varphi_{xx}+v^R_{xx})}{\theta}\psi_x\tilde{\theta}_{xx}\Biggr]\mathrm{d}x\mathrm{d}t \nonumber \\
			&\leq C\int_{0}^{t} \bigl\|(\varphi_x, v_x^R, \tilde{\theta}_x, \theta_x^R, v_{xx}^R, \theta_{xx}^R)\bigr\|_{L^{\infty}} \cdot\|\bigl(\psi_x, \tilde{\theta}_{xx},\psi_{xx}\bigr)\|_{L^2}^2 \,\mathrm{d}t \nonumber \\
			&\quad+ C\int_{0}^{t} \bigl\|(\psi_{xx}, \tilde{\theta}_{xx})\bigr\|_{L^{\infty}}
			\bigl(\|\varphi_{xx}\|_{L^2}^2 + \|\tilde{\theta}_{xx}\|_{L^2}^2\bigr) \,\mathrm{d}t \nonumber \\
			&\leq C\left(E^{\frac{1}{2}}(t) \int_{0}^{t} \mathcal{D}(t) \,\mathrm{d}t+\delta_R \int_{0}^{t} \mathcal{D}(t) \,\mathrm{d}t\right).
		\end{align}
		We now estimate $R_7$, some tedious calculations give
		\begin{align*}
			(p-p^R)_x\frac{2v^2}{R\theta^2}u^R_{xx}\tilde{\theta}_{xx}=-\frac{1}{\theta}(\varphi_{x}+v^R_{x})u^R_{xx} \tilde{\theta}_{xx}+\frac{v}{\theta^2}(\tilde{\theta}+\theta^R)_xu^R_{xx}\tilde{\theta}_{xx}+\frac{v^2\theta^Rv^R_xu^R_{xx}}{(v^R)^2\theta^2}\tilde{\theta}_{xx}-\frac{\theta^R_xu^R_{xx}v^2}{v^R\theta^2}\tilde{\theta}_{xx},
		\end{align*}
		\begin{align*}
			(p-p^R)_{xx}&=-\left(\frac{\theta_x}{v^2}-\frac{2\theta v_x}{v^3}\right)(\varphi_x+v^R_x)-\frac{R\theta}{v^2}(\varphi_{xx}+v^R_{xx})-\frac{Rv_x}{v^2}(\tilde{\theta}_x+\theta^R_x)+\frac{R}{v}(\tilde{\theta}_{xx}+\theta^R_{xx})\\
			&+\frac{2R\theta_x^Rv_x^R+R\theta^Rv^R_{xx}}{(v^R)^2}-\frac{2R\theta^R(v_x^R)^2}{(v^R)^3}-\frac{R\theta_{xx}^R}{v^R}.
		\end{align*}
		Substituting these into the expression of $R_7$ and applying Lemma \ref{lem2.1} along with Sobolev and H\"{o}lder inequalities, we derive that
		\begin{align*}
			\int_{0}^{t} R_{7}\,\mathrm{d}t 
			&= -\int_{0}^{t}\int\left((p-p^R)_{xx}u^R_x\frac{v^2}{R\theta^2}\tilde{\theta}_{xx}+2(p-p^R)_{x}u^R_{xx}\frac{v^2}{R\theta^2}\tilde{\theta}_{xx}+(p-p^R)u^R_{xxx}\frac{v^2}{R\theta^2}\tilde{\theta}_{xx}\right) \mathrm{d}x\mathrm{d}t\\
			&\leq C\int_{0}^{t} \bigl\| \bigl( u^R_x, \varphi_x, \varphi_x^2, u^R_{xx} \bigr) \bigr\|_{L^\infty}
			\cdot \bigl\| \bigl( \varphi_{xx}, \varphi_x, \tilde{\theta}_{xx}, \tilde{\theta}_x \bigr) \bigr\|_{L^2}^2 \,\mathrm{d}t\\
			&\quad+C\int_{0}^{t}\Bigl(\|v_x^R\|_{L^6}\|\theta^R_x\|_{L^6}\|u_x^R\|_{L^6}
			+\|\theta_{xx}^R\|_{L^4}\|u_x^R\|_{L^4}+\|\theta^R_x\|^2_{L^6}\|u_x^R\|_{L^6} \\
			&\qquad\qquad\quad
			+\bigl(\|v_x^R\|_{L^4}+\|\theta_x^R\|_{L^4}\bigr)\|u_{xx}^R\|_{L^4} +\|u_{xxx}^R\|_{L^2}\Bigr)\|\tilde{\theta}_{xx}\|_{L^2}\,\mathrm{d}t \\
			&\leq C \left(E^{\frac{1}{2}}(t) \int_{0}^{t} \mathcal{D}(t) \,\mathrm{d} t+\delta_{R} \int_{0}^{t} \mathcal{D}(t) \,\mathrm{d} t+\delta_{R}^{\frac{3}{2}-\frac{1}{q}}+\delta_{R}^{\frac{1}{2}}\right).
		\end{align*}
		Following a similar procedure, we obtain the estimates for $R_8$ and $R_9$, 
		\begin{align*}
			\int_{0}^{t}R_{8}\mathrm{d}t &=\int_{0}^{t}\int\left( \frac{(\tilde{q} + q^{R})^2 v}{\kappa (\tilde{\theta} + \theta^{R})} \right)_{xx}\frac{v^2}{R\theta^2} \tilde{\theta}_{xx} \mathrm{d}x\mathrm{d}t\\
			&\leq C \left(E^{\frac{1}{2}}(t) \int_{0}^{t} D(t) \mathrm{d} t+ E(t) \int_{0}^{t} D(t) \mathrm{d} t+\delta_{R}^{\frac{3}{2}-\frac{1}{q}}+\delta_{R}^{\frac{1}{2}}\right),
		\end{align*}
		and	
		\begin{align*}
			\int_{0}^{t} R_{9} \mathrm{d}t&=\int_{0}^{t}\int \left[ \frac{v}{\mu} \left( \tilde{S}^{2} + 2 \tilde{S} S^{R} + (S^{R})^{2} \right) \right]_{xx}\frac{v^2}{R\theta^2} \tilde{\theta}_{xx}  \mathrm{d}x\mathrm{d}t\\
			&\leq C \left(E^{\frac{1}{2}}(t) \int_{0}^{t} \mathcal{D}(t) \mathrm{d}t+\delta_{R}^{2} \int_{0}^{t}\mathcal{D}(t) \mathrm{d}t+\delta_{R}^{\frac{1}{2}}+\delta_{R}^{\frac{5}{2} - \frac{1}{q}}\right).
		\end{align*}
		The term $R_{10}$ is estimated similarly to $R_3$, yielding	
		\begin{align}
			\int_{0}^{t} R_{10} \mathrm{d}t \leq C\int_{0}^{t} \int \left(\frac{\theta_{x}^{R}}{v^{R}}\right)_{x x x} \tilde{\theta}_{x x} \mathrm{d} x 
			\leq C\delta_{R}^{\frac{1}{2}}+C\delta_{R}^{\frac{3}{2}-\frac{1}{q}}.
		\end{align}
		For term \(R_{11}\), we first note the following identities:
		$$
		\theta^R_{xt}=-(\gamma-1)\left(\frac{\theta^R_{x}u^R_{x}}{v^R}+\frac{\theta^Ru^R_{xx}}{v^R}-\frac{\theta^Ru^R_{x}v^R_{x}}{(v^R)^2}\right),
		$$
		\begin{align*}
			\theta^R_{xxt}
			&=-(\gamma-1)\bigg(\frac{\theta^R_{xx}u^R_{x}}{v^R}+\frac{\theta^R_{x}u^R_{xx}}{v^R}-\frac{\theta^R_{x}u^R_{x}v^R_{x}}{(v^R)^2}+\frac{\theta^R_{x}u^R_{xx}}{v^R}+\frac{\theta^Ru^R_{xxx}}{v^R}\\
			&\quad-\frac{\theta^Ru^R_{xx}v^R_{x}}{(v^R)^2}-\frac{\theta^R_{x}u^R_{x}v^R_{x}}{(v^R)^2}-\frac{\theta^Ru^R_{xx}v^R_{x}}{(v^R)^2}-\frac{\theta^Ru^R_{x}v^R_{xx}}{(v^R)^2}+\frac{2\theta^Ru^R_{x}(v^R_{x})^2}{(v^R)^3}\bigg),
		\end{align*}
		as well as the expansion
		\begin{align*}
			&\left[\frac{(\tilde{q}+q^R)^2\theta^R_{t}}{(\tilde{\theta}+\theta^R)^2}\right]_{xx}\\
			&=\left(\frac{2(\tilde{q}+q^R)(\tilde{q}+q^R)_x\theta^R_{t}+(\tilde{q}+q^R)^2\theta^R_{tx}}{(\tilde{\theta}+\theta^R)^2}-\frac{2(\tilde{q}+q^R)^2\theta^R_{t}(\tilde{\theta}+\theta^R)_x}{(\tilde{\theta}+\theta^R)^3}\right)_{x}\\
			&=\frac{2(\tilde{q}+q^R)_x^2\theta^R_t+2(\tilde{q}+q^R)(\tilde{q}+q^R)_{xx}\theta^R_t+4(\tilde{q}+q^R)(\tilde{q}+q^R)_x\theta^R_{tx}+(\tilde{q}+q^R)^2\theta^R_{txx}}{(\tilde{\theta}+\theta^R)^2}\\
			&\quad-\frac{4(\tilde{q}+q^R)(\tilde{q}+q^R)_x\theta^R_t(\tilde{\theta}+\theta^R)_x+2(\tilde{q}+q^R)^2\theta^R_{tx}(\tilde{\theta}+\theta^R)_x+(\tilde{q}+q^R)^2\theta^R_t(\tilde{\theta}+\theta^R)_{xx}}{(\tilde{\theta}+\theta^R)^3}\\
			&\quad+\frac{3(\tilde{q}+q^R)^2\theta^R_t(\tilde{\theta}+\theta^R)_x^2}{(\tilde{\theta}+\theta^R)^4}.
		\end{align*}
		Substituting these identities into \(R_{11}\) and invoking Lemma \ref{lem2.1} together with Sobolev and Hölder's inequalities, we deduce that
		\begin{align}
			\int_{0}^{t}R_{11}\mathrm{d}t\leq C\int_{0}^{t}\int\left|\left(\frac{(\tilde{q}+q^R)^2\theta^R_{t}}{(\tilde{\theta}+\theta^R)^2}\right)_{xx}\tilde{\theta}_{xx}\right|\mathrm{d}x\mathrm{d}t
			\leq C\left(E(t)\int_{0}^{t}\mathcal{D}(t)\mathrm{d}t+\delta_{R}\int_{0}^{t}\mathcal{D}(t)\mathrm{d}t+\delta_R^{\frac{1}{2}}\right).
		\end{align}
		For the term \(R_{12}\), we first recall the identities for the derivatives of the rarefaction wave flux \(q^R\):
		$$
		q_{x}^{R}=-\frac{\kappa \theta_{x x}^{R}}{v^{R}}+\frac{\kappa (\theta_x^R)^2 v_{x}^{R}}{\left(v^{R}\right)^{2}}, \quad 
		q_{x x}^{R}=-\frac{\kappa \theta_{x x x}^{R}}{v^{R}}+\frac{\kappa \theta_{x x}^{R} v_{x}^{R}}{\left(v^{R}\right)^{2}}+\frac{\kappa \theta_{x x}^{R} v_{x}^{R}+\kappa \theta_{x}^{R} v_{x x}^{R}}{\left(v^{R}\right)^{2}}-\frac{2\kappa \theta_{x}^{R}\left(v_{x}^{R}\right)^{2}}{\left(v^{R}\right)^{3}}.
		$$
		Invoking Lemma \ref{lem2.1} together with Hölder's inequality, we deduce:
		\begin{align*}
			\int_{0}^{t} R_{12}\,\mathrm{d}t
			&= -\int_{0}^{t}\int\bigg(v_{xx}\tilde{q}+2v_{x}\tilde{q}_x+v\tilde{q}_{xx}+q^R_{xx}\varphi+2q^R_x\varphi_{x}+q^R\varphi_{xx}\bigg)\frac{v^2}{\kappa R\theta^2}\tilde{q}_{xx}\mathrm{d}x\mathrm{d}t\\
			&\quad-\int_{0}^{t}\int\frac{\tau_1v}{\kappa R\theta^2}\left(v_t-\frac{v\theta_t}{\theta}\right)\tilde{q}_{xx}^2\mathrm{d}x\mathrm{d}t \\
			&\leq C \int_{0}^{t} \bigl\| \bigl( \tilde{q}, \tilde{q}_x, v_{xx}^R, \theta_{xx}^R, \theta_x^R v_x^R \bigr) \bigr\|_{L^\infty}
			\cdot \bigl\| \bigl( \varphi_{xx}, \tilde{q}_{xx}, \tilde{q}_x \bigr) \bigr\|_{L^2}^2 \,\mathrm{d}t 
			 \\
			&\quad+ C\int_{0}^{t}\Bigl(\|\tilde{q}\|_{L^{\infty}}\|v_{xx}^{R}\|_{L^{2}}
			+\|\theta_{xx}^{R}\|_{L^{2}}
			+\|\theta_{xx}^{R}\|_{L^{4}}\|v_{x}^{R}\|_{L^{4}}
			+\|\theta_{x}^{R}\|_{L^{4}}\|v_{xx}^{R}\|_{L^{4}} \\
			&\qquad\qquad\quad
			+\|\theta_{x}^{R}\|_{L^{4}}\|(v_{x}^{R})^{2}\|_{L^{4}}\Bigr)
			\|\tilde{q}_{xx}\|_{L^2} \,\mathrm{d}t
			+C E^{\frac{1}{2}}(t) \int_{0}^{t}\mathcal{D}(t) \,\mathrm{d}t \\
			&\leq C \left(E^{\frac{1}{2}}(t) \int_{0}^{t}\mathcal{D}(t)\mathrm{d}t+\delta_{R} \int_{0}^{t} \mathcal{D}(t) \,\mathrm{d}t+\delta_{R}^{\frac{3}{2}-\frac{1}{q}}+\delta_{R}^{\frac{1}{2}}\right).
		\end{align*}
		By the same argument, we get
		\begin{equation}
			\begin{aligned}
				\int_{0}^{t} R_{13} d t\leq C \left(E^{\frac{1}{2}}(t) \int_{0}^{t}\mathcal{D}(t)\mathrm{d}t+\delta_{R} \int_{0}^{t} \mathcal{D}(t) \mathrm{d} t+\delta_{R}^{\frac{3}{2}-\frac{1}{q}}+\delta_{R}^{\frac{1}{2}}\right).
			\end{aligned}
		\end{equation}
		On the other hand, we recall the identity for the mixed third-order derivative of \(q^R\):
		\begin{align*}
			q^R_{txx}
			&=\kappa(\gamma-1)\bigg(\frac{\theta^R_{xxx}u_x^R+2\theta^R_{xx}u^R_{xx}}{(v^R)^2}-\frac{4\theta^R_{xx}u^R_xv^R_x}{(v^R)^3}+\frac{2\theta^R_xu^R_{xxx}}{(v^R)^2}-\frac{5\theta^R_xu^R_{xx}v^R_x+2\theta^R_xu^R_xv^R_{xx}}{(v^R)^3}\\
			&\quad+\frac{8\theta^R_xu^R_x(v^R_x)^2}{(v^R)^4}+\frac{\theta^Ru^R_{xxxx}}{(v^R)^2}-\frac{3\theta^Ru^R_{xxx}v^R_x}{(v^R)^3}-\frac{\theta^Ru^R_{xx}v^R_{xx}}{(v^R)^3}+\frac{5\theta^Ru^R_{xx}(v^R_x)^2}{(v^R)^4}\\
			&\quad+\frac{4\theta^Ru^R_xv^R_xv^R_{xx}}{(v^R)^4}-\frac{8\theta^Ru^R_x(v^R_x)^3}{(v^R)^5}\bigg),
		\end{align*}
		thus, 
		\begin{align}
			\int_{0}^{t}R_{14}\mathrm{d}t&=-\int_{0}^{t}\int\tau_{1}q^R_{txx}\frac{v^2}{\kappa R\theta^2}\tilde{q}_{xx}\mathrm{d}x\mathrm{d}t\nonumber\\
			&\leq C\int_{0}^{t}\big(\|\theta^R_{xxx}\|_{L^4}\| u^R_x\|_{L^4}+\|\theta^R_{xx}\|_{L^4}\| u^R_x\|_{L^4}+\|\theta^R_{xx}\|_{L^6}\| u^R_x\|_{L^6}\| v^R_x\|_{L^6}\nonumber\\
			&\qquad\qquad+\| \theta^R_{x}\|_{L^8}\| v^R_{x}\|_{L^8}^2\| u^R_{x}\|_{L^8}+\| u^R_{xxxx}\|_{L^2}\big)\| \tilde{q}_{xx}\|_{L^2}\mathrm{d}t\nonumber\\
			&\leq C\delta_{R}^{\frac{3}{2}-\frac{1}{q}}+C\delta_{R}^{\frac{1}{2}}.
		\end{align}
		Similarly, we have
		\begin{equation*}
			\begin{aligned}
				\int_{0}^{t}R_{15}\mathrm{d}t=-\int_{0}^{t}\int\tau_{2}S^R_{txx}\frac{v^2}{\mu R\theta}\tilde{S}_{xx}\mathrm{d}x\mathrm{d}t
				\leq C\delta_{R}^{\frac{3}{2}-\frac{1}{q}}+C\delta_{R}^{\frac{1}{2}}.
			\end{aligned}
		\end{equation*}
		Finally, for the last term $R_{16}$, we derive that
		\begin{align}
			\int_{0}^{t} R_{16} \mathrm{d}t&=\int_{0}^{t}\int\bigg[\left(\frac{v_x}{\theta}-\frac{v\theta_x}{\theta^2}\right)\tilde{\theta}_{xx}\psi_{xx}-\frac{v}{\theta R} \left(2v_x-\frac{v\theta_x}{\theta}\right)\psi_{xx}\tilde{S}_{xx}\nonumber\\& \qquad+\frac{2v}{R\theta^2}\left(v_x-\frac{v\theta_x}{\theta}\right)\tilde{q}_{xx}\tilde{\theta}_{xx}\bigg]\mathrm{d}x\mathrm{d}t
			\leq C E^{\frac{1}{2}}(t) \int_{0}^{t}\mathcal{D}(t) \mathrm{d} t+C\delta_{R} \int_{0}^{t} \mathcal{D}(t) \mathrm{d} t.
		\end{align}
		
		In summary, combining all the above estimates with \eqref{3.82}, we obtain 
		\begin{align}
			& \left\|\varphi_{x x}\right\|_{L^{2}}^{2}+\left\|\varphi_{x x}\right\|_{L^{2}}^{2}+\left\|\tilde{\theta}_{x x}\right\|_{L^{2}}^{2}+\left\|\tilde{q}_{x x}\right\|_{L^{2}}^{2}+\left\|\tilde{S}_{x x}\right\|_{L^{2}}^{2}+\int_{0}^{t}\left(\left\|\tilde{q}_{x x}\right\|_{L^{2}}^{2}+\left\|\tilde{S}_{x x}\right\|_{L^{2}}^{2}\right)\mathrm{d}t \nonumber\\
			\leq &C \left(E_{0}+ E^{\frac{1}{2}}(t) \int_{0}^{t} \mathcal{D} (t) \mathrm{d} t+\delta_R\int_{0}^{t} \mathcal{D}(t)  \mathrm{d} t+\delta_{R}^{\frac{1}{2}}+\delta_{R}^{\frac{2}{3}- \frac{1}{q}}\right).
		\end{align}
		
		Therefore, the proof of this lemma is completed.
	\end{proof}
	
	\begin{lemma}
		There exists a constant $C$ such that for any $0\le t\le T$, we have
		\begin{align}
			\int_{0}^{t}\int\left(\varphi_{x}^2+\tilde{\theta}_{x}^2+\psi_{x}^2\right)\mathrm{d}x\mathrm{d}t&\leq C\left(E_0+E^\frac{1}{2}(t)\int_{0}^{t}\mathcal{D}(t)\mathrm{d}t+\delta_{R}\int_{0}^{t}\mathcal{D}(t)\mathrm{d}t+\delta_R^{\frac{1}{2}}+\delta_R^{\frac{3}{2}-\frac{1}{q}}\right).
		\end{align}
	\end{lemma}
	\begin{proof}
		Multiplying $(\ref{3.48})_{4}$ by $\tilde{\theta}_{x}$, and integrating the results over $[0,t]\times\mathbb{R}$, we can obtain
		\begin{equation}
			\begin{aligned}\label{4.62}
				\int_{0}^{t}\int\kappa\tilde{\theta}_x^2\mathrm{d}x\mathrm{d}t= \int_{0}^{t}\int\left(-	\tau_1q^R_t\tilde{\theta}_x-\tau_1\tilde{q}_t\tilde{\theta}_x - vq\tilde{\theta}_x+v^{R}q^{R}\tilde{\theta}_x\right)\mathrm{d}x\mathrm{d}t.
			\end{aligned}
		\end{equation}
		We now estimate the terms on the right‑hand side of \eqref{4.62}. First, from equation \eqref{3.9}, we derive
		$$q^R_t=\kappa(\gamma-1)\left(\frac{\theta^R_xu^R_x}{(v^R)^2}+\frac{\theta^Ru^R_{xx}}{(v^R)^2}-\frac{\theta^Ru^R_xv^R_x}{(v^R)^3}\right)+\frac{\kappa\theta^R_xu^R_x}{(v^R)^2}.$$
		Using the above expression and $\epsilon$-Young inequality, we get
		\begin{align*}
			\int_{0}^{t}\int|-\tau_1q^R_t\tilde{\theta}_x|\mathrm{d}x\mathrm{d}t&\leq\int_{0}^{t}\int\epsilon(-\tau_1\tilde{\theta})^2\mathrm{d}x\mathrm{d}t+C(\epsilon)\int_{0}^{t}\int(q^R_t)^2\mathrm{d}x\mathrm{d}t\\
			&\leq\int_{0}^{t}\int\epsilon(-\tau_1\tilde{\theta}_x)^2\mathrm{d}x\mathrm{d}t+C(\delta_R^{3-\frac{2}{q}}+\delta_R).
		\end{align*}	
		Second, using
		\begin{align*}
			&\tilde{\theta}_t=\frac{\gamma-1}{R}(\frac{q^2v}{\kappa\theta}+\frac{S^2v}{\mu}-Q_2^R+\frac{\tau_1q^2\theta_t}{2\kappa\theta^2}+\frac{q\theta_x}{\theta}-pu_x+p^Ru^R_x-\tilde{q}_x),\\\nonumber
			&p^Ru_x^R-pu_x=\frac{R(v^R\tilde{\theta}-\theta^R\varphi)u^R_x}{vv^R}-\frac{R\theta}{v}\psi_x,\nonumber
		\end{align*}
		we obtain the bound 
		\begin{align}\label{4.64}
			\int_{0}^{t}\int(\tilde{\theta}_t)^2\mathrm{d}x\mathrm{d}t&\leq C	\int_{0}^{t}\int\left(\frac{v}{\kappa\theta}(\tilde{q}^2+2q^R\tilde{q}+(q^R)^2)\right)^2+\left(\frac{v}{\mu}(\tilde{S}^2+2S^R\tilde{S}+(S^R)^2)\right)^2\mathrm{d}x\mathrm{d}t\nonumber\\
			&\quad+	\int_{0}^{t}\int\left(-\frac{\kappa\theta^R_{xx}v^R+\kappa\theta^R_xv^R_x}{(v^R)^2}\right)^2+\frac{(\tilde{q}^2+2\tilde{q}q^R+(q^R)^2)^2\theta^2_t}{\theta^4}\mathrm{d}x\mathrm{d}t\nonumber\\
			&\quad+	\int_{0}^{t}\int\frac{(\tilde{q}^2+(q^R)^2)((\theta^R_x)^2+\tilde{\theta}^2)}{\theta^2}+(p^Ru^R_x-pu_x)^2+\tilde{q}_x^2\mathrm{d}x\mathrm{d}t\nonumber\\
			&\leq CE^\frac{1}{2}(t)\int_{0}^{t}\mathcal{D}(t)\mathrm{d}t+C\delta_{R}\int_{0}^{t}\mathcal{D}(t)\mathrm{d}t+C(\delta_R^{3-\frac{2}{q}}+\delta_R).
		\end{align}
		Next, employing the Cauchy inequality, equation \ref{4.64} and Lemma \ref{lem4.2}, we have
		\begin{align*}
			\int_{0}^{t}\int-\tau_1\tilde{q}_t\tilde{\theta}_x\mathrm{d}x\mathrm{d}t&=-\int\int_{0}^{t}\frac{d}{\mathrm{d}t}(\tau_1\tilde{q}\tilde{\theta}_x)\mathrm{d}t\mathrm{d}x+\int_{0}^{t}\int\tau_{1}\tilde{q}\tilde{\theta}_{tx}\mathrm{d}x\mathrm{d}t\nonumber\\
			&\geq -\frac{1}{2}\int\tilde{\theta}_x^2|_{0}^{t}\mathrm{d}x-\frac{1}{2}\int(\tau_{1}\tilde{q})^2|_{0}^{t}\mathrm{d}x-\int_{0}^{t}\int(\tau_1\tilde{q}_x)^2\mathrm{d}x\mathrm{d}t-\int_{0}^{t}\int(\tilde{\theta}_t)^2\mathrm{d}x\mathrm{d}t\nonumber\\
			&\geq C\int\tilde{q}^2|_{0}^{t}\mathrm{d}x
			-CE^\frac{1}{2}(t)\int_{0}^{t}\mathcal{D}(t)\mathrm{d}t-C\delta_{R}\int_{0}^{t}\mathcal{D}(t)\mathrm{d}t
			-C(\delta_R^{3-\frac{2}{q}}+\delta_R).
		\end{align*}
		For the last terms, we have
		\begin{equation}
			\begin{aligned}\label{4.72}
				\int_{0}^{t}\int(v^Rq^R-vq)\tilde{\theta}_x\mathrm{d}x\mathrm{d}t&\leq\epsilon\int_{0}^{t}\int\tilde{\theta}_x^2\mathrm{d}x\mathrm{d}t+C\int_{0}^{t}\int\left((v^R)^2\tilde{q}^2+\varphi^2q^2\right)\mathrm{d}x\mathrm{d}t\\
				&\leq\epsilon\int_{0}^{t}\int\tilde{\theta}_x^2\mathrm{d}x\mathrm{d}t+C\delta_{R}\int_{0}^{t}\mathcal{D}(t)\mathrm{d}t.
			\end{aligned}
		\end{equation}
		Thus, by combining the above estimates and Lemma \ref{lem4.1}, we obtain
		\begin{equation}
			\begin{aligned}\label{4.63}
				\int_{0}^{t}\int\tilde{\theta}_x^2\mathrm{d}x\mathrm{d}t	&\leq C\left(E_0+E^{\frac{1}{2}}(t)\int_{0}^{t}\mathcal{D}(t)\mathrm{d}t+\delta_{R}\int_{0}^{t}\mathcal{D}(t)\mathrm{d}t+\delta_{R}^{\frac{3}{2}-\frac{1}{q}}+\delta_{R}^\frac{1}{2}\right).
			\end{aligned}
		\end{equation}
		
		Multiplying $(\ref{3.48})_{2}$ by $\varphi_{x}$ and integrating the results over $[0,t]\times\mathbb{R}$, we can obtain
		\begin{equation}
			\begin{aligned}\label{4.50}
				\int_{0}^{t}\int\frac{R\theta}{v^2}\varphi_x^2\mathrm{d}x\mathrm{d}t
				=\int_{0}^{t}\int\bigg(\psi_t\varphi_x+\frac{R}{v}(\tilde{\theta}_x+\theta^R_x)\varphi_x-\frac{R\theta}{v^2}v^R_x\varphi_x-\frac{R\theta^R_x}{v^R}\varphi_x\\
				+\frac{R\theta^Rv^R_x}{(v^R)^2}\varphi_x- \tilde{S}_x\varphi_x+Q_{1}^{R}\varphi_x\bigg)\mathrm{d}x\mathrm{d}t.
			\end{aligned}
		\end{equation}
		Next, we estimate each term on the right-hand side of equation \eqref{4.50}. First, using $\epsilon$-Young inequality and equation $\eqref{3.48}_1$, we can obtain
		\begin{equation}
			\begin{aligned}
				\int_{0}^{t}\int\psi_t\varphi_x\mathrm{d}x\mathrm{d}t&=\int_{0}^{t}\int\frac{d}{\mathrm{d}t}(\psi\varphi_x)\mathrm{d}x\mathrm{d}t-\int_{0}^{t}\int\psi\varphi_{xt}\mathrm{d}x\mathrm{d}t\\
				&\leq\int(C(\epsilon)\psi^2+\epsilon\varphi_x^2)|_{0}^{t}\mathrm{d}x+\int_{0}^{t}\int\psi^2_x\mathrm{d}x\mathrm{d}t\\
			\end{aligned}
		\end{equation}
		For the other terms in \eqref{4.50}, we  can similarly obtain
		\begin{equation*}
			\begin{aligned}
				\int_{0}^{t}\int\frac{R}{v}\tilde{\theta}_x\varphi_x\mathrm{d}x\mathrm{d}t\leq \int_{0}^{t}\int\epsilon\varphi_x^2\mathrm{d}x\mathrm{d}t+C(\epsilon)\int_{0}^{t}\int\frac{R^2}{v^2}\tilde{\theta}^2_x\mathrm{d}x\mathrm{d}t,
			\end{aligned}
		\end{equation*}
		\begin{align*}
			\int_{0}^{t}\int\left(\frac{R}{v}\theta^R_x\varphi_x-\frac{R\theta^R_x}{v^R}\varphi_x\right)\mathrm{d}x\mathrm{d}t&=\int_{0}^{t}\int\frac{-R\theta^R_x\varphi_x\varphi}{vv^R}\mathrm{d}x\mathrm{d}t\\
			&\geq-\int_{0}^{t}\int\epsilon\left(\frac{-R\varphi_x}{vv^R}\right)^2\mathrm{d}x\mathrm{d}t-C\int_{0}^{t}\int(\varphi\theta^R_x)^2\mathrm{d}x\mathrm{d}t\\
			&\geq-\int_{0}^{t}\int\epsilon\left(\frac{-R\varphi_x}{vv^R}\right)^2\mathrm{d}x\mathrm{d}t-C\delta^2_R\int_{0}^{t}\mathcal{D}(t)\mathrm{d}t,
		\end{align*}
		\begin{align*}
			\int_{0}^{t}\int\left(-\frac{R\theta}{v^2}v^R_x\varphi_x+\frac{R\theta^Rv^R_x}{(v^R)^2}\varphi_x\right)\mathrm{d}x\mathrm{d}t&=\int_{0}^{t}\int\frac{R\varphi_xv^R_x(\tilde{\theta}(v^R)^2-2\theta^Rv^R\varphi-\theta^R\varphi^2)}{v^2(v^R)^2}\mathrm{d}x\mathrm{d}t\\
			&\leq C\int_{0}^{t}\int(|v^R_x|+|\theta^R|+|v^R|)(|\tilde{\theta}|^2+|\varphi_x|^2+|\varphi|^2)\mathrm{d}x\mathrm{d}t\\
			&\leq C\delta_R\int_{0}^{t}\mathcal{D}(t)\mathrm{d}t,
		\end{align*}
		and
		\begin{align*}
			\int_{0}^{t}\int|-\tilde{S}_x\varphi_x|\mathrm{d}x\mathrm{d}t\leq \int_{0}^{t}\int\epsilon\varphi_x^2\mathrm{d}x\mathrm{d}t+C(\epsilon) \int_{0}^{t}\int\tilde{S}_x^2\mathrm{d}x\mathrm{d}t.
		\end{align*}
		For the last term on the right-hand side of \eqref{4.50}, we get
		\begin{align*}
			\int_{0}^{t}\int| Q_1^R\varphi_x|\mathrm{d}x\mathrm{d}t&= \int_{0}^{t}\int|-\left(\frac{\mu u^R_x}{v^R}\right)_x\varphi_x|\mathrm{d}x\mathrm{d}t\\
			&\leq  \int_{0}^{t}\int\epsilon\varphi_x^2\mathrm{d}x\mathrm{d}t+C(\epsilon) \int_{0}^{t}\int\left(\frac{\mu u^R_x}{v^R}\right)_x^2\mathrm{d}x\mathrm{d}t\\
			&\leq \int_{0}^{t}\int\epsilon\varphi_x^2\mathrm{d}x\mathrm{d}t+C(\delta_R^{3-\frac{2}{q}}+\delta_R).
		\end{align*}
		Thus, by combining the above estimates and using \eqref{4.63}, Lemmas \ref{lem4.1}, \ref{lem4.2}, we obtain
		\begin{align}\label{4.58}
			\int_{0}^{t}\int\varphi_x^2\mathrm{d}x\mathrm{d}t	&\leq C\left(E_0+E^{\frac{1}{2}}(t)\int_{0}^{t}\mathcal{D}(t)\mathrm{d}t+\delta_{R}\int_{0}^{t}\mathcal{D}(t)\mathrm{d}t+\delta_{R}^{\frac{3}{2}-\frac{1}{q}}+\delta_{R}^\frac{1}{2}\right)+\int_{0}^{t}\int\psi^2_x\mathrm{d}x\mathrm{d}t.
		\end{align}
		
		Multiplying $(\ref{3.48})_{5}$ by $\psi_{x}$ and integrating the results over $[0,t]\times\mathbb{R}$ yields
		\begin{equation}
			\begin{aligned}\label{68}
				\int_{0}^{t}\int\mu\psi_x^2\mathrm{d}x\mathrm{d}t= \int_{0}^{t}\int\left(-	\tau_2S^R_t\psi_{x}-\tau_2\tilde{S}_t\psi_{x} - vS\psi_{x}+v^{R}S^{R}\psi_{x}\right)\mathrm{d}x\mathrm{d}t.
			\end{aligned}
		\end{equation}
		We now estimate each term on the right-hand side of equation \eqref{68}. 	First, using the expression
		$$S^R_t=-R\mu\left(\frac{\theta_{xx}^R}{(v^R)^2}-\frac{2\theta^R_xv^R_x}{(v^R)^3}-\frac{\theta^Rv^R_{xx}}{(v^R)^3}+\frac{2\theta^R(v^R_x)^2}{(v^R)^4}\right)-\frac{\mu(u^R_x)^2}{(v^R)^2},$$
		we obtain
		\begin{equation}
			\begin{aligned}
				\int_{0}^{t}\int|-\tau_2S^R_t\psi_{x}|\mathrm{d}x\mathrm{d}t&\leq\int_{0}^{t}\int\epsilon(-\tau_2\psi_x)^2\mathrm{d}x\mathrm{d}t+C(\epsilon)\int_{0}^{t}\int(S^R_t)^2\mathrm{d}x\mathrm{d}t\\
				&\leq \epsilon	\int_{0}^{t}\int\tau_{2}\psi_{x}^2\mathrm{d}x\mathrm{d}t+C(\delta_R^{3-\frac{2}{q}}+\delta_R).
			\end{aligned}
		\end{equation}
		Second, from equation $(\ref{3.48})_{2}$, we have
		\begin{align*}
			\psi_t&=\tilde{S}_x-Q^R_1-\frac{R}{v}(\tilde{\theta}_x+\theta^R_x)+\frac{R\theta}{v^2}(\varphi_x+v^R_x)+\frac{R\theta^R_x}{v^R}-\frac{R\theta^Rv^R_x}{(v^R)^2}\\
			&=\tilde{S}_x-Q^R_1-\frac{R}{v}\tilde{\theta}_x+\frac{R\theta^R_x\varphi}{vv^R}+\frac{R\theta}{v^2}\varphi_x+\frac{Rv^R_x(\tilde{\theta}(v^R)^2)-\theta^R\varphi(2v^R+\varphi)}{v^2(v^R)^2},
		\end{align*}
		and consequently
		\begin{equation*}
			\begin{aligned}
				\int_{0}^{t}\int\psi_{t}^2\mathrm{d}x\mathrm{d}t\leq C\int_{0}^{t}\int\left(\tilde{\theta}^2_x+\varphi_{x}^2\right)\mathrm{d}x\mathrm{d}t+CE^\frac{1}{2}(t)\int_{0}^{t}\mathcal{D}(t)\mathrm{d}t+C\delta_R	\int_{0}^{t}\mathcal{D}(t)\mathrm{d}t+C(\delta_R^{3-\frac{2}{q}}+\delta_R).
			\end{aligned}
		\end{equation*}
		Hence,
		\begin{align}
			\int_{0}^{t}\int-\tau_2\tilde{S}_t\psi_{x}\mathrm{d}x\mathrm{d}t&=	-\int\int_{0}^{t}\frac{d}{\mathrm{d}t}(\tau_2\tilde{S}\psi_{x})\mathrm{d}t\mathrm{d}x+\int\int_{0}^{t}\tau_2\tilde{S}\psi_{tx}\mathrm{d}x\mathrm{d}t \nonumber\\
			&\geq-\int\tau_2(\tilde{S}^2+\epsilon(\psi_{x})^2)|_{0}^{t}\mathrm{d}x-	C(\epsilon)\int_{0}^{t}\int\tau_2\tilde{S}_x^2\mathrm{d}x\mathrm{d}t-\epsilon	\int_{0}^{t}\int\psi_t^2\mathrm{d}x\mathrm{d}t\nonumber\\
			&\geq -\int\tau_2(\tilde{S}^2+\epsilon(\psi_{x})^2)|_{0}^{t}\mathrm{d}x -C\delta_{R}\int_{0}^{t}\mathcal{D}(t)\mathrm{d}t-C\epsilon\int_{0}^{t}\int\left(\tilde{\theta}^2_x+\varphi_{x}^2\right)\mathrm{d}x\mathrm{d}t\nonumber\\
			&\quad-C(\delta_R^{3-\frac{2}{q}}+\delta_R).
		\end{align}
		Next, we have
		\begin{equation}
			\begin{aligned}
				\int_{0}^{t}\int|(v^RS^R-vS)\psi_x|\mathrm{d}x\mathrm{d}t&=\int_{0}^{t}\int|(v^R\tilde{S}+\varphi S)\psi_x|\mathrm{d}x\mathrm{d}t\\
				&\leq \epsilon\int_{0}^{t}\int\psi_x^2\mathrm{d}x\mathrm{d}t+C(\epsilon)E^{\frac{1}{2}}(t)\int_{0}^{t}\mathcal{D}(t)\mathrm{d}t.
			\end{aligned}
		\end{equation}
		Thus, by combining the above estimates and using \eqref{4.63}, Lemma \ref{lem4.2}, we obtain
		\begin{align}\label{4.68}
			\int_{0}^{t}\int\psi_x^2\mathrm{d}x\mathrm{d}t\leq 
			C\left(E_0+E^{\frac{1}{2}}(t)\int_{0}^{t}\mathcal{D}(t)\mathrm{d}t+\delta_{R}\int_{0}^{t}\mathcal{D}(t)\mathrm{d}t+\delta_{R}^{\frac{3}{2}-\frac{1}{q}}+\delta_{R}^\frac{1}{6}\right)
			+\epsilon\int_{0}^{t}\int\varphi_x^2\mathrm{d}x\mathrm{d}t.
		\end{align}
		In summary, combining \eqref{4.58}, \eqref{4.63} and \eqref{4.68}, we obtain
		\begin{align*}
			\int_{0}^{t}\int\left(\varphi_{x}^2+\tilde{\theta}_{x}^2+\psi_{x}^2\right)\mathrm{d}x\mathrm{d}t&\leq C\left(E_0+E^\frac{1}{2}(t)\int_{0}^{t}\mathcal{D}(t)\mathrm{d}t+\delta_R\int_{0}^{t}\mathcal{D}(t)\mathrm{d}t+\delta_R^{\frac{1}{2}}+\delta_R^{\frac{3}{2}-\frac{1}{q}}\right).
		\end{align*}
		This completes the proof of the Lemma.
	\end{proof}
	
	Next, we establish the estimates for the second-order dissipative terms.
	\begin{lemma}\label{lem4.5}
		There exists a constant $C$ such that for any $0\le t\le T$, we have
		\begin{align}
			\int_{0}^{t}\int\left(\varphi_{xx}^2+\tilde{\theta}_{xx}^2+\psi_{xx}^2\right)\mathrm{d}x\mathrm{d}t&\leq C\left(E_0+E^\frac{1}{2}(t)\int_{0}^{t}\mathcal{D}(t)\mathrm{d}t+\delta_R\int_{0}^{t}\mathcal{D}(t)\mathrm{d}t+\delta_R^{\frac{1}{2}}+\delta_R^{\frac{3}{2}-\frac{1}{q}}\right).
		\end{align}
	\end{lemma}
	\begin{proof}
		Multiplying $(\ref{3.51})_{4}$ by $\tilde{\theta}_{xx}$ and integrating the results over $[0,t]\times\mathbb{R}$, we can obtain
		\begin{equation}
			\begin{aligned}\label{4.100}
				\int_{0}^{t}\int\kappa\tilde{\theta}_{xx}^2\mathrm{d}x\mathrm{d}t= \int_{0}^{t}\int\left(-	\tau_1q^R_{tx}\tilde{\theta}_{xx}-\tau_1\tilde{q}_{tx}\tilde{\theta}_{xx} - (vq+v^{R}q^{R})_x\tilde{\theta}_{xx}\right)\mathrm{d}x\mathrm{d}t,
			\end{aligned}
		\end{equation}
		We now estimate each term appearing on the right-hand side of equation \eqref{4.100}. First, by using $\epsilon$-young inequality and $(\ref{3.9})$, we know
		\begin{align}
			\int_{0}^{t}\int|-\tau_1q^R_{tx}\tilde{\theta}_{xx}|\mathrm{d}x\mathrm{d}t&\leq\int_{0}^{t}\int\epsilon(-\tau_1\tilde{\theta}_{xx})^2\mathrm{d}x\mathrm{d}t+C(\epsilon)\int_{0}^{t}\int(q^R_{tx})^2\mathrm{d}x\mathrm{d}t\nonumber\\
			&\leq\int_{0}^{t}\int\epsilon(-\tau_1\tilde{\theta}_{xx})^2\mathrm{d}x\mathrm{d}t+C\int_{0}^{t}\int |\theta^R_{xx}u^R_{x}+\theta^R_{x}u^R_{xx}+\theta^R_{x}u^R_{x}v^R_{x}\nonumber\\
			&\quad+\theta^Ru^R_{xxx}+\theta^Ru^R_{xx}v^R_{x}+\theta^Ru^R_{x}v^R_{xx}+\theta^Ru^R_{x}(v^R_{x})^2|^2\mathrm{d}x\mathrm{d}t\nonumber\\
			&\leq\int_{0}^{t}\int\epsilon(-\tau_1\tilde{\theta}_{xx})^2\mathrm{d}x\mathrm{d}t+C(\delta_R^{5-\frac{2}{q}}+\delta_R).
		\end{align}
		Second, by invoking equation $\eqref{3.51}_3$, we have
		\begin{align*}
			\tilde{\theta}_{tx}&=\left[\frac{\tau_{1}(\tilde{q}+q^R)^2}{2\kappa(\tilde{\theta}+\theta^R)^2}\tilde{\theta}_{t}\right]_{x}+\left[\frac{(\tilde{\theta}+\theta^R)_{x}(\tilde{q}+q^R)}{\tilde{\theta}+\theta^R}\right]_{x}-
			\left(-\frac{R\theta}{v^2}(\varphi_x+v^R_x)+\frac{R}{v}(\tilde{\theta}_x+\theta^R_x)\right)\psi_{x}\nonumber\\
			&+\frac{R\theta}{v}\psi_{xx}+(p-p^R)_{x}u^R_x+(p-p^R)u^R_{xx}+\tilde{q}_{xx}
			+\left[\frac{(\tilde{q}+q^R )^2v}{\kappa(\tilde{\theta}+\theta^R)}\right]_{x}\nonumber\\
			&+\left[\frac{v}{\mu}(\tilde{S}^2+2S^R\tilde{S}+(S^R)^2)\right]_{x}-(Q^R_{2})_{x}+\left[\frac{\tau_{1}(\tilde{q}+q^R)^2\theta^R_{t}}{2\kappa(\tilde{\theta}+\theta^R)^2}\right]_{x},
		\end{align*}
		which, together with Lemma \ref{lem4.3}, yields
		\begin{equation}
			\begin{aligned}
				\int_{0}^{t}\int\tilde{\theta}^2_{tx}\mathrm{d}x\mathrm{d}t\leq C\left(E(0)+E^{\frac{1}{2}}(t)\int_{0}^{t}\mathcal{D}(t)\mathrm{d}t+\delta_{R}\int_{0}^{t}\mathcal{D}(t)\mathrm{d}t+\delta_R^{3-\frac{2}{q}}+\delta_R\right).
			\end{aligned}
		\end{equation}
		Thus, we can derive
		\begin{align*}
			\int_{0}^{t}\int|-\tau_1\tilde{q}_{tx}\tilde{\theta}_{xx}|\mathrm{d}x\mathrm{d}t&\leq \int\int_{0}^{t}\frac{\mathrm{d}}{\mathrm{d}t}(\tau_1\tilde{q}_x\tilde{\theta}_{xx})\mathrm{d}x\mathrm{d}t-\int\int_{0}^{t}\tau_1\tilde{q}_x\tilde{\theta}_{xxt}\mathrm{d}x\mathrm{d}t\\
			&\leq\int\tau_1(C\tilde{q}_x^2+\epsilon\theta_{xx}^2)|_0^t\mathrm{d}x+C\int_{0}^{t}\int\tilde{q}_{xx}^2\mathrm{d}x\mathrm{d}t+C\int_{0}^{t}\int\tilde{\theta}_{tx}^2\mathrm{d}x\mathrm{d}t\\
			&\leq CE^{\frac{1}{2}}(t)\int_{0}^{t}\mathcal{D}(t)\mathrm{d}t+C\delta_{R}\int_{0}^{t}\mathcal{D}(t)\mathrm{d}t+C(\delta_R^{3-\frac{2}{q}}+\delta_R).
		\end{align*}
		We now proceed the last term on the right-hand side of \eqref{4.100}. Similarly with \eqref{4.72}, we have
		\begin{equation}
			\begin{aligned}
				&\int_{0}^{t}\int|- (vq+v^{R}q^{R})_x\tilde{\theta}_{xx}|\mathrm{d}x\mathrm{d}t\\
				&\leq \int_{0}^{t}\int\epsilon\tilde{\theta}_{xx}^2\mathrm{d}x\mathrm{d}t
				+C(\epsilon)\int_{0}^{t}\int (v_x^2\tilde{q}^2+v^2\tilde{q}_x^2+\varphi_x^2(q^R)^2+\varphi^2(q^R_x)^2)\mathrm{d}x\mathrm{d}t\\
				&\leq\int_{0}^{t}\int\epsilon\tilde{\theta}_{xx}^2\mathrm{d}x\mathrm{d}t+CE^\frac{1}{2}(t)\int_{0}^{t}\mathcal{D}(t)\mathrm{d}t
				+C\delta_{R}\int_{0}^{t}\mathcal{D}(t)\mathrm{d}t+C(\delta_R^{3-\frac{2}{q}}+\delta_R).
			\end{aligned}
		\end{equation}
		Thus, by combining the above estimates, one can obtain
		\begin{equation}
			\begin{aligned}\label{96}
				\int_{0}^{t}\int\tilde{\theta}_{xx}^2\mathrm{d}x\mathrm{d}t&\leq C\left(E_0+E^\frac{1}{2}\int_{0}^{t}\mathcal{D}(t)\mathrm{d}t+\delta_R\int_{0}^{t}\mathcal{D}(t)\mathrm{d}t+\delta_R^{\frac{1}{2}}+\delta_R^{\frac{3}{2}-\frac{1}{q}}\right).
			\end{aligned}
		\end{equation}
		
		Multiplying $(\ref{3.51})_{2}$ by $\varphi_{xx}$ and integrating the results over $[0,t]\times\mathbb{R}$, we can obtain
		\begin{equation}
			\begin{aligned}\label{4.90}
				\int_{0}^{t}\int\frac{R\theta}{v^2}\varphi^2_{xx}\mathrm{d}x\mathrm{d}t&=\int_{0}^{t}\int	\bigg(\psi_{tx}\varphi_{xx}-\left(\frac{R\theta_x}{v^2}-\frac{2R\theta v_x}{v^3}\right)(\varphi_x+v^R_x)\varphi_{xx}-\frac{R\theta}{v^2}v^R_{xx}\varphi_{xx}\\
				&\quad-\frac{Rv_x}{v^2}(\tilde{\theta}_x+\theta^R_x)\varphi_{xx}+\frac{R}{v}(\tilde{\theta}_{xx}+\theta^R_{xx})\varphi_{xx}+\frac{2R\theta_x^Rv_x^R+R\theta^Rv^R_{xx}}{(v^R)^2}\varphi_{xx}\\
				&\quad-\frac{2R\theta^R(v_x^R)^2}{(v^R)^3}\varphi_{xx}-\frac{R\theta_{xx}^R}{v^R}\varphi_{xx}
				-\tilde{S}_{xx}\varphi_{xx}+(Q^R_{1})_{x}\varphi_{xx}\bigg)\mathrm{d}x\mathrm{d}t.
			\end{aligned}
		\end{equation}
		Next, we estimate each term on the right-hand side of equation \eqref{4.90}. First, using $\epsilon$-young inequality, we have
		\begin{align}
			\int_{0}^{t}\int\psi_{tx}\varphi_{xx}\mathrm{d}x\mathrm{d}t&=\int\int_{0}^{t}\frac{d}{\mathrm{d}t}(\psi_{x}\varphi_{xx})\mathrm{d}x\mathrm{d}t+\int_{0}^{t}\int\psi_{xx}^2\mathrm{d}x\mathrm{d}t\nonumber\\
			&\leq\int(C(\epsilon)\psi_{x}^2+\epsilon\varphi_{xx}^2)|_{0}^{t}\mathrm{d}x+\int_{0}^{t}\int\psi_{xx}^2\mathrm{d}x\mathrm{d}t,
		\end{align}
		and
		\begin{equation}
			\begin{aligned}
				&\int_{0}^{t}\int\bigg|-\left(\frac{R\theta_x}{v^2}-\frac{2R\theta v_x}{v^3}\right)(\varphi_x+v^R_x)\varphi_{xx}\bigg|\mathrm{d}x\mathrm{d}t\\
				&\leq	\int_{0}^{t}\int\epsilon\left(\frac{R}{v^2}\varphi_{xx}\right)^2\mathrm{d}x\mathrm{d}t+C(\epsilon)\int_{0}^{t}\int\left(\theta_x-\frac{2\theta v_x}{v}\right)^2(\varphi_x+v^R_x)^2\mathrm{d}x\mathrm{d}t\\
				&\leq\int_{0}^{t}\int\epsilon\left(\frac{R}{v^2}\varphi_{xx}\right)^2\mathrm{d}x\mathrm{d}t+CE^\frac{1}{2}(t)\int_{0}^{t}\mathcal{D}(t)\mathrm{d}t.
			\end{aligned}
		\end{equation}
		Furthermore, by using $\epsilon$-young inequality and Lemma \ref{lem2.1}, one can obtain
		\begin{equation}
			\begin{aligned}
				\int_{0}^{t}\int|-\frac{R\theta}{v^2}v^R_{xx}\varphi_{xx}|\mathrm{d}x\mathrm{d}t&\leq\int_{0}^{t}\int\epsilon\left(\frac{R\theta}{v^2}\varphi_{xx}\right)^2\mathrm{d}x\mathrm{d}t+C(\epsilon)\int_{0}^{t}\int(v^R_{xx})^2\mathrm{d}x\mathrm{d}t\\
				&\leq\int_{0}^{t}\int\epsilon\left(\frac{R\theta}{v^2}\varphi_{xx}\right)^2\mathrm{d}x\mathrm{d}t+C(\delta_R^{3-\frac{2}{q}}+\delta_R).
			\end{aligned}
		\end{equation}
		And similarly, we have
		\begin{align*}
			&\int_{0}^{t}\int|-\frac{Rv_x}{v^2}(\tilde{\theta}_x+\theta^R_x)\varphi_{xx}|\mathrm{d}x\mathrm{d}t\\
			&\leq\int_{0}^{t}\int\epsilon\left(-\frac{R}{v^2}\varphi_{xx}\right)^2\mathrm{d}x\mathrm{d}t+CE^\frac{1}{2}(t)\int_{0}^{t}\mathcal{D}(t)\mathrm{d}t+C\delta_R\int_{0}^{t}\mathcal{D}(t)\mathrm{d}t+\delta_R,
		\end{align*}
		\begin{align*}
			\int_{0}^{t}\int\frac{R}{v}\tilde{\theta}_{xx}\varphi_{xx}\mathrm{d}x\mathrm{d}t\leq\epsilon\int_{0}^{t}\int\left(\frac{R}{v}\varphi_{xx}\right)^2\mathrm{d}x\mathrm{d}t+C(\epsilon)\int_{0}^{t}\int\tilde{\theta}_{xx}^2\mathrm{d}x\mathrm{d}t,
		\end{align*}
		
		\begin{align*}
			\int_{0}^{t}\int\frac{R}{v}\theta^R_{xx}\varphi_{xx}\mathrm{d}x\mathrm{d}t\leq\epsilon\int_{0}^{t}\int\left(\frac{R}{v}\varphi_{xx}\right)^2\mathrm{d}x\mathrm{d}t+C(\delta_R^{3-\frac{2}{q}}+\delta_R),
		\end{align*}
		and
		\begin{align*}
			\int_{0}^{t}\int\frac{2R\theta_x^Rv_x^R+R\theta^Rv^R_{xx}}{(v^R)^2}\varphi_{xx}\mathrm{d}x\mathrm{d}t&\leq C\int_{0}^{t}\int\epsilon\left(\frac{R\varphi_{xx}}{(v^R)^2}\right)^2\mathrm{d}x\mathrm{d}t+\int_{0}^{t}\int(2\theta_x^Rv^R_x+\theta^Rv^R_{xx})^2\mathrm{d}x\mathrm{d}t\nonumber\\
			&\leq C\int_{0}^{t}\int\epsilon\left(\frac{R\varphi_{xx}}{(v^R)^2}\right)^2\mathrm{d}x\mathrm{d}t+C(\delta_R^{3-\frac{2}{q}}+\delta_R).
		\end{align*}
		For the last term, we have
		\begin{equation}
			\begin{aligned}
				&\int_{0}^{t}\int|-\frac{2R\theta^R(v_x^R)^2}{(v^R)^3}\varphi_{xx}-\frac{R\theta_{xx}^R}{v^R}\varphi_{xx}|\mathrm{d}x\mathrm{d}t\\
				&\leq \int_{0}^{t}\int\epsilon\left(\frac{R}{v^R}\varphi_{xx}\right)^2\mathrm{d}x\mathrm{d}t+\int_{0}^{t}\int C(\epsilon)\left(\frac{2\theta^R(v^R_x)^2}{(v^R)^2}+\theta^R_{xx}\right)^2\mathrm{d}x\mathrm{d}t\\
				&\leq \int_{0}^{t}\int\epsilon\left(\frac{R}{v^R}\varphi_{xx}\right)^2\mathrm{d}x\mathrm{d}t+C(\delta_R^{3-\frac{2}{q}}+\delta_R).
			\end{aligned}
		\end{equation}
		Thus, by combining the above estimates and using \eqref{96}, one can obtain
		\begin{align}\label{106}
			\int_{0}^{t}\int\varphi_{xx}^2\mathrm{d}x\mathrm{d}t&\leq C\left(E_0+E^\frac{1}{2}(t)\int_{0}^{t}\mathcal{D}(t)\mathrm{d}t+\delta_R\int_{0}^{t}\mathcal{D}(t)\mathrm{d}t+\delta_R^{\frac{1}{2}}+\delta_R^{\frac{3}{2}-\frac{1}{q}}\right)+\int_{0}^{t}\int\psi_{xx}^2\mathrm{d}x\mathrm{d}t.
		\end{align}

		Multiplying $(\ref{3.51})_{5}$ by $\psi_{xx}$ and integrating the results over $[0,t]\times\mathbb{R}$, we obtain
		\begin{equation}
			\begin{aligned}\label{4.107}
				\int_{0}^{t}\int\mu\psi_{xx}^2\mathrm{d}x\mathrm{d}t=	\int_{0}^{t}\int\tau_{2}\tilde{S}_{tx}\psi_{xx}+(v\tilde{S}+S^R\varphi)_{x}\psi_{xx}+\tau_{2}S^R_{tx}\psi_{xx}\mathrm{d}x\mathrm{d}t.
			\end{aligned}
		\end{equation}
		Next, we estimate each term on the right-hand side of equation \eqref{4.107}. First, we note that
		\begin{align*}
			\psi_{tx}&=\left(\frac{R\theta_x}{v^2}-\frac{2R\theta v_x}{v^3}\right)(\varphi_x+v^R_x)+\frac{R\theta}{v^2}(\varphi_{xx}+v^R_{xx})+\frac{Rv_x}{v^2}(\tilde{\theta}_x+\theta^R_x)-\frac{R}{v}(\tilde{\theta}_{xx}+\theta^R_{xx})\\
			&\quad-\frac{2R\theta_x^Rv_x^R+R\theta^Rv^R_{xx}}{(v^R)^2}+\frac{2R\theta^R(v_x^R)^2}{(v^R)^3}+\frac{R\theta_{xx}^R}{v^R}
			+\tilde{S}_{xx}-(Q^R_{1})_{x},
		\end{align*}
		and
		\begin{equation}
			\begin{aligned}
				\int_{0}^{t}\int\psi_{tx}^2\mathrm{d}x\mathrm{d}t&\leq \int_{0}^{t}\int\left(\tilde{\theta}_{xx}^2+\varphi_{xx}^2\right)\mathrm{d}x\mathrm{d}t+C\left(E^\frac{1}{2}(t)\int_{0}^{t}\mathcal{D}(t)\mathrm{d}t
				+\delta_{R}\int_{0}^{t}\mathcal{D}(t)\mathrm{d}t+\delta_R^{3-\frac{2}{q}}+\delta_R\right).
			\end{aligned}
		\end{equation}
		Then, applying the $\epsilon$-Young inequality, estimate \eqref{96} and Lemmas \ref{lem4.2}-\ref{lem4.3}, we derive
		\begin{align}
			\int_{0}^{t}\int\tau_{2}\tilde{S}_{tx}\psi_{xx}\mathrm{d}x\mathrm{d}t&=\int\int_{0}^{t}\frac{d}{\mathrm{d}t}(\tau_2\tilde{S}_x\psi_{xx})\mathrm{d}x\mathrm{d}t-\int_{0}^{t}\int\tau_2\tilde{S}\psi_{txx}\mathrm{d}x\mathrm{d}t \nonumber\\
			&\leq\int\tau_2(C\tilde{S}_x^2+\epsilon\psi_{xx}^2)|_{0}^{t}\mathrm{d}x+C(\epsilon)\int_{0}^{t}\int\tilde{S}_{xx}^2\mathrm{d}x\mathrm{d}t+\epsilon\int_{0}^{t}\int\psi_{tx}^2\mathrm{d}x\mathrm{d}t \nonumber\\
			&\leq \epsilon\int_{0}^{t}\int\varphi_{xx}^2\mathrm{d}x\mathrm{d}t
			+C\left(E_0+E^\frac{1}{2}\int_{0}^{t}\mathcal{D}(t)\mathrm{d}t+\delta_R\int_{0}^{t}\mathcal{D}(t)\mathrm{d}t+\delta_R^{\frac{1}{2}}+\delta_R^{\frac{3}{2}-\frac{1}{q}}\right).
		\end{align}
		The remaining terms in \eqref{4.107} can be estimated in a similar way. We get
		\begin{align*}
			\int_{0}^{t}\int\tau_{2}S^R_{tx}\psi_{xx}\mathrm{d}x\mathrm{d}t&\leq\epsilon\int_{0}^{t}\int\tau_{2}\psi_{xx}^2\mathrm{d}x\mathrm{d}t+C\int_{0}^{t}\int (S^R_{tx})^2\mathrm{d}x\mathrm{d}t\\
			&\leq\epsilon\int_{0}^{t}\int\tau_{2}\psi_{xx}^2\mathrm{d}x\mathrm{d}t+C(\delta_R^{3-\frac{2}{q}}+\delta_R),
		\end{align*}
		and
		\begin{align*}
			\int_{0}^{t}\int(v\tilde{S}+S^R\varphi)_{x}\psi_{xx}\mathrm{d}x\mathrm{d}t\leq\int_{0}^{t}\int\epsilon\psi^2_{xx}\mathrm{d}x\mathrm{d}t+C\left(E^\frac{1}{2}(t)\int_{0}^{t}\mathcal{D}(t)\mathrm{d}t
			+\delta_{R}\int_{0}^{t}\mathcal{D}(t)\mathrm{d}t+\delta_R^{3-\frac{2}{q}}+\delta_R\right).
		\end{align*}
		Thus, by combining the above estimates and using \eqref{96}, we finally obtain
		\begin{align}\label{113}
			\int_{0}^{t}\int\psi_{xx}^2\mathrm{d}x\mathrm{d}t&\leq C\left(E_0+E^\frac{1}{2}(t)\int_{0}^{t}\mathcal{D}(t)\mathrm{d}t+\delta_R\int_{0}^{t}\mathcal{D}(t)\mathrm{d}t+\delta_R^{\frac{1}{2}}+\delta_R^{\frac{3}{2}-\frac{1}{q}}\right)
			+\epsilon\int_{0}^{t}\int\varphi_{xx}^2\mathrm{d}x\mathrm{d}t
		\end{align}
		In summary, combining \eqref{96}, \eqref{106} and \eqref{113}, we derive
		\begin{align*}
			\int_{0}^{t}\int\left(\varphi_{xx}^2+\tilde{\theta}_{xx}^2+\psi_{xx}^2\right)\mathrm{d}x\mathrm{d}t&\leq C\left(E_0+E^\frac{1}{2}(t)\int_{0}^{t}\mathcal{D}(t)\mathrm{d}t+C\delta_R\int_{0}^{t}\mathcal{D}(t)\mathrm{d}t+\delta_R^{\frac{1}{2}}+\delta_R^{\frac{3}{2}-\frac{1}{q}}\right).
		\end{align*}
		The proof of this Lemma is finished.
	\end{proof}

	Combining the results of Lemmas \ref{lem4.1}–\ref{lem4.5}, we complete the proof of Proposition \ref{prop3.2}.

\end{document}